\documentclass{article}
\usepackage[T1]{fontenc}
\usepackage[utf8]{inputenc}
\usepackage[a4paper, margin=1in]{geometry}

\usepackage{microtype}
\usepackage{graphicx}
\usepackage{subcaption}
\usepackage{booktabs} 

\usepackage[colorlinks,urlcolor=blue,citecolor=blue,linkcolor=blue]{hyperref}

\usepackage{amsmath}
\usepackage{amssymb}
\usepackage{mathtools}
\usepackage{amsthm}

\usepackage[capitalize,noabbrev]{cleveref}

\theoremstyle{plain}
\newtheorem{theorem}{Theorem}[section]
\newtheorem{proposition}[theorem]{Proposition}

\theoremstyle{definition}

\theoremstyle{remark}

\usepackage{float}

\usepackage[authoryear]{natbib}

\usepackage{listings}

\crefname{lstlisting}{code}{codes}
\Crefname{lstlisting}{Code}{Codes}

\lstdefinestyle{pyalgo}{
  language=Python,
  basicstyle=\ttfamily\small,
  keywordstyle=\color{blue},
  commentstyle=\color{gray},
  stringstyle=\color{teal},
  frame=single,
  columns=fullflexible,
  keepspaces=true,
  showstringspaces=false,
}

\usepackage{xcolor}
\usepackage{enumitem}

\newcommand{\R}{\mathbb{R}}
\newcommand{\E}{\mathbb{E}}

\newcommand{\Z}{\mathbb{Z}}
\newcommand{\PP}{\mathbb{P}}
\newcommand{\cF}{\mathcal{F}}

\newcommand{\tr}{\operatorname{trace}}

\newcommand{\loss}{L}
\newcommand{\stepsize}{\alpha}

\newcommand{\proj}{\mathcal{P}}
\newcommand{\sign}{\operatorname{sign}}
\newcommand{\diag}{\operatorname{diag}}

\newcommand{\Lips}{{c}} 
\newcommand{\Sv}{{S}} 
\newcommand{\Svnn}{{\Sigma}} 
\newcommand{\sv}{{s}} 
\newcommand{\svnn}{{\sigma}} 

\newcommand{\GD}{\textsc{GD}}
\newcommand{\Muon}{\textsc{Muon}}

\newcommand{\frob}[1]{\left\lVert #1 \right\rVert_{\mathrm{F}}}

\newcommand{\inner}[2]{\left\langle #1,#2 \right\rangle}

\newcommand{\eps}{\varepsilon}

\title{
Insights on Muon from Simple Quadratics
}

\author{
Antoine Gonon
\and
Andreea-Alexandra Muşat
\and
Nicolas Boumal
\and
\small Institute of Mathematics, EPFL, Switzerland
}

\date{} 

\begin{document}
\maketitle
\begin{abstract}
\Muon{} updates weight matrices along (approximate) polar factors of the gradients and has shown strong empirical performance in large-scale training.
Existing attempts at explaining its performance largely focus on single-step comparisons (on quadratic proxies) and worst-case guarantees that treat the inexactness of the polar-factor as a nuisance ``to be argued away''. 
We show that already on simple strongly convex functions such as $\loss(W)=\tfrac12\|W\|_{\mathrm{F}}^2$, these perspectives are insufficient, suggesting that understanding \Muon{} requires going beyond local proxies and pessimistic worst-case bounds. 
Instead, our analysis exposes two observations that already affect behavior on simple quadratics and are not well captured by prevailing abstractions:
(i) approximation error in the polar step can qualitatively alter discrete-time dynamics and \emph{improve} reachability and finite-time performance---an effect practitioners exploit to tune \Muon{}, but that existing theory largely treats as a pure accuracy compromise; and
(ii) structural properties of the objective affect finite-budget constants beyond the prevailing conditioning-based explanations. 
Thus, any general theory covering these cases must either incorporate these ingredients explicitly or explain why they are irrelevant in the regimes of interest.
\end{abstract}

\paragraph{Code.} \href{https://github.com/agonon/muon-on-quadratics}{github.com/agonon/muon-on-quadratics}.

\section{Introduction}

\Muon{} (MomentUm Orthogonalized by Newton--Schulz)
has gained attention 
as an optimizer for large language models 
by improving training speed 
on GPT-style models 
\cite{jordan2024muon,liu2025muonscalablellmtraining,shah2025practicalefficiencymuon}. 
At a high level,
\Muon{} can be viewed as SGD with Nesterov-momentum, 
except that instead of taking 
a step directly in the direction of the momentum matrix,
it first approximately orthogonalizes it 
using a small number of
Newton--Schulz (polynomial) iterations. 
These iterations approximate the polar factor 
(equivalently, a matrix-sign map),
resulting in an update direction that 
lies approximately on the Stiefel manifold
before a standard step is taken
\cite{jordan2024muon}. Code is provided in \Cref{app:muon-pseudocode}.

This design has motivated a growing theoretical literature 
that views \Muon{}-style updates through the lens of 
spectral-norm steepest descent and linear minimization oracles, 
and studies convergence under smoothness assumptions 
with various idealizations such as exact orthogonalization, 
non-Nesterov momentum,
simplified step-size rules, and locally quadratic models
\cite{pethick2025normconstrainedlmo,
li2025noteconvergencemuon,shen2025convergencemuon,
chen2025muonspectralnorm,kovalev2025understandinggo,
riabinin2025gluon,gruntkowska2025ef21muon,sato2025muonconvergence,nagashima2026improvedconvergenceratesmuon}.
More recent work begins to incorporate \emph{inexact} 
polar steps and quantify how approximation error propagates 
in worst-case bounds
\cite{shulgin2025beyondideal,anon2025newtonschulz,lau2025polargrad}.

Despite this progress, it remains unclear 
what parts of \Muon{}'s behavior are explained by these idealized viewpoints. 
Two themes in particular recur across the literature and practitioner discussions:
(i) analyses that emphasize quadratic or locally quadratic 
behavior and per-step improvement along a spectral/Stiefel direction,
and (ii) analyses that treat projection inexactness 
primarily as a computational cost--accuracy tradeoff.
These themes naturally 
lead to several basic questions that 
are not settled by existing results.

\begin{enumerate}[
  label=\textbf{Q\arabic*:},
  leftmargin=0pt,
  labelindent=0pt,
  labelwidth=!,
  align=left,
  itemsep=0.05em
]
  \item \textbf{
Does \Muon{} actually find the global minimizer of a smooth, 
strongly convex function? 
}
Many guarantees for \Muon{}-like methods 
target stationarity measures over general smooth objectives 
\cite{pethick2025normconstrainedlmo,
li2025noteconvergencemuon,shen2025convergencemuon,
chen2025muonspectralnorm,kovalev2025understandinggo,
riabinin2025gluon,gruntkowska2025ef21muon,sato2025muonconvergence,
shulgin2025beyondideal,anon2025newtonschulz,lau2025polargrad,nagashima2026improvedconvergenceratesmuon}. 
These analyses rely 
on tools from optimization theory 
(e.g., descent lemmas)
that 
are most effective under strong convexity assumptions.
But what if we \emph{do} have strong convexity? 
Can we already guarantee basic properties such as convergence 
to the global minimizer? 
And more quantitatively, 
what advantage does \Muon{} offer over standard methods such as \GD{} 
in terms of convergence speed to a given loss level? 

\item \textbf{
Ignoring computational cost, is an exact polar 
factorization always preferable to an approximate one, 
or can approximations \emph{improve} optimization dynamics?
}
A common modeling choice treats inexactness as an error term 
that monotonically worsens 
a baseline method 
\cite{shulgin2025beyondideal,anon2025newtonschulz,lau2025polargrad},
and much numerical work aims 
to reduce approximation error for a fixed compute budget 
\cite{amsel2025polarexpress,cesista2025muonoptcoeffs,grishina2025chebyshevns,boumal2025polarpoly}. 
Yet \Muon{}'s practical implementation 
\emph{deliberately} uses low-degree 
approximations whose errors are not negligible.
Is approximation merely an implementation compromise, 
or can projection inexactness qualitatively change the discrete-time 
dynamics in a way that can \emph{help} reach small-loss regimes? 

\item 
\textbf{Is it appropriate to assess 
\Muon{} by comparing it to competing algorithms 
one iteration at a time?}
Several analyses compare the per-step decrease obtained 
by the polar/Stiefel direction to that of the Euclidean 
gradient direction under line-search optimal step sizes or quadratic proxies 
\cite{davis2025whenspectral,su2025isotropiccurvature}. 
Does local, ``greedy'' superiority translate into a global convergence?
\end{enumerate}

\textbf{Contributions.} We provide a few answers:
\begin{enumerate}[
  label=\textbf{A\arabic*:},
  leftmargin=0pt,
  labelindent=0pt,
  labelwidth=!,
  align=left,
  itemsep=0.05em
]
  \item \textbf{
For \Muon{} to minimize even a strongly convex loss, its step sizes must vanish---this is not the default (\Cref{sec:setup}).}  
We exhibit strongly convex losses 
where, under 
constant step sizes and exact projection, 
the dynamics become confined to discrete lattices 
(``grid confinement''): 
for almost all initializations, 
the iterates never enter sufficiently small loss balls. 
Moreover, even in favorable (reachable) cases, the best possible iteration count to reach loss $\eps$ is $O(1/\sqrt{\eps})$,
in contrast with \GD{}'s $O(\log(1/\eps))$ on the same objective. Thus, \Muon{} has no advantage in 
the asymptotic
$\eps$-dependence. 
Beyond asymptotic rates, we analyze whether finite-time speedups can arise through more favorable big-O constants in ill-conditioned settings, a prevailing hypothesis behind \Muon{}'s practical success \cite{jordan2024muon}. 
While such effects appear in worst-case constructions, we show that in typical quadratic instances, the conditioning alone does not explain finite-budget performance. Rather, finer spectral properties such as spectrum shape also matter. 

\item \textbf{Inexact projection is not just a nuisance (\Cref{sec:noisy}).} 
On a simple quadratic, introducing stochastic perturbations to the polar step restores eventual reachability of small-loss neighborhoods and can reduce the iteration count required to reach a given loss level. That count is not a monotone function of the perturbation magnitude: an intermediate noise level yields the fastest convergence on this toy model.
This demonstrates that, even in a simplified error model, inexact projection can play a constructive algorithmic role rather than acting solely as a computational compromise. For the Newton--Schulz approximation used in practice, experiments on quadratics show that approximation error can either accelerate or slow convergence depending on the spectral structure of the loss, highlighting that approximation error plays a role beyond a simple cost-accuracy tradeoff. 

\item 
\textbf{One-step superiority can be globally misleading (\Cref{sec:line-search}).}
We entertain a greedy policy which, at each step, picks the best update between a gradient step or a Stiefel/polar step, both with their optimal step size.
We exhibit general quadratics such that it is always the Stiefel/polar step that is picked, yet the \GD{} trajectory overall is much faster.
This shows that per-iteration comparisons are not a reliable proxy for end-to-end convergence speed.
\end{enumerate}

\textbf{Take-home message.}
Even on simple strongly convex quadratics, \Muon{} exhibits
non-trivial discrete-time behavior that already acts as a filter as to 
which explanatory narratives can be universally meaningful. 
In particular, our analysis challenges two common viewpoints, stating instead:
(i) local one-step superiority arguments are not reliable predictors of
end-to-end speed, and
(ii) worst-case guarantees that degrade monotonically with polar approximation
error 
do not qualitatively capture the role of inexact projection. 

At the same time, this work brings forward two ingredients that are 
quite overlooked in existing theory but 
that must be confronted by any general explanation of \Muon{}'s behavior. 
First, approximation error in the polar step can play a constructive
role by accelerating convergence. 
Second, while \Muon{} does not improve worst-case 
asymptotic $\eps$-dependence on 
convex objectives, our quadratic experiments show that 
any advantage must come from problem-dependent big-O constants rather than asymptotic rates. 
A prevailing hypothesis behind \Muon{}'s practical success is a form of conditioning insensitivity: 
since \GD{}'s worst-case guarantees deteriorate with the condition number $\kappa$, 
one might expect \Muon{} to win systematically on poorly conditioned problems. 
However, this worst-case dependence of \GD{} on $\kappa$ is not representative of typical quadratic behavior, 
where performance depends on finer spectral properties beyond $\kappa$ alone. 
In our controlled setting, conditioning alone is therefore not predictive: 
at fixed $\kappa$, changing the spectrum shape can flip whether \Muon{} beats \GD{}. 
Moreover, \Muon{}'s loss decrease is comparatively stable across the spectrum families we test, 
so the variation in who wins is largely driven by \GD{}'s sensitivity to finer spectral structure.

\section{Related Work}
\label{sec:related-work} 

A few \emph{recurring methodological lenses} have shaped 
the current understanding of \Muon{}-like methods. 
Existing theoretical work 
either 
(i) analyzes \Muon{}-like updates through smoothness-based descent arguments and norm-constrained subproblems, or
(ii) motivates \Muon{}-like directions via per-iteration improvement on quadratic proxies.
In parallel, numerical work on polar approximations largely treats inexactness as a cost--accuracy knob.
We review these threads through the three aspects we focus on: 
strong convexity, one-step comparisons, and the role of inexact projection.

\textbf{\Muon{} as approximate orthogonalization of momentum.} 
Muon 
(code in \Cref{app:muon-pseudocode})
combines (i) maintaining a Nesterov-style momentum matrix and (ii) applying a small number of GPU-friendly Newton--Schulz / polynomial steps that approximate its polar factor (matrix sign), yielding an update direction that lies approximately on the Stiefel manifold, before a standard step is taken \cite{jordan2024muon}. 
We use the implementation of the nanoGPT speedrun benchmark as a reference point since it keeps track of the latest practical choices that matter for state-of-the-art performance \cite{nanogpt24speedrun}. 
In particular, we will see that a theoretical challenge posed by this implementation is that it uses non-vanishing step sizes (see \Cref{sec:setup}), even though this implies  inherent obstructions to convergence already on simple strongly convex functions. 
Subsequent empirical studies explore how \Muon{} compares to AdamW and SGD-momentum across architectures and scales, and how implementation choices such as normalization, weight decay, learning-rate schedules, and the number/degree of polar steps affect outcomes \cite{liu2025muonscalablellmtraining,shah2025practicalefficiencymuon}.
These empirical works make it clear that \Muon{} 
is defined not only by a direction (polar/Stiefel) 
but by a well-tuned combination of 
step-size, momentum, and projection inexactness, 
which motivates asking which components are essential for qualitative behavior.

\textbf{Smoothness-based analyses via norm-constrained subproblems and descent lemmas.}
A large body of theoretical work studies \Muon{} updates by interpreting them as solutions of spectral-norm constrained subproblems 
that would be solved at each gradient step
(LMOs, steepest descent in non-Euclidean norms, or related trust-region steps). The same works then apply standard smoothness tools (descent lemmas based on Lipschitz gradients) to derive worst-case guarantees over general smooth objectives 
\cite{pethick2025normconstrainedlmo,
li2025noteconvergencemuon,shen2025convergencemuon,
chen2025muonspectralnorm,kovalev2025understandinggo,
riabinin2025gluon,gruntkowska2025ef21muon,sato2025muonconvergence,nagashima2026improvedconvergenceratesmuon}. 
A representative object is the linear minimization oracle for a norm $\|\cdot\|$,
\[
\textrm{LMO}(D) \in \textrm{argmin}_{\|V\|\le 1}\ \inner{D}{V},
\]
for which the spectral-norm LMO is given by the negative polar factor $-UV^\top$ when $D=U\Sv V^\top$ is an SVD \cite{pethick2025normconstrainedlmo}.
This is precisely the matrix-sign / polar-factor direction that \Muon{} approximates.

These analyses typically yield \emph{stationarity-type} conclusions under standard assumptions (Lipschitz gradients, access to unbiased stochastic gradients with bounded variance), and often consider idealized update rules (exact projection, no momentum or non-Nesterov momentum, and simplified step-size rules such as vanishing ones)
\cite{pethick2025normconstrainedlmo,
li2025noteconvergencemuon,shen2025convergencemuon,
chen2025muonspectralnorm,kovalev2025understandinggo,
riabinin2025gluon,gruntkowska2025ef21muon,sato2025muonconvergence,nagashima2026improvedconvergenceratesmuon}.
Because the underlying toolkit is built around smoothness and descent, it is natural to ask what these guarantees imply in the most favorable setting 
these tools have been designed for---\emph{strong convexity}. 
This leads directly to \textbf{Q1}: if we \emph{do} have a smooth, strongly convex objective, does \Muon{} necessarily converge to the global minimizer in discrete time under standard (non-vanishing, as in the nanoGPT speedrun contest \cite{nanogpt24speedrun}) step sizes?
What is the best possible asymptotic dependence of the time-to-$\eps$ loss on $\eps$? 

\textbf{Quadratic proxies and one-step improvement arguments.}
A separate thread uses quadratic or locally quadratic models to justify when spectral/Stiefel directions should offer better \emph{one-step} decrease than Euclidean gradients, often emphasizing conditioning and alignment effects
\cite{davis2025whenspectral,su2025isotropiccurvature}.
This perspective is compelling because it connects \Muon{}-style directions to an interpretable per-iteration metric.
At the same time, it naturally raises \textbf{Q3}: is it appropriate to assess \Muon{} by comparing it to competing algorithms one iteration at a time?
In particular, does greedy one-step superiority reliably translate into faster end-to-end convergence along the full trajectory? 

\textbf{Inexact projection.}
Muon's projection is an approximate polar decomposition / matrix-sign computation implemented with a small number of Newton--Schulz-like or polynomial iterations.
A rapidly growing literature aims at improving these approximations for a fixed compute budget, including polynomial design, stability/precision engineering 
\cite{amsel2025polarexpress,grishina2025chebyshevns,cesista2025muonoptcoeffs,lau2025polargrad,boumal2025polarpoly,kang2025psdpoly}.
These approaches mainly treat approximation quality as a cost--accuracy knob. 

On the theory side, recent works begin to model \emph{inexact} LMOs / inexact polar steps and track how an error parameter propagates through worst-case bounds
\cite{shulgin2025beyondideal,anon2025newtonschulz,lau2025polargrad}.
This is an important step toward realism, but the prevailing abstraction still treats inexactness as a nuisance term that monotonically worsens guarantees relative to the exact-projection baseline.

Muon's practice, however, deliberately operates in a regime where the polar step is low-degree and the error is not negligible \cite{jordan2024muon,liu2025muonscalablellmtraining,shah2025practicalefficiencymuon}.
Moreover, several works explicitly tune approximation mechanisms to improve optimization speed rather than to minimize approximation error per se
\cite{ahn2025dion,khaled2025muonbp,he2025lowrankmuon}.
This naturally motivates \textbf{Q2}: if we set aside computational cost, is an exact polar factor always preferable, or can inexact projection qualitatively change discrete-time dynamics in a way that can \emph{help} optimization?

\section{Exact Projection with Fixed Step Sizes: Grid Confinement, Rates, and What They (Don't) Explain}
\label{sec:setup}

This section isolates a single design choice that is often idealized in theory:
\emph{exact} Stiefel/polar projection (equivalently, an exact spectral-norm LMO / exact matrix sign),
combined with fixed step sizes.
Our goal is not to advocate this idealization, but to understand precisely what it implies---and what it fails to imply---already on smooth strongly convex objectives. 
We ask:

\begin{quote}
\textbf{Q1}: \emph{With exact projection and fixed step sizes, does a \Muon{}-style update necessarily reach the global minimizer of a smooth, strongly convex objective? If it can reach a loss level $\eps$, what is the best possible dependence on $\eps$? Besides 
the dependence on $\eps$, what could explain speedups in practice---perhaps through conditioning-dependent constants?}
\end{quote}

For reference, the (usually unique) polar factor of a matrix $M \in \mathbb{R}^{d_1 \times d_2}$ with SVD
\begin{align}
  M = U \Svnn V^\top \qquad \textrm{ is } \qquad \proj(M) = UV^\top,
  \label{eq:polardef}
\end{align}
where $d = \min(d_1, d_2)$, matrices $U \in \mathbb{R}^{d_1 \times d}, V \in \mathbb{R}^{d_2 \times d}$ have orthonormal columns, and $\Svnn = \mathrm{diag}(\svnn_1, \ldots, \svnn_d)$ holds the singular values $\svnn_1 \geq \cdots \geq \svnn_d \geq 0$.

\subsection{Global convergence to minimizer?}
\label{sec:setup-minimizer}
We start with convergence to the global minimizer.  
We exhibit a simple counterexample where, for generic initializations, the iterates become confined to discrete lattices and miss sufficiently small loss balls, even on simple strongly convex quadratics. 
We then explain how this is reflected in 
existing stationarity-type bounds through an irreducible residual term. 
The only remedy 
to have a generic convergence result 
is to let the effective step size vanish. 

Consider the isotropic quadratic loss
\begin{equation}
  \loss(W) = \frac12 \frob{W}^2,
  \quad W\in\R^{d\times d},
  \label{eq:isotropic-quadratic-setup}
\end{equation}
which is smooth and strongly convex, with $\nabla \loss(W)=W$.
The exact Stiefel/polar update (no momentum) 
with constant step size $\stepsize>0$ 
is
\begin{equation}
  W_{t+1} = W_t - \stepsize\, \proj(W_t).
  \label{eq:exact-stiefel-update-setup}
\end{equation}
If $W_0 = U\Sv_0V^\top$ with $U, V$ orthogonal and $\Sv_0$ diagonal with entries $\sv_{0,1},\ldots,\sv_{0,d}$ (not necessarily nonnegative), then $\proj(W_0) = U \sign(\Sv_0) V^\top$, where $\sign(\Sv_0)$ is diagonal with entries $\sign(\sv_{0,i})$.
Thus, $W_1 = W_0 - \stepsize\,\proj(W_0) = U(\Sv_0 - \stepsize\,\sign(\Sv_0))V^\top$: the matrices $U, V$ remain unchanged.
By induction, $W_t = U \Sv_t V^\top$ for all $t$, with $\Sv_t = \diag(\sv_{t,1},\ldots,\sv_{t,d})$ following simple dynamics ($d$ separate scalar recursions):
\begin{equation}
  \sv_{t+1,i} = \sv_{t,i} - \stepsize\,\sign(\sv_{t,i}).
  \label{eq:singular-sign}
\end{equation}
Since $L(W_t) = \frac12 \sum_{i=1}^d \sv_{t,i}^2$,
the overall dynamics reduce to 
$d$ independent copies of 1-D sign-GD: 
\begin{equation}
  \ell(s)=\frac12 s^2,
  \qquad
  \sv_{t+1} = \sv_t - \stepsize\,\sign(\sv_t).
  \label{eq:1d-sign-setup}
\end{equation}
This trajectory remains on the lattice $\sv_0+\stepsize\Z$.
Starting from $\sv_0>0$, it decreases by $\stepsize$ until it crosses $0$, and then it falls into a $2$-cycle between the two lattice points closest to $0$.
Consequently, unless $\sv_0\in \stepsize\Z$ (an unlikely condition under generic initialization), 
the iterates never reach arbitrarily small neighborhoods of the minimizer. 
On the matrix quadratic \eqref{eq:isotropic-quadratic-setup}, the same grid confinement occurs simultaneously across all singular values via \eqref{eq:singular-sign}, so the obstruction persists, even though the loss is smooth and strongly convex. 

\textbf{Momentum does not break grid confinement.} 
There are multiple natural ways to combine momentum with Stiefel/polar steps; we summarize the main variants and their relation to \Muon{} in \Cref{app:momentum}.
\Muon{}'s official implementation  
computes a momentum term $m_{t,i}$ \emph{before} projection,
and takes a step along the projected momentum:
\begin{equation}
  \sv_{t+1,i} = \sv_{t,i} - \stepsize\,\sign(m_{t,i}),
  \label{eq:momentum-sign-grid}
\end{equation}
so $\sv_{t,i}\in \sv_{0,i}+\stepsize\Z$ for all $t$.
Thus, the same lattice obstruction persists. 
This answers the first part of \textbf{Q1} negatively for exact projection with fixed step sizes, with or without \Muon{}-style momentum.

\textbf{Implications for worst-case smooth theory: 
where grid confinement appears in stationarity results.} 
Despite the negative reachability result above, 
there are many existing stationarity-type guarantees for \Muon{} 
that establish convergence to stationary points. 
We now explain how grid confinement manifests 
in these results, and how these results implicitly acknowledge
the need for vanishing step sizes to 
converge to stationarity. 

For a constant step size $\stepsize$, 
simple algebraic manipulations of standard descent-lemma arguments yield
\begin{equation}
  \left[ \min_{0\le t\le T-1}\|\nabla \loss(W_t)\|_\ast \right]
  \ \le\
  \frac{\loss(W_0)-\loss_\ast}{T\stepsize}
  \;+\;
  \frac{\Lips}{2}\,d\,\stepsize.
  \label{eq:stationarity-constant}
\end{equation}
It is valid for any $\Lips$-smooth loss $\loss \colon \R^{d_1\times d_2}\to\R$, 
assuming $T$ steps from initialization $W_0$, 
with exact projection and fixed step size $\stepsize>0$ (see \Cref{app:stationarity}). 
Here, $\|\cdot\|_\ast$ is the nuclear norm, $d=\min(d_1,d_2)$, and $\loss_\ast=\inf_W \loss(W)$. 

Variants of this type of bound appear as an intermediary step 
in many (if not all) existing stationarity analyses for \Muon{} 
\cite{pethick2025normconstrainedlmo,
li2025noteconvergencemuon,shen2025convergencemuon,
chen2025muonspectralnorm,kovalev2025understandinggo,
riabinin2025gluon,gruntkowska2025ef21muon,sato2025muonconvergence,
shulgin2025beyondideal,anon2025newtonschulz,lau2025polargrad,nagashima2026improvedconvergenceratesmuon}. 
As $T\to\infty$, the first term vanishes but the second remains: the best stationarity level one can guarantee is $O(\stepsize)$ unless $\stepsize\to 0$.
This is the analytic ``shadow'' of grid confinement: fixed step sizes induce a non-vanishing residual term.  
On losses also satisfying a P\L{} inequality, 
this translates into a residual suboptimality gap $L(W_t)-L_\ast$ of order $O(\stepsize^2)$, consistent with what we 
saw 
on $\loss(W)=\frac12\|W\|_{\mathrm{F}}^2$ where the floor loss is $\Theta(\stepsize^2)$. 
Because of that, 
existing analyses typically assume vanishing step sizes 
in their final statements to eliminate this irreducible term and to obtain a convergence result.

Beyond worst-case smooth theory, 
\citet{ma2026preconditioningbenefitsspectralorthogonalization} 
study specific quadratic and quartic losses. 
By restricting to specific objectives, 
they obtain tighter convergence rates. 
Their analysis relies on (approximate) decoupling along singular directions, as in \eqref{eq:singular-sign}. 
That setting must confront the same discrete-time issue highlighted above, and indeed the analysis adopts vanishing step-size schedules to obtain convergence to the minimizer.
Our point here is not to compare assumptions, but to emphasize the structural fact:
\emph{with exact projection, non-vanishing step sizes 
are fundamentally incompatible with generic minimizer reachability}.

\subsection{Best possible $\eps$-dependence under fixed steps: no better than $O(1/\sqrt{\eps})$}
\label{sec:setup-eps}

Suppose we ignore the reachability obstruction (e.g., by assuming a lucky initialization that aligns with the lattice, or by asking only for the fastest possible approach to an $\eps$-scale neighborhood before cycling).
We now ask the second part of \textbf{Q1}:
\emph{does \Muon{} have an asymptotically better dependence on $\eps$ than \GD{} on smooth strongly convex objectives, under exact projection and fixed step sizes?}

Answer: \textbf{No}, it is at best $O(1/\sqrt{\eps})$ on a generic strongly convex quadratic, which is worse than \GD{}'s $O(\log(1/\eps))$. 

To see that it cannot be better than $O(1/\sqrt{\eps})$, consider again the isotropic quadratic \eqref{eq:isotropic-quadratic-setup} and its decoupled singular-value dynamics \eqref{eq:singular-sign}. 
Each step on $\ell(\sv)=\frac12 \sv^2$ 
decreases the magnitude of each singular value 
$\sv_t$ 
by exactly $\stepsize$ until the first sign flip 
where it enters a $2$-cycle. 
Cycling occurs after $t=\Theta(|\sv_0|/\stepsize)$ steps. 
The 2-cycle oscillates between values of magnitude
at most $\stepsize$, 
hence the loss at that point is at most $\stepsize^2/2$. 
For the loss to be at most $\eps$, we need $\stepsize = O(\sqrt{\eps})$, resulting in a time-to-$\eps$ loss of $\Theta(1/\sqrt{\eps})$ steps. 
This occurs simultaneously across all singular values via \eqref{eq:singular-sign}, so the same scaling applies to the 
full loss $\loss(W)=\frac12\|W\|_{\mathrm{F}}^2$. 

\textbf{Contrast with \GD{}.}
On the same quadratic, \GD{}
converges in one step with step size $1$. 
More generally, $O(1/\eps)$ steps are enough on any smooth \emph{convex} objective, given that the step size is small enough (e.g., $1/L$ for $L$-smooth functions), and $O(\log(1/\eps))$ steps suffice under \emph{strong} convexity. 

\subsection{If not $\eps$-dependence, then what? Conditioning and spectrum shape at a fixed budget}
\label{sec:setup-conditioning}

Sections~\ref{sec:setup-minimizer} and~\ref{sec:setup-eps} rule out a general explanation of \Muon{}'s success based purely on asymptotic $\eps$-dependence under exact projection and fixed step sizes. 
A remaining hypothesis is \emph{finite-budget advantage}:
for a fixed iteration budget $T$, \Muon{}-style updates might enjoy better problem-dependent constants than \GD{} in the big-O rates.

A common practitioner narrative emphasizes \emph{ill-conditioning}~\cite{jordan2024muon}:
since \GD{}'s worst-case guarantees deteriorate with the condition number $\kappa$,
while Stiefel/polar updates can achieve rates independent of $\kappa$ 
on certain quadratics and quartics \cite{ma2026preconditioningbenefitsspectralorthogonalization}, 
one might expect \Muon{} to systematically outperform \GD{} in poorly conditioned cases.

However, several logical gaps remain.
First, $\kappa$-insensitivity of \Muon{} in specific worst-case analyses 
does not imply insensitivity to finer spectral structure at fixed $\kappa$.
Second, although the worst-case dependence of \GD{} on $\kappa$ is well understood, 
average-case analyses indicate that the performance of first-order methods on quadratics generally depends
on the \emph{full spectrum} of the Hessian rather than on $\kappa$ alone~\cite{pedregosa20avgcasequadratic,cunha22avgcaseuniversality}.
What is thus not clear a priori is how sensitive each method is to spectrum shape at fixed $\kappa$, 
and whether these variations are large enough to alter the \emph{relative ordering} between \Muon{} and \GD{}. 
We therefore turn to the following question:
\begin{quote}
\emph{Can we construct equally ill-conditioned problems for which the ranking between \Muon{} and \GD{} flips?}
\end{quote}
A positive answer would imply that conditioning alone is not predictive of the relative performance of \Muon{} and \GD{} on generic quadratics.

\textbf{Experimental setup.}
We consider quadratics of the form
\begin{equation}
  \loss(W)
  \;=\;
  \tfrac12\langle W, A W\rangle + \langle B, W\rangle + c.
  \label{eq:quad-xp}
\end{equation}
All matrices are square of size $n\times n$, with $n=100$. 
We set $A = Q S Q^\top$, with $Q$ a random orthogonal matrix obtained as the $Q$ factor
of a QR decomposition of a matrix with i.i.d.\ $\mathcal N(0,1)$ entries, 
and $S=\mathrm{diag}(s_1,\dots,s_n)$ set according to $7$ different spectrum shapes 
(uniform, Gaussian, spiked, etc.; see \Cref{fig:eig-distributions-main}).
All spectrum families share the same endpoints $(s_{\min},s_{\max})$ and hence the
same condition number $\kappa=s_{\max}/s_{\min}$. 
For each quadratic instance, we sample $100$ random initializations
$[W_0]_{ij}\overset{\text{i.i.d.}}{\sim}\mathcal N(0,1/n)$ and run each optimizer for $T=500$ iterations.
Learning rates are selected by choosing the run that achieves the smallest loss
up to time $T$. Full details are given in
\Cref{app:controlled-spectrum}.

\begin{figure*}[t]
  \centering
  \includegraphics[width=0.98\linewidth]{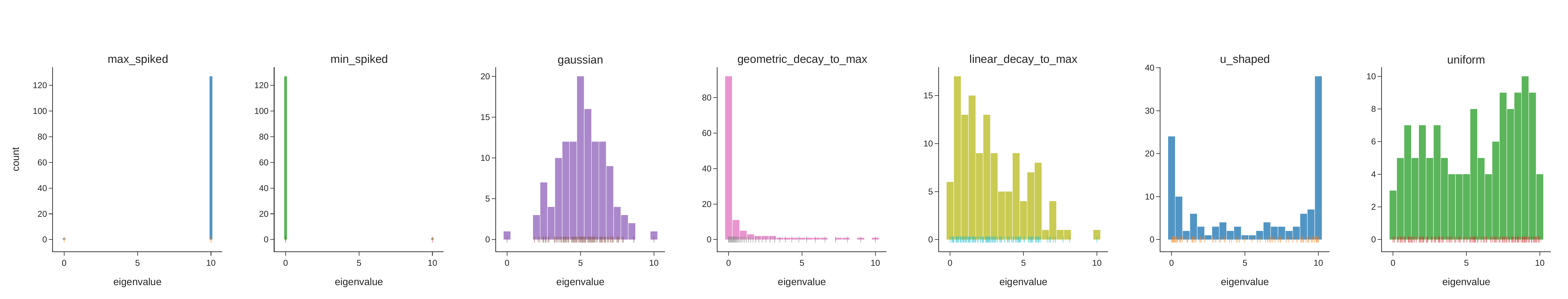}
  \caption{Eigenvalue distributions for the controlled-spectrum families. All families share the same endpoints $(s_{\min},s_{\max})=(10^{-3},10)$ and thus the same condition number $\kappa=s_{\max}/s_{\min} = 10^4$, but have different spectrum shapes.}
  \label{fig:eig-distributions-main}
\end{figure*}

\textbf{Main result: at fixed conditioning, 
\Muon{} can either outperform or underperform \GD{} depending on spectrum shape.}
We consider quadratic instances with a large Hessian condition number, $\kappa=10^4$.
Across all runs, the gradient condition numbers observed along \GD{} trajectories
lie in a comparable range, between $10^4$ and $10^8$
(see \Cref{app:grad-conditioning}).
\Cref{tab:exact-muon-vs-gd-win} reports, for each spectrum family, the fraction of
initializations for which \emph{exact-projection \Muon{}} (without momentum)
achieves a smaller loss than \GD{} at 
$t\in\{T/10,T/2,T\}$ over $100$ random initializations.
The results are binary: for each spectrum shape, \Muon{} either
consistently outperforms \GD{} or consistently underperforms it.
Since all instances share the same Hessian condition number and include
\GD{} runs with similar gradient conditioning (\Cref{app:grad-conditioning}), 
these results show that
\emph{conditioning alone does not determine which method is faster}
even in this quadratic setting. 
Despite identical $\kappa$, varying only the spectrum shape is sufficient to flip the ranking between \Muon{} and \GD{}.
Any general explanation of \Muon{}'s advantage must therefore account for finer spectral structure,
or explain why it becomes irrelevant in the regimes of interest.

\begin{table}[t]
\centering
\small
\setlength{\tabcolsep}{3pt}
\renewcommand{\arraystretch}{1.15}
\caption{Win rates for exact-projection \Muon{} vs.\ \GD{} (comparing the best loss achieved by each algorithm across all learning rates, up to time $t$); a ``1'' means \Muon{} always wins. All spectra share the same endpoints $(10^{-3}, 10^1)$ and condition number $\kappa=10^4$.}
\label{tab:exact-muon-vs-gd-win}
\begin{tabular}{lccc}
\toprule
kind & $t=T/10$ & $t=T/2$ & $t=T$ \\
\midrule
\texttt{max\_spiked} & 0 & 0 & 0 \\
\texttt{min\_spiked} & 1 & 1 & 1 \\
\texttt{uniform} & 0 & 0 & 0 \\
\texttt{gaussian} & 0 & 0 & 0 \\
\texttt{linear\_decay\_to\_max} & 0 & 0 & 0 \\
\texttt{u\_shaped} & 0 & 0 & 1 \\
\texttt{geometric\_decay\_to\_max} & 1 & 1 & 1 \\
\bottomrule
\end{tabular}
\end{table}

\textbf{Magnitude of loss reduction at fixed $\kappa$.}
To quantify this effect more directly, 
we measure how many orders of magnitude each method reduces the loss from the same initialization, averaged over runs. 
\Cref{fig:improvement-bars} shows that there is at least one order of magnitude difference in loss reduction between \Muon{} and \GD{} 
in all cases. 
Moreover, 
\Muon{} consistently reduces the loss by a similar number of orders of magnitude across all spectrum shapes.
In contrast, \GD{}'s reduction varies substantially with the spectrum shape.
Thus, at fixed $\kappa$, the flip in ranking is primarily 
driven by the fact that \GD{}'s progress varies substantially with the spectral distribution,
while \Muon{} remains comparatively stable across the spectra we test.

\textbf{Relation to spectra in neural networks.} 
We note that the spectrum shapes that favor \Muon{} here, those with a 
bulk of small eigenvalues and a few large ones (\texttt{min\_spiked} 
and 
\texttt{geometric\_decay\_to\_max}), resemble Hessian spectra reported in neural networks 
\cite{sagun2017singularityhessian,sagun2018empiricalhessianminspiked,ghorbani2019hessianeigvaldensity,li2019hessianbasedanalysissgd,gur-ari2019gradient,papyan2020hessiancharacterization,song2025doessgdreallyhappenintinysubspaces,tang2025powerlawhessianspectrum},
though we cannot draw definitive conclusions beyond this qualitative similarity. 

\begin{figure}[t]
  \centering
  \includegraphics[width=0.95\linewidth]{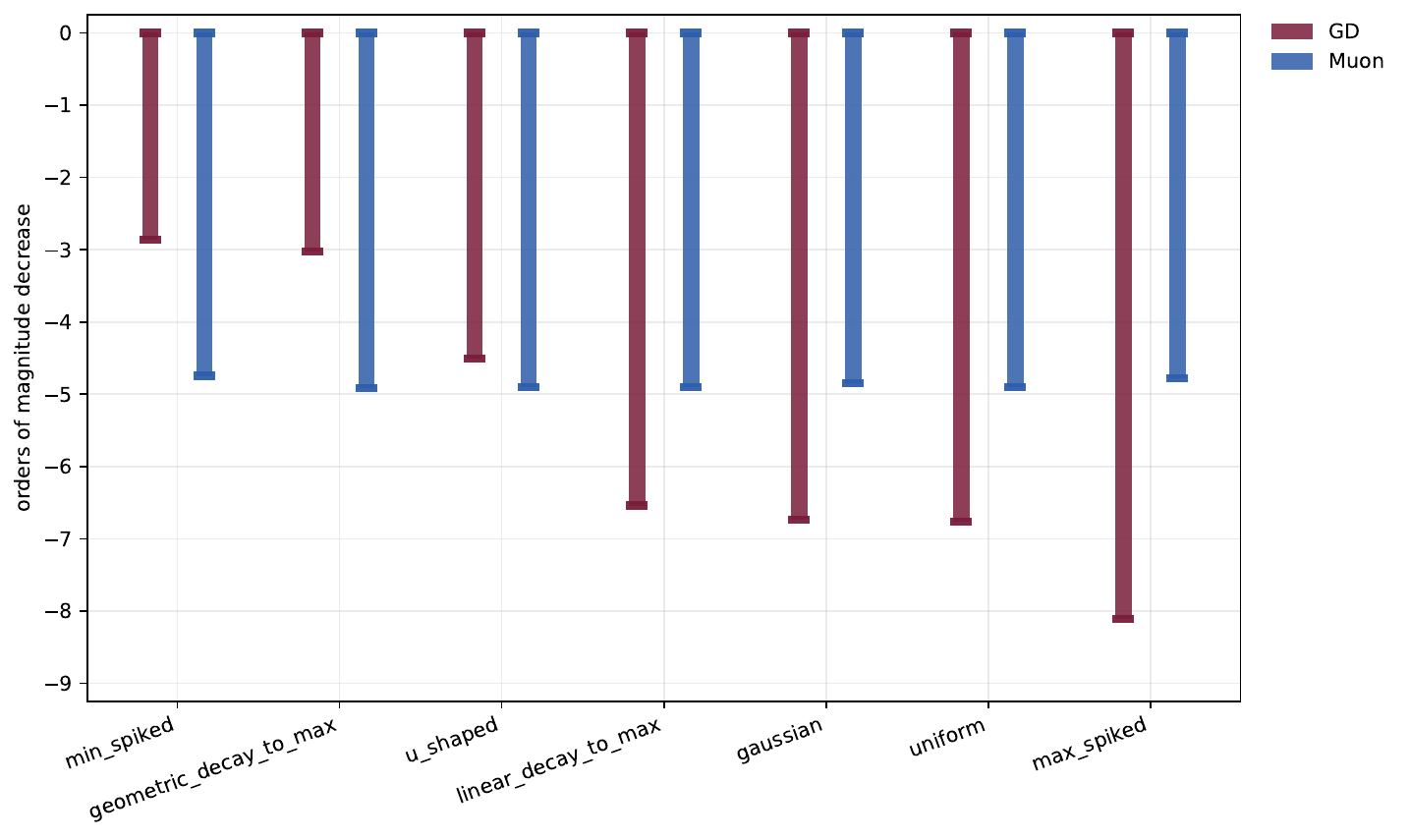}
\caption{
Orders of magnitude of loss decrease after $T=500$ iterations on the controlled-spectrum quadratic family (\Cref{sec:setup-conditioning}).
For each spectrum shape (all sharing the same endpoints and condition number $\kappa=10^4$),
bars are aligned at the common initial loss so that their lengths represent the logarithmic decrease achieved.
The vertical axis is indexed by integers $-1,-2,\ldots$, corresponding to factors of $10^{-1},10^{-2},\ldots$ reduction.
\Muon{} achieves a comparable reduction across spectrum families,
whereas \GD{} exhibits variation,
which shows that the flip in ranking is primarily driven by the fact that the performance of \GD{} depends on the spectral distribution, 
while \Muon{} remains comparatively stable across the spectra we test. 
For the \texttt{max\_spiked} family, \GD{}'s loss is numerically very close to $0$.
For plot readability, we clip the corresponding bar at roughly $-8$ orders of magnitude.
}

  \label{fig:improvement-bars}
\end{figure}

\begin{figure*}[t]
  \centering
  \begin{subfigure}[t]{0.44\linewidth}
    \centering
    \includegraphics[width=\linewidth]{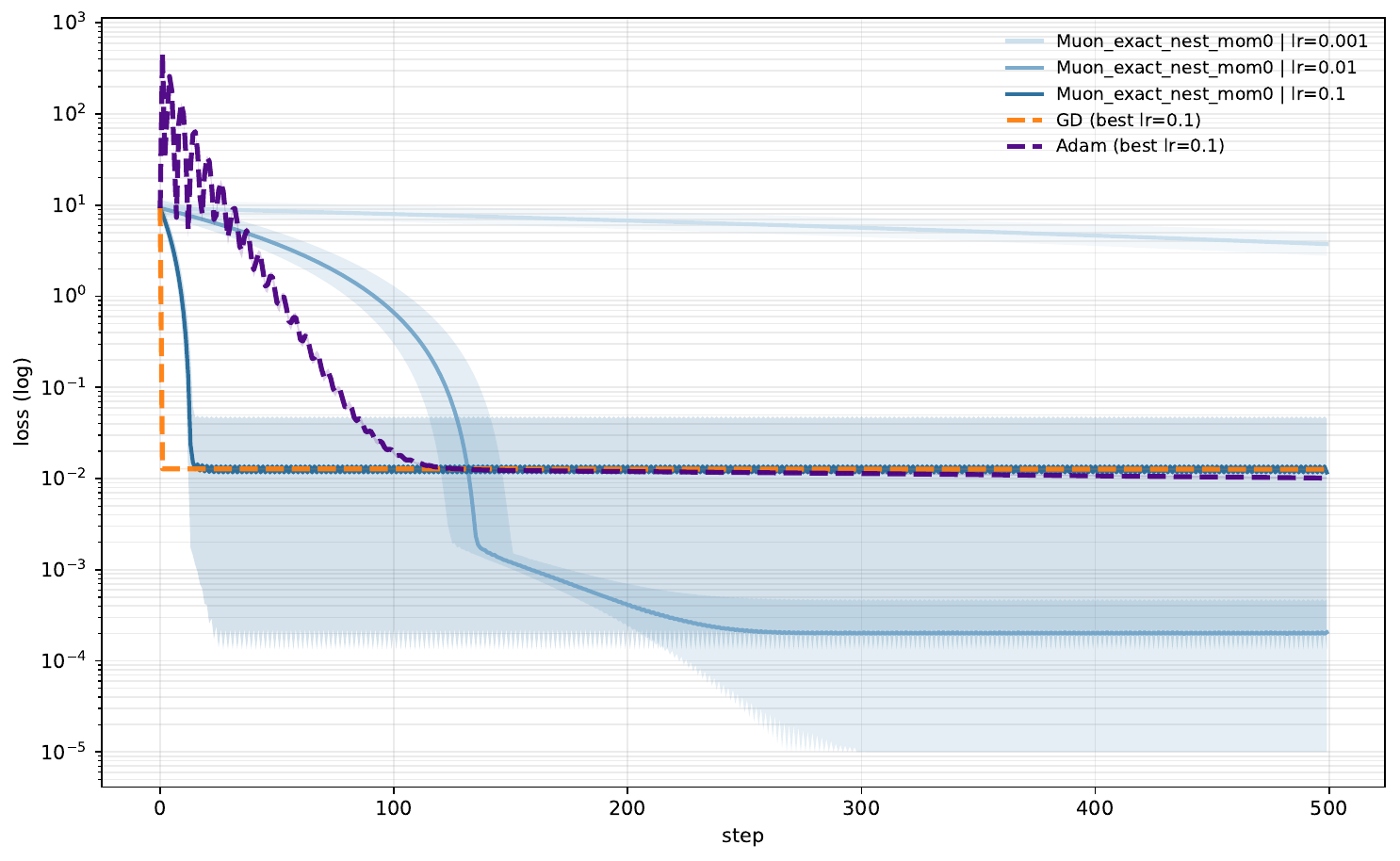}
    \caption{\texttt{min\_spiked} spectrum (favorable to \Muon{}).}
    \label{fig:avg-traj-minspike}
  \end{subfigure}\hfill
  \begin{subfigure}[t]{0.44\linewidth}
    \centering
    \includegraphics[width=\linewidth]{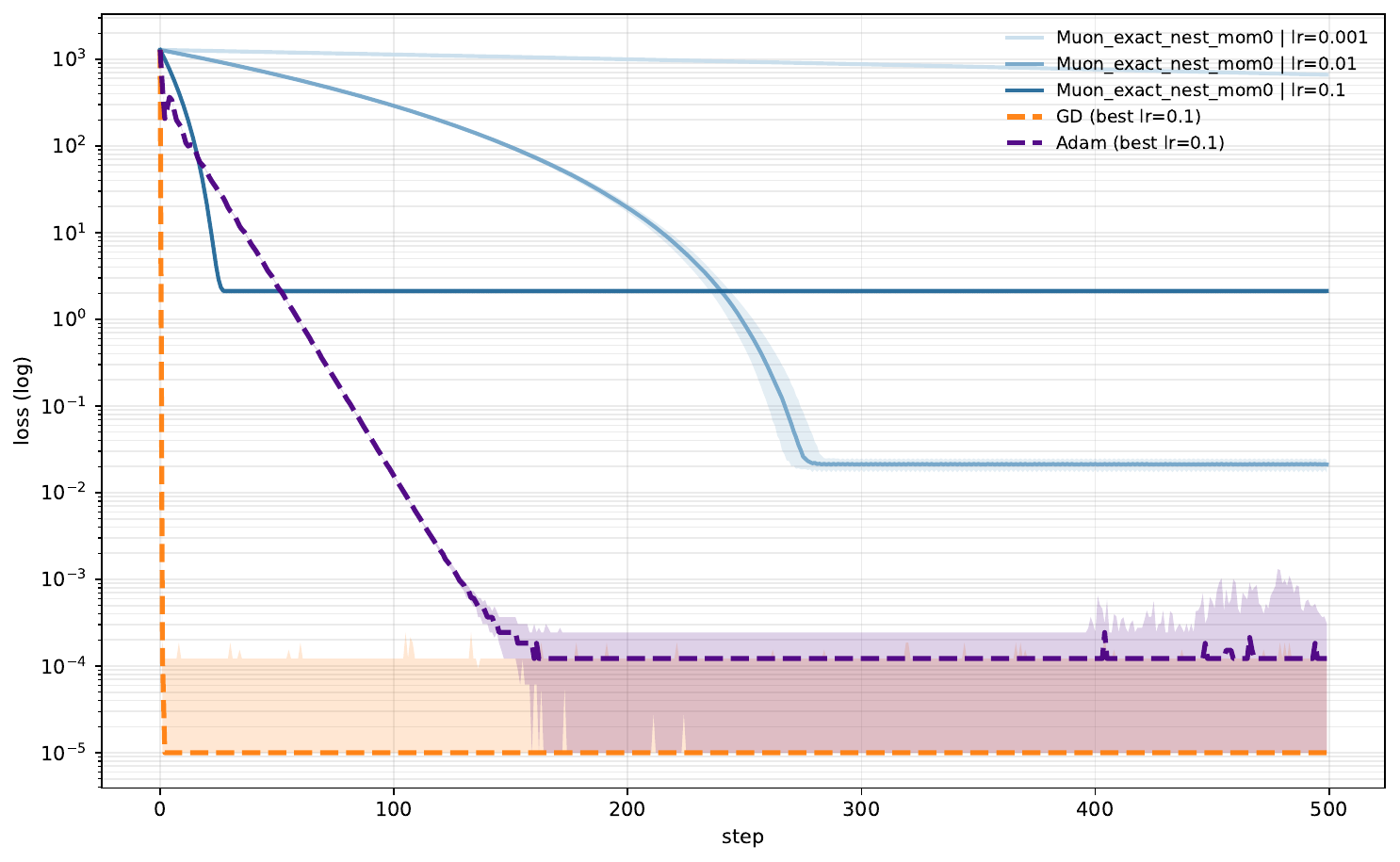}
    \caption{\texttt{max\_spiked} spectrum (unfavorable to \Muon{}).}
    \label{fig:avg-traj-maxspike}
  \end{subfigure}
  \caption{
  Averaged loss curves with shaded bands covering $95\%$ of trajectories
  for the quadratic objective
$\loss(W)=\tfrac12\langle W, A W\rangle + \langle B, W\rangle + c$,
averaged over $100$ random initializations
$[W_0]_{ij}\overset{\text{i.i.d.}}{\sim}\mathcal N(0,1/n)$.
Across panels, the matrix $A$ has the same condition number but different spectrum shapes (see \Cref{fig:eig-distributions-main}). 
On the \texttt{min\_spiked} spectrum, the loss achieved by \GD{} is numerically very close to zero.
We clip it at $10^{-5}$ for plot readability.
}
  \label{fig:avg-traj-spiked}
\end{figure*}

\textbf{Ablations: vanishing step sizes, other \Muon{} variants, and sample trajectories.} 
(i) \Cref{app:controlled-spectrum-magnitude} reports the average loss ratios at different steps $t\in\{T/10,T/2,T\}$, 
complementing the final time $t=T$ results in \Cref{fig:improvement-bars}. 
We also include in \Cref{app:controlled-spectrum-bars} the same plot as \Cref{fig:improvement-bars} but 
including the initial loss level for each spectrum family, 
instead of just the reduction achieved from it. 
(ii) \Cref{app:controlled-spectrum-trajectories} shows representative \emph{sample trajectories} per spectrum family 
to complement the averaged ones in \Cref{fig:avg-traj-spiked}. 
(iii) The same appendix includes ablations with vanishing learning-rate schedules and
different \Muon{} variants (with or without Nesterov-style momentum and
with exact or Newton--Schulz approximate projection). 
While vanishing step sizes affect late-stage dynamics, 
and momentum or approximate projection mainly influence how noisy the loss curves are,
they do not change the qualitative conclusion on these quadratics:  
conditioning alone does not explain \Muon{}'s speed relative to \GD{}. 
The same ranking across spectra is observed across these variants.

\section{Approximate Projection Can Change the Dynamics (and Sometimes the Speed)}
\label{sec:noisy}

The previous section analyzed an idealization common in theory:
\emph{exact} polar projection with a fixed step size.
In practice, however, \Muon{} does \emph{not} compute a polar factor exactly: it applies a low-degree Newton--Schulz-type polynomial (plus implementation choices such as rescaling and momentum).
A widespread modeling stance treats this inexactness as a nuisance term that monotonically worsens guarantees relative to the exact baseline.
But, is it?

\begin{quote}
\textbf{Q2}: 
\emph{Already on simple strongly convex objectives, can approximate projection qualitatively change reachability and change iteration counts?}
\end{quote}

\subsection{A minimal model: perturbed sign dynamics}

On the isotropic quadratic $\loss(W)=\tfrac12\|W\|_{\mathrm{F}}^2$, exact Stiefel updates reduce to coordinate-wise sign dynamics on singular values (\Cref{sec:setup-minimizer}). 
An approximate polar factor replaces each singular value $\sv_{t,i}$ by 
its sign plus a distortion:
\begin{equation}
  \sv_{t+1,i}
  \;=\;
  \sv_{t,i} - \stepsize\bigl(\sign(\sv_{t,i}) + \eta_{t,i}\bigr).  
  \label{eq:approx-sign-generic}
\end{equation}
In practice, \Muon{} approximates the $\sign$ function 
by one that maps most of $[0,1]$ to $[1-\delta,1+\delta]$, with $\delta\approx0.3$. This means $|\eta_{t,i}| \lesssim 0.3$ in practice. These distortions $\eta_{t,i}$ are neither independent across $i$ nor across $t$. 
Indeed, all singular values are first jointly mapped to $[0,1]$ before applying the polynomial approximation (\Cref{alg:muon-aux}). 
This global 
rescaling step 
creates dependencies across $i$.
Moreover, the optimizer correlates the iterates over time. 

Here, we take a first step to analyze 
and isolate the \emph{role} of inexactness 
by studying a simplified perturbation model:
\begin{equation}
  \sv_{t+1}
  \;=\;
  \sv_t - \stepsize\bigl(\sign(\sv_t) + \sigma \xi_{t+1}\bigr),
  \label{eq:noisy-sign}
\end{equation}
where $(\xi_t)_t$ are i.i.d.\ standard Gaussians $\mathcal{N}(0,1)$, and $\sigma \geq 0$ controls the perturbation magnitude. 

\subsection{Noise breaks confinement on the isotropic quadratic, and typical hitting times can be non-monotone}

Define the hitting time of an $\eps$-loss for $\ell(\sv)=\tfrac12 \sv^2$ by
\[
T_\eps^{(\sigma)} := \inf\{t\ge 0:\ \ell(\sv_t)\le \eps\},
\]
where $(\sv_t)_{t\geq0}$ follows~\eqref{eq:approx-sign-generic} with some given $\sv_0$.
For $\sigma=0$, the trajectory remains on a lattice $\sv_0+\stepsize\Z$ and generically falls into a two-cycle around $0$, making sufficiently small loss levels unreachable. 
A first key observation is that noise breaks this lattice confinement: as soon as $\sigma>0$, 
reachability is restored almost surely; see~\Cref{prop:noise-breaks-grid} for a precise statement and~\Cref{sec:noise-breaks-grid} for a proof.

Beyond reachability, we consider a practical proxy: 
the distribution of $T_\eps^{(\sigma)}$ as a function of $\sigma$. 
On the 1D toy model \eqref{eq:noisy-sign}, 
the typical scale of $T_\eps^{(\sigma)}$ deteriorates both as $\sigma\to 0$ 
(recovering near-deterministic cycling) and as $\sigma\to\infty$ (diffusive overshoot dominates). 
Indeed, it must be greater than any fixed $t$ with high probability in both limits:
$\PP(T_\eps^{(\sigma)} > t)\to 1$ as $\sigma\to 0$ or $\sigma\to\infty$ for any fixed $t$;
see~\Cref{prop:small-large-noise-long-hitting} and~\Cref{sec:small-large-noise-long-hitting}.

In experiments, we report 
the typical iteration count needed to reach $\eps$-loss (median of $T_\eps^{(\sigma)}$) 
over $10^4$ independent runs. 
For reference, the baseline $\lceil \sv_0/\stepsize \rceil$ corresponds to the number of steps  
at which the noiseless trajectory first crosses $0$ and enters the cycling phase, 
and thus represents the fastest possible approach to $\eps$-loss \emph{without} noise. 

\Cref{fig:median-T} shows that the iteration count is \emph{non-monotone} in
$\sigma$: 
there is a ``sweet spot'', an intermediate noise level,  
which improves which losses are reachable in a fixed iteration count.  
Noise is therefore \emph{beneficial} on this toy model.




\begin{figure}[t]
  \centering
  \includegraphics[width=0.65\linewidth]{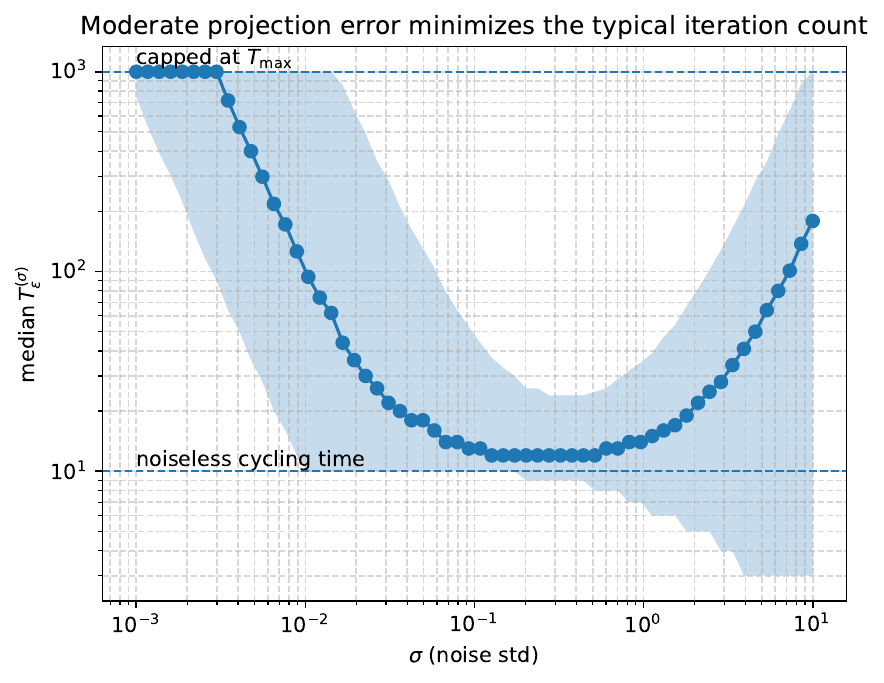}
  \caption{Median $\eps$-hitting time exhibits a similar ``sweet spot'' where moderate 
  projection error $\sigma$ in \eqref{eq:noisy-sign} accelerates convergence. 
  $T_\eps^{(\sigma)}$ capped at $T_{\max}=1000$ steps (top dashed line).
  Noiseless cycling time $|\sv_0|/\stepsize$ shown for reference (bottom dashed line). 
The shaded band shows the central $95\%$ interval of $T_\eps^{(\sigma)}$ across runs. 
Details in \Cref{app:noisy-sign-xps}. 
  }
  \label{fig:median-T}
\end{figure}

\subsection{Quadratic evidence with actual Newton--Schulz projection: approximation can change who wins}

The synthetic perturbation model above isolates a mechanism, 
but Newton--Schulz approximation is not i.i.d.\ noise.
We therefore also test with the actual Newton--Schulz approximate projection on 
the same controlled-spectrum quadratic family as in \Cref{sec:setup-conditioning}.
We compare two \Muon{} variants: \Muon{} with exact projection 
vs.\ \Muon{} with Newton--Schulz approximate projection 
(under no momentum and constant step sizes; details in \Cref{app:controlled-spectrum-ns}). 

\Cref{tab:ns-vs-exact-win} reports win rates (over random initializations) for \Muon{} 
with Newton--Schulz approximate projection vs.\ \Muon{} with exact projection at milestones $t\in\{T/10,T/2,T\}$.
The outcome is \emph{not uniform across spectra}:
in the most favorable shape for \Muon{} in \Cref{sec:setup-conditioning} (min-spiked), 
the approximate method tends to reach smaller losses by time $T$,
while on many other shapes it does not. 

\begin{table}[t]
\centering
\small
\setlength{\tabcolsep}{4pt}
\renewcommand{\arraystretch}{1.15}
\caption{Win rates for \Muon{} with Newton--Schulz approximate projection vs.\ \Muon{} with exact projection, comparing the best loss achieved across learning rates up to time $t$. Spectrum naming matches \Cref{fig:eig-distributions}; see \Cref{app:controlled-spectrum-ns} for exact definitions and hyperparameters.}
\label{tab:ns-vs-exact-win}
\begin{tabular}{lccc}
\toprule
kind & $t=T/10$ & $t=T/2$ & $t=T$ \\
\midrule
\texttt{max\_spiked} & 0 & 0 & 0.92 \\
\texttt{min\_spiked} & 0.46 & 0.57 & 0.94 \\
\texttt{uniform} & 0 & 0 & 0 \\
\texttt{gaussian} & 0 & 0 & 0 \\
\texttt{linear\_decay\_to\_max} & 0 & 0 & 0 \\
\texttt{u\_shaped} & 0 & 0 & 0 \\
\texttt{geometric\_decay\_to\_max} & 0 & 0 & 0 \\
\bottomrule
\end{tabular}
\end{table}

\Cref{app:controlled-spectrum-ns} complements \Cref{tab:ns-vs-exact-win} with 
mean best-loss ratios: when approximate projection wins, 
it does so by a significant margin in best-loss value (e.g., $30\times$ smaller), 
otherwise it remains within a factor $2$ of the exact method. 
Taken together, these experiments support that: 
\emph{approximate projection can matter for iteration complexity, but 
whether it helps or not depends on the setup}. 
This makes it difficult to justify a universal 
``approximation is worse / better'' abstraction, 
and it motivates more fine-grained theory of inexact polar steps that goes 
beyond monotone degradation with an error magnitude. 
In particular, what are the effects of 
different 
error profiles 
(e.g. bounded vs.\ heavy-tailed, 
centered vs.\ biased) \cite{amsel2025polarexpress}? 

\section{One-Step Improvement on Quadratics is Not a Reliable Proxy for End-to-End Speed}
\label{sec:line-search}
A recurring narrative (cf.\ \textbf{Q3}) is that \Muon{}-like directions are preferable because, on a quadratic or locally quadratic model, they can achieve a larger \emph{per-iteration} decrease than the Euclidean gradient direction when paired with a suitable step size (often motivated by line search on a quadratic proxy) \cite{davis2025whenspectral,su2025isotropiccurvature}.
This type of argument is appealing because it is local, interpretable, and easy to validate numerically on a single step. 
Here, we aim at clarifying what local analysis can and cannot certify about full trajectories in general. We observe the following:

\begin{quote}
\emph{Even on quadratics---the setting most favorable to 
existing one-step proxy arguments---greedy per-step superiority does not necessarily imply faster end-to-end convergence.}
\end{quote}

\textbf{Setup.}
Consider the quadratic $\loss(W) = \frac12 \langle W, A W\rangle$. 
At an iterate $W$, we compare two update directions:
the Euclidean gradient direction $D_{\rm GD}(W)\;=\;\nabla \loss(W)$, 
and the Stiefel/polar direction
$
D_{\rm St}(W)\;=\;\proj(\nabla \loss(W))$. 
To avoid privileging either method through a step-size choice, we let each direction take its own \emph{exact} line-search step size 
$\alpha^\star(D) \in \arg\min_{\alpha\in\R} \loss(W - \alpha D)$. The associated one-step 
decrease is $\Delta(D) \coloneqq \loss(W) - \loss(W - \alpha^\star(D) D)$. 
The local narrative ``Stiefel is better than \GD{}'' corresponds to regimes where $\Delta(D_{\rm St}) > \Delta(D_{\rm GD})$.

\textbf{Observation: local advantage can coexist with global disadvantage.}
There exist quadratics and initializations $W_0$ such that 
a \emph{greedy} policy that, at each iterate,
chooses between a gradient step and a Stiefel/polar step, each with its own 
\emph{optimal} line-search step size, 
yet has a worse end-to-end loss decrease than \GD{} with line search. 
\Cref{fig:counterexample-quad-main} illustrates this 
(and \Cref{app:exact-1d-line-search} reports how the greedy policy 
also loses on the gradient norm and distance to the minimizer). 

For more general objectives, the message is the same: 
extra care is needed when one-step comparisons are made via 
quadratic proxies (e.g., a Taylor expansion or a local quadratic upper bound): 
we just saw that even on true quadratics, where there is no proxy gap, 
\emph{one-step superiority does not imply end-to-end speed-up}. 
This suggests that one-step improvement arguments should come 
with additional structure that rules out the exhibited phenomena. 

\begin{figure}[t]
  \centering
  \includegraphics[width=0.65\linewidth]{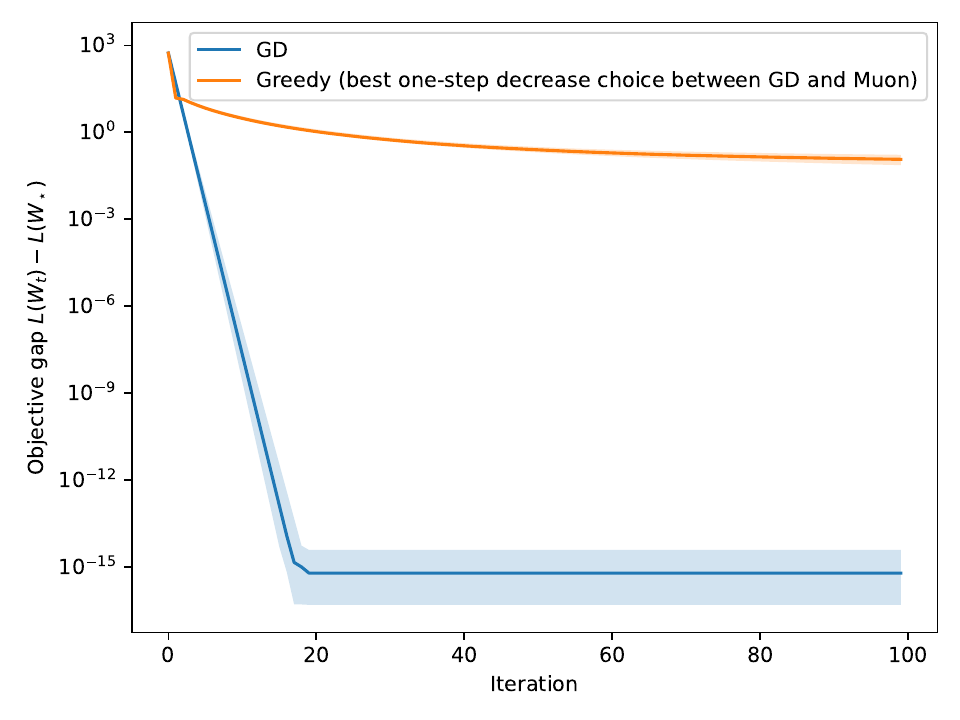}
\caption{
Objective gap $\loss(W_t)-\loss(W_\star)$ for the quadratic $\loss(W)=\tfrac12 \langle W, A W\rangle$ (all matrices are $n\times n$, $n=100$).
Here $A=QS Q^\top$ with $Q$ random orthogonal and $S=\diag(1,\ldots,1,10^3)$, so $\kappa=10^3$.
We compare \GD{} (exact line search along the Euclidean gradient) to a \emph{greedy} policy that, 
at each iteration, picks the direction (gradient or Stiefel/polar) yielding the larger exact line-search decrease.
On this instance, the greedy policy selects the Stiefel/polar step at every iteration, yet \GD{} converges much faster end-to-end.
Curves show the median across $100$ random initializations $W_0$ with i.i.d.\ entries $(W_0)_{ij}\sim\mathcal N(0,1/n)$; shaded bands contain the central $95\%$ of runs.
}
  \label{fig:counterexample-quad-main}
\end{figure}

\section{Conclusion}

This work highlights that the behavior of \Muon{} is governed by
discrete-time phenomena that are not yet fully understood,
even on simple strongly convex quadratics. 

Addressing these aspects theoretically likely requires new tools:
one-step local proxies 
are generally not sufficient to predict end-to-end speed-up,
and worst-case guarantees can yield pessimistic predictions
that degrade monotonically with approximation error.

The analysis suggests that at least two ingredients deserve finer modeling: 
(i) the spectral properties of the loss beyond mere conditioning, 
affecting big-O constants; and
(ii) the approximation error in the polar projection step,
which qualitatively alters reachability and can accelerate hitting times.
Open questions remain on how best to characterize these aspects theoretically.

\section*{Acknowledgements}

This work was supported in part by the Swiss 
State Secretariat for Education, Research and Innovation (SERI)
under contract number MB22.00027.

\bibliographystyle{plainnat}
\bibliography{references}

@misc{jordan2024muon,
  author       = {Keller Jordan},
  title        = {{M}uon: An optimizer for the hidden layers in neural networks},
  howpublished = {\url{https://kellerjordan.github.io/posts/muon/}},
  year         = {2024},
  note         = {Accessed January 2026}
}

@article{liu2025muonscalablellmtraining,
  author       = {Jingyuan Liu and
                  Jianlin Su and
                  Xingcheng Yao and
                  Zhejun Jiang and
                  Guokun Lai and
                  Yulun Du and
                  Yidao Qin and
                  Weixin Xu and
                  Enzhe Lu and
                  Junjie Yan and
                  Yanru Chen and
                  Huabin Zheng and
                  Yibo Liu and
                  Shaowei Liu and
                  Bohong Yin and
                  Weiran He and
                  Han Zhu and
                  Yuzhi Wang and
                  Jianzhou Wang and
                  Mengnan Dong and
                  Zheng Zhang and
                  Yongsheng Kang and
                  Hao Zhang and
                  Xinran Xu and
                  Yutao Zhang and
                  Yuxin Wu and
                  Xinyu Zhou and
                  Zhilin Yang},
  title        = {{M}uon is Scalable for {LLM} Training},
  journal      = {CoRR},
  volume       = {abs/2502.16982},
  year         = {2025},
  url          = {https://doi.org/10.48550/arXiv.2502.16982},
  doi          = {10.48550/ARXIV.2502.16982},
  eprinttype    = {arXiv},
  eprint       = {2502.16982},
  bibsource    = {dblp computer science bibliography, https://dblp.org}
}

@misc{nagashima2026improvedconvergenceratesmuon,
      title={Improved Convergence Rates of {M}uon Optimizer for Nonconvex Optimization}, 
      author={Shuntaro Nagashima and Hideaki Iiduka},
      year={2026},
      eprint={2601.19400},
      archivePrefix={arXiv},
      primaryClass={math.OC},
      url={https://arxiv.org/abs/2601.19400}, 
}

@inproceedings{pethick2025normconstrainedlmo,
  author       = {Thomas Pethick and
                  Wanyun Xie and
                  Kimon Antonakopoulos and
                  Zhenyu Zhu and
                  Antonio Silveti{-}Falls and
                  Volkan Cevher},
  title        = {Training Deep Learning Models with Norm-Constrained LMOs},
  booktitle    = {Forty-second International Conference on Machine Learning, {ICML}
                  2025, Vancouver, BC, Canada, July 13-19, 2025},
  publisher    = {OpenReview.net},
  year         = {2025},
  url          = {https://openreview.net/forum?id=2Oqm2IzTy9},
  bibsource    = {dblp computer science bibliography, https://dblp.org}
}

@misc{tuddenham2022orthogonalisinggradientsspeedneural,
      title={Orthogonalising gradients to speed up neural network optimisation}, 
      author={Mark Tuddenham and Adam Prügel-Bennett and Jonathan Hare},
      year={2022},
      eprint={2202.07052},
      archivePrefix={arXiv},
      primaryClass={cs.LG},
      url={https://arxiv.org/abs/2202.07052}, 
}

@misc{ma2026preconditioningbenefitsspectralorthogonalization,
      title={Preconditioning Benefits of Spectral Orthogonalization in {M}uon}, 
      author={Jianhao Ma and Yu Huang and Yuejie Chi and Yuxin Chen},
      year={2026},
      eprint={2601.13474},
      archivePrefix={arXiv},
      primaryClass={cs.LG},
      url={https://arxiv.org/abs/2601.13474}, 
}

@misc{li2025noteconvergencemuon,
      title={A Note on the Convergence of {M}uon}, 
      author={Jiaxiang Li and Mingyi Hong},
      year={2025},
      eprint={2502.02900},
      archivePrefix={arXiv},
      primaryClass={math.OC},
      url={https://arxiv.org/abs/2502.02900}, 
}

@article{riabinin2025gluon,
  author       = {Artem Riabinin and
                  Egor Shulgin and
                  Kaja Gruntkowska and
                  Peter Richt{\'{a}}rik},
  title        = {Gluon: Making {M}uon {\&} {S}cion Great Again! (Bridging Theory and
                  Practice of {LMO}-based Optimizers for {LLM}s)},
  journal      = {CoRR},
  volume       = {abs/2505.13416},
  year         = {2025},
  url          = {https://doi.org/10.48550/arXiv.2505.13416},
  doi          = {10.48550/ARXIV.2505.13416},
  eprinttype    = {arXiv},
  eprint       = {2505.13416},
  bibsource    = {dblp computer science bibliography, https://dblp.org}
}

@article{shen2025convergencemuon,
  author       = {Wei Shen and
                  Ruichuan Huang and
                  Minhui Huang and
                  Cong Shen and
                  Jiawei Zhang},
  title        = {On the Convergence Analysis of {M}uon},
  journal      = {CoRR},
  volume       = {abs/2505.23737},
  year         = {2025},
  url          = {https://doi.org/10.48550/arXiv.2505.23737},
  doi          = {10.48550/ARXIV.2505.23737},
  eprinttype    = {arXiv},
  eprint       = {2505.23737},
  bibsource    = {dblp computer science bibliography, https://dblp.org}
}

@article{chen2025muonspectralnorm,
  author       = {Lizhang Chen and
                  Jonathan Li and
                  Qiang Liu},
  title        = {{M}uon Optimizes Under Spectral Norm Constraints},
  journal      = {CoRR},
  volume       = {abs/2506.15054},
  year         = {2025},
  url          = {https://doi.org/10.48550/arXiv.2506.15054},
  doi          = {10.48550/ARXIV.2506.15054},
  eprinttype    = {arXiv},
  eprint       = {2506.15054},
  bibsource    = {dblp computer science bibliography, https://dblp.org}
}

@misc{sato2025muonconvergence,
      title={Convergence Bound and Critical Batch Size of {M}uon Optimizer}, 
      author={Naoki Sato and Hiroki Naganuma and Hideaki Iiduka},
      year={2025},
      eprint={2507.01598},
      archivePrefix={arXiv},
      primaryClass={cs.LG},
      url={https://arxiv.org/abs/2507.01598}, 
}

@article{papyan2020hessiancharacterization,
  author  = {Vardan Papyan},
  title   = {Traces of Class/Cross-Class Structure Pervade Deep Learning Spectra},
  journal = {Journal of Machine Learning Research},
  year    = {2020},
  volume  = {21},
  number  = {252},
  pages   = {1--64},
  url     = {http://jmlr.org/papers/v21/20-933.html}
}

@InProceedings{ghorbani2019hessianeigvaldensity,
  title = 	 {An Investigation into Neural Net Optimization via Hessian Eigenvalue Density},
  author =       {Ghorbani, Behrooz and Krishnan, Shankar and Xiao, Ying},
  booktitle = 	 {Proceedings of the 36th International Conference on Machine Learning},
  pages = 	 {2232--2241},
  year = 	 {2019},
  editor = 	 {Chaudhuri, Kamalika and Salakhutdinov, Ruslan},
  volume = 	 {97},
  series = 	 {Proceedings of Machine Learning Research},
  month = 	 {09--15 Jun},
  publisher =    {PMLR},
  url = 	 {https://proceedings.mlr.press/v97/ghorbani19b.html}
}

@misc{sagun2017singularityhessian,
author       = {Levent Sagun and
                  L{\'{e}}on Bottou and
                  Yann LeCun},
title={Eigenvalues of the Hessian in Deep Learning: Singularity and Beyond},
journal      = {CoRR},
year={2017},
url={https://arxiv.org/abs/1611.07476},
  eprinttype    = {arXiv},
  eprint       = {1611.07476},
}

@inproceedings{sagun2018empiricalhessianminspiked,
  author       = {Levent Sagun and
                  Utku Evci and
                  V. Ugur G{\"{u}}ney and
                  Yann N. Dauphin and
                  L{\'{e}}on Bottou},
  title        = {Empirical Analysis of the Hessian of Over-Parametrized Neural Networks},
  booktitle    = {6th International Conference on Learning Representations, {ICLR} 2018,
                  Vancouver, BC, Canada, April 30 - May 3, 2018, Workshop Track Proceedings},
  publisher    = {OpenReview.net},
  year         = {2018},
  url          = {https://openreview.net/forum?id=rJO1_M0Lf},
  bibsource    = {dblp computer science bibliography, https://dblp.org}
}

@inproceedings{
tang2025powerlawhessianspectrum,
title={Investigating the Overlooked Hessian Structure: From {CNN}s to {LLM}s},
author={Qian-Yuan Tang and Yufei Gu and Yunfeng Cai and Mingming Sun and Ping Li and zhou Xun and Zeke Xie},
booktitle={Forty-second International Conference on Machine Learning},
year={2025},
url={https://openreview.net/forum?id=o62ZzfCEwZ}
}

@book{LeGall2022,
  author    = {{Le Gall}, Jean-Fran\c{c}ois},
  title     = {Measure Theory, Probability, and Stochastic Processes},
  series    = {Graduate Texts in Mathematics},
  volume    = {295},
  publisher = {Springer Cham},
  year      = {2022},
  isbn      = {978-3-031-14205-5},
  doi       = {10.1007/978-3-031-14205-5},
  url       = {https://link.springer.com/book/10.1007/978-3-031-14205-5},
}

@article{shah2025practicalefficiencymuon,
  author       = {Ishaan Shah and
                  Anthony M. Polloreno and
                  Karl Stratos and
                  Philip Monk and
                  Adarsh Chaluvaraju and
                  Andrew Hojel and
                  Andrew Ma and
                  Anil Thomas and
                  Ashish Tanwer and
                  Darsh J. Shah and
                  Khoi Nguyen and
                  Kurt Smith and
                  Michael Callahan and
                  Michael Pust and
                  Mohit Parmar and
                  Peter Rushton and
                  Platon Mazarakis and
                  Ritvik Kapila and
                  Saurabh Srivastava and
                  Somanshu Singla and
                  Tim Romanski and
                  Yash Vanjani and
                  Ashish Vaswani},
  title        = {Practical Efficiency of {M}uon for Pretraining},
  journal      = {CoRR},
  volume       = {abs/2505.02222},
  year         = {2025},
  url          = {https://doi.org/10.48550/arXiv.2505.02222},
  doi          = {10.48550/ARXIV.2505.02222},
  eprinttype    = {arXiv},
  eprint       = {2505.02222},
  bibsource    = {dblp computer science bibliography, https://dblp.org}
}

@InProceedings{cunha22avgcaseuniversality,
  title = 	 {Only tails matter: Average-Case Universality and Robustness in the Convex Regime},
  author =       {Cunha, Leonardo and Gidel, Gauthier and Pedregosa, Fabian and Scieur, Damien and Paquette, Courtney},
  booktitle = 	 {Proceedings of the 39th International Conference on Machine Learning},
  pages = 	 {4474--4491},
  year = 	 {2022},
  volume = 	 {162},
  series = 	 {Proceedings of Machine Learning Research},
  month = 	 {17--23 Jul},
  publisher =    {PMLR},
  url = 	 {https://proceedings.mlr.press/v162/cunha22a.html}
}

@misc{li2019hessianbasedanalysissgd,
      title={Hessian based analysis of SGD for Deep Nets: Dynamics and Generalization}, 
      author={Xinyan Li and Qilong Gu and Yingxue Zhou and Tiancong Chen and Arindam Banerjee},
      year={2019},
      eprint={1907.10732},
      archivePrefix={arXiv},
      primaryClass={cs.LG},
      url={https://arxiv.org/abs/1907.10732}, 
}

@misc{gur-ari2019gradient,
title={Gradient Descent Happens in a Tiny Subspace},
author={Guy Gur-Ari and Daniel A. Roberts and Ethan Dyer},
year={2019},
url={https://openreview.net/forum?id=ByeTHsAqtX},
}

@inproceedings{song2025doessgdreallyhappenintinysubspaces,
title={Does {SGD} really happen in tiny subspaces?},
author={Minhak Song and Kwangjun Ahn and Chulhee Yun},
booktitle={The Thirteenth International Conference on Learning Representations},
year={2025},
url={https://openreview.net/forum?id=v6iLQBoIJw}
}

@InProceedings{pedregosa20avgcasequadratic,
  title = 	 {Acceleration through spectral density estimation},
  author =       {Pedregosa, Fabian and Scieur, Damien},
  booktitle = 	 {Proceedings of the 37th International Conference on Machine Learning},
  pages = 	 {7553--7562},
  year = 	 {2020},
  volume = 	 {119},
  series = 	 {Proceedings of Machine Learning Research},
  month = 	 {13--18 Jul},
  publisher =    {PMLR},
  url = 	 {https://proceedings.mlr.press/v119/pedregosa20a.html}
}

@article{kovalev2025understandinggo,
  author       = {Dmitry Kovalev},
  title        = {Understanding Gradient Orthogonalization for Deep Learning via Non-Euclidean
                  Trust-Region Optimization},
  journal      = {CoRR},
  volume       = {abs/2503.12645},
  year         = {2025},
  url          = {https://doi.org/10.48550/arXiv.2503.12645},
  doi          = {10.48550/ARXIV.2503.12645},
  eprinttype    = {arXiv},
  eprint       = {2503.12645},
  bibsource    = {dblp computer science bibliography, https://dblp.org}
}

@article{shulgin2025beyondideal,
  author       = {Egor Shulgin and
                  Sultan AlRashed and
                  Francesco Orabona and
                  Peter Richt{\'{a}}rik},
  title        = {Beyond the Ideal: Analyzing the Inexact {M}uon Update},
  journal      = {CoRR},
  volume       = {abs/2510.19933},
  year         = {2025},
  url          = {https://doi.org/10.48550/arXiv.2510.19933},
  doi          = {10.48550/ARXIV.2510.19933},
  eprinttype    = {arXiv},
  eprint       = {2510.19933},
  bibsource    = {dblp computer science bibliography, https://dblp.org}
}

@article{gruntkowska2025ef21muon,
  author       = {Kaja Gruntkowska and
                  Alexander Gaponov and
                  Zhirayr Tovmasyan and
                  Peter Richt{\'{a}}rik},
  title        = {Error Feedback for {M}uon and Friends},
  journal      = {CoRR},
  volume       = {abs/2510.00643},
  year         = {2025},
  url          = {https://doi.org/10.48550/arXiv.2510.00643},
  doi          = {10.48550/ARXIV.2510.00643},
  eprinttype    = {arXiv},
  eprint       = {2510.00643},
  bibsource    = {dblp computer science bibliography, https://dblp.org}
}

@inproceedings{anon2025newtonschulz,
title={Convergence of {M}uon with {N}ewton--{S}chulz},
author={Kim, Gyu Yeol and Oh, {Min-hwan}},
booktitle={The Fourteenth International Conference on Learning Representations},
year={2026},
url={https://openreview.net/forum?id=lJSfxtLpLm}
}

@article{davis2025whenspectral,
  author       = {Damek Davis and
                  Dmitriy Drusvyatskiy},
  title        = {When do spectral gradient updates help in deep learning?},
  journal      = {CoRR},
  volume       = {abs/2512.04299},
  year         = {2025},
  url          = {https://doi.org/10.48550/arXiv.2512.04299},
  doi          = {10.48550/ARXIV.2512.04299},
  eprinttype    = {arXiv},
  eprint       = {2512.04299},
  bibsource    = {dblp computer science bibliography, https://dblp.org}
}

@article{su2025isotropiccurvature,
  author       = {Weijie Su},
  title        = {Isotropic Curvature Model for Understanding Deep Learning Optimization:
                  Is Gradient Orthogonalization Optimal?},
  journal      = {CoRR},
  volume       = {abs/2511.00674},
  year         = {2025},
  url          = {https://doi.org/10.48550/arXiv.2511.00674},
  doi          = {10.48550/ARXIV.2511.00674},
  eprinttype    = {arXiv},
  eprint       = {2511.00674},
  bibsource    = {dblp computer science bibliography, https://dblp.org}
}

@article{amsel2025polarexpress,
  author       = {Noah Amsel and
                  David Persson and
                  Christopher Musco and
                  Robert Gower},
  title        = {The Polar Express: Optimal Matrix Sign Methods and Their Application
                  to the {M}uon Algorithm},
  journal      = {CoRR},
  volume       = {abs/2505.16932},
  year         = {2025},
  url          = {https://doi.org/10.48550/arXiv.2505.16932},
  doi          = {10.48550/ARXIV.2505.16932},
  eprinttype    = {arXiv},
  eprint       = {2505.16932},
  bibsource    = {dblp computer science bibliography, https://dblp.org}
}

@misc{kang2025psdpoly,
      title={Factorization-free Orthogonal Projection onto the Positive Semidefinite Cone with Composite Polynomial Filtering}, 
      author={Shucheng Kang and Haoyu Han and Antoine Groudiev and Heng Yang},
      year={2025},
      eprint={2507.09165},
      archivePrefix={arXiv},
      primaryClass={math.OC},
      url={https://arxiv.org/abs/2507.09165}, 
}

@misc{nanogpt24speedrun,
  author       = {Keller Jordan and Jeremy Bernstein and Brendan Rappazzo and
                  @fernbear.bsky.social and Boza Vlado and You Jiacheng and
                  Franz Cesista and Braden Koszarsky and @Grad62304977},
  title        = {modded-nanogpt: Speedrunning the NanoGPT baseline},
  year         = {2024},
  url          = {https://github.com/KellerJordan/modded-nanogpt},
  note         = {Accessed January 13th, 2026}
}

@article{grishina2025chebyshevns,
  author       = {Ekaterina Grishina and
                  Matvey Smirnov and
                  Maxim V. Rakhuba},
  title        = {Accelerating {N}ewton-{S}chulz Iteration for Orthogonalization via {C}hebyshev-type
                  Polynomials},
  journal      = {CoRR},
  volume       = {abs/2506.10935},
  year         = {2025},
  url          = {https://doi.org/10.48550/arXiv.2506.10935},
  doi          = {10.48550/ARXIV.2506.10935},
  eprinttype    = {arXiv},
  eprint       = {2506.10935},
  bibsource    = {dblp computer science bibliography, https://dblp.org}
}

@misc{cesista2025muonoptcoeffs,
  author = {Franz Louis Cesista and You Jiacheng and Keller Jordan},
  title = {{S}queezing 1-2\% Efficiency Gains Out of {M}uon by Optimizing the {N}ewton-{S}chulz {C}oefficients},
  year = {2025},
  month = {February},
  day = {21},
  url = {https://leloykun.github.io/ponder/muon-opt-coeffs/},
  note         = {Accessed January 2026}
}

@article{lau2025polargrad,
  author       = {Tim Tsz{-}Kit Lau and
                  Qi Long and
                  Weijie Su},
  title        = {PolarGrad: {A} Class of Matrix-Gradient Optimizers from a Unifying
                  Preconditioning Perspective},
  journal      = {CoRR},
  volume       = {abs/2505.21799},
  year         = {2025},
  url          = {https://doi.org/10.48550/arXiv.2505.21799},
  doi          = {10.48550/ARXIV.2505.21799},
  eprinttype    = {arXiv},
  eprint       = {2505.21799},
  bibsource    = {dblp computer science bibliography, https://dblp.org}
}

@online{boumal2025polarpoly,
  author = {Boumal, Nicolas and Gonon, Antoine},
  title = {Designing Polynomials for {Muon's} Polar Factorization: Focus
    on Quintics},
  date = {2025-06-30},
  year = {2025},
  url = {www.racetothebottom.xyz/posts/polar-poly/},
  langid = {en},
  note = {Accessed January 2026}

}

@article{ahn2025dion,
      title={Dion: Distributed Orthonormalized Updates}, 
      author={Kwangjun Ahn and Byron Xu and Natalie Abreu and Ying Fan and Gagik Magakyan and Pratyusha Sharma and Zheng Zhan and John Langford},
      year={2025},
      eprint={2504.05295},
      archivePrefix={arXiv},
      primaryClass={cs.LG},
      url={https://arxiv.org/abs/2504.05295}, 
}

@article{he2025lowrankmuon,
  author       = {Chuan He and
                  Zhanwang Deng and
                  Zhaosong Lu},
  title        = {Low-rank Orthogonalization for Large-scale Matrix Optimization with
                  Applications to Foundation Model Training},
  journal      = {CoRR},
  volume       = {abs/2509.11983},
  year         = {2025},
  url          = {https://doi.org/10.48550/arXiv.2509.11983},
  doi          = {10.48550/ARXIV.2509.11983},
  eprinttype    = {arXiv},
  eprint       = {2509.11983},
  bibsource    = {dblp computer science bibliography, https://dblp.org}
}

@article{khaled2025muonbp,
  author       = {Ahmed Khaled and
                  Kaan Ozkara and
                  Tao Yu and
                  Mingyi Hong and
                  Youngsuk Park},
  title        = {{M}uonBP: Faster {M}uon via Block-Periodic Orthogonalization},
  journal      = {CoRR},
  volume       = {abs/2510.16981},
  year         = {2025},
  url          = {https://doi.org/10.48550/arXiv.2510.16981},
  doi          = {10.48550/ARXIV.2510.16981},
  eprinttype    = {arXiv},
  eprint       = {2510.16981},
  bibsource    = {dblp computer science bibliography, https://dblp.org}
}

\newpage
\appendix
\onecolumn

\section{Muon in the modded-nanoGPT speedrun setup (pseudocode)}
\label{app:muon-pseudocode}

We consider the code used in the \href{https://github.com/KellerJordan/modded-nanogpt/blob/7d502b97dac8b8116b675dd6d670cf3ea8ec0037/train_gpt.py}{nanoGPT speedrun training script}, as it 
stays up-to-date with the latest best practices for performance. A simplified (non-distributed) version is presented in \Cref{alg:muon-speedrun}. In the speedrun setup, \Muon{} is applied primarily to 2D weight matrices
(e.g., linear layers), while other parameters (e.g., embeddings, biases, and
non-matrix tensors) are typically handled by AdamW. 

In particular, 
\Cref{alg:muon-speedrun} and \Cref{alg:muon-aux} 
include for reference the global rescaling applied before the Newton--Schulz-like polynomials (\href{https://github.com/KellerJordan/modded-nanogpt/blob/7d502b97dac8b8116b675dd6d670cf3ea8ec0037/train_gpt.py#L162}{line 162 in the reference code}), the ``Nesterov-style'' momentum 
variant (\href{https://github.com/KellerJordan/modded-nanogpt/blob/7d502b97dac8b8116b675dd6d670cf3ea8ec0037/train_gpt.py#L705}{line 705}), and the step size and
momentum schedules used in the speedrun (\href{https://github.com/KellerJordan/modded-nanogpt/blob/7d502b97dac8b8116b675dd6d670cf3ea8ec0037/train_gpt.py#L1519}{lines 1519} 
and 
\href{https://github.com/KellerJordan/modded-nanogpt/blob/7d502b97dac8b8116b675dd6d670cf3ea8ec0037/train_gpt.py#L355}{355}).

\begin{lstlisting}[style=pyalgo, caption={Muon update as used in the modded-nanoGPT speedrun (simplified, per-matrix view)}, label={alg:muon-speedrun}]

def muon_step(W, grad_W, v_W, t, S, T, lr0, weight_decay):
    # schedules
    lr_t = lr0 * speedrun_lr(t, S)
    mu_t = speedrun_momentum(t, T)

    # Nesterov-style momentum
    v_W = mu_t * v_W + (1 - mu_t) * grad_W #v_W is the momentum buffer kept across steps
    m   = mu_t * v_W + (1 - mu_t) * grad_W

    # approximate polar direction
    u = polar_express(m)

    # rectangular correction factor used in the speedrun
    s = sqrt(max(1, W.shape[0] / W.shape[1]))

    # parameter update
    W = W - lr_t * s * u

    # weight decay (simplified, implemented in code via gated / cautious decay)
    W = W - lr_t * weight_decay * W

    return W, v_W
\end{lstlisting}

\begin{lstlisting}[style=pyalgo, caption={Speedrun schedules and Polar Express orthogonalization}, label={alg:muon-aux}]
def speedrun_lr(step, S, cooldown_frac=0.55):
    x = step / S
    cooldown_weight = max(0.0, min(1.0, (1 - x) / cooldown_frac))
    lr_min = 0.1
    lr_max = 1.0 + 0.52*(x > 1/3) + 0.21*(x > 2/3)
    return lr_min + cooldown_weight*(lr_max - lr_min)

def speedrun_momentum(step, T, warmup=300, cooldown=50):
  warmup_frac = max(0.0, min(1.0, step / warmup))
  cooldown_frac = max(0.0, min(1.0, (T - cooldown - step) / cooldown))
  momentum_min = 0.85
  momentum_max = 0.95
  return momentum_min + min(warmup_frac, cooldown_frac) * (momentum_max - momentum_min)

polar_express_coeffs = [
    (8.156554524902461, -22.48329292557795, 15.878769915207462),
    (4.042929935166739, -2.808917465908714, 0.5000178451051316),
    (3.8916678022926607, -2.772484153217685, 0.5060648178503393),
    (3.285753657755655, -2.3681294933425376, 0.46449024233003106),
    (2.3465413258596377, -1.7097828382687081, 0.42323551169305323)
]

def polar_express(G):
    """
    Polar Express approximate orthogonalization (5 iterations, bf16).
    """
    X = G.to(bfloat16)
    if X.shape[0] > X.shape[1]:
        X = X.T

    X = X / (1.02 * norm(X) + 1e-6)

    for (a, b, c) in polar_express_coeffs:
        # compute pol(X) = a * X + b * (X*X^T) * X + c * (X*X^T)^2 * X
        A = X @ X.T
        B = b * A + c * (A @ A)
        X = a * X + B @ X

    if G.shape[0] > G.shape[1]:
        X = X.T
    return X
\end{lstlisting}

\section{Momentum variants for \Muon{}-like methods}
\label{app:momentum}

We mention three basic variants of momentum below and clarify their relation with the original \Muon{} implementation (\Cref{alg:muon-speedrun}).  
We discuss the implications of each variant for the scalar loss 
$\ell(\sv)=\tfrac12\sv^2$ studied in \eqref{eq:1d-sign-setup}: 
\begin{quote}
  Does momentum help escape the lattice confinement
  $\sv_t\in \sv_0+\eta\Z$ induced by exact polar projection and constant step size?
\end{quote}

Answer: \textbf{no}. In the scalar proxy, grid confinement persists for the two
\Muon{}-style variants where momentum is applied \emph{before} projection: 
variants (ii) and (iii) below, the latter being the default in \Muon{}.

While stated in terms of the scalar loss $\ell(\sv)=\tfrac12\sv^2$ for simplicity, with projection $\proj(\sv)=\sign(\sv)$, 
all three variants extend verbatim to the matrix case by 
interpreting $\sv$ as a matrix, and $\proj(\cdot)$ as the exact polar projection operator. 

\paragraph{(i) Momentum \emph{after} projection (Orthogonal-SGDM).}
A first option is to first project, and accumulate momentum on the projected direction
(called ``Orthogonal-SGDM'' in \cite{tuddenham2022orthogonalisinggradientsspeedneural}):
\begin{equation}
  m_t = \mu\, m_{t-1} + (1-\mu)\,\proj(\sv_t),
  \qquad
  \sv_{t+1} = \sv_t - \eta\, m_t.
  \label{eq:orth-sgdm}
\end{equation}
In 1D, this variant does \emph{not} confine every iterate $s_t$ to the lattice $\sv_0+\eta\Z$.  
As explained by \citet{jordan2024muon}, this variant is however \emph{not} the one used in \Muon{}, 
since \citet{tuddenham2022orthogonalisinggradientsspeedneural} found it can underperform 
a well-tuned SGD-momentum baseline. 

\paragraph{(ii) Momentum \emph{before} projection (Muon with standard SGD momentum).}
This is SGD with classical momentum, with the polar projection applied to the momentum buffer:
\begin{equation}
  g_t = \nabla_\sv \ell(\sv_t),
  \qquad
  m_t = \mu\, m_{t-1} + (1-\mu)\,g_t,
  \qquad
  \sv_{t+1} = \sv_t - \eta\,\proj(m_t).
  \label{eq:muon-std-mom}
\end{equation}
In the 1D case $\ell(\sv)=\tfrac12\sv^2$, $\proj(m_t)=\sign(m_t)\in\{\pm 1\}$, so the update step is always $\pm\eta$.
Hence
\begin{equation}
  \sv_{t+1}-\sv_t \in \eta\,\Z
  \quad\Longrightarrow\quad
  \sv_t\in \sv_0+\eta\Z\ \text{for all }t,
  \label{eq:grid-std-mom}
\end{equation}
and the lattice obstruction from the no-momentum case remains. 
This variant is discussed in \cite{jordan2024muon}, but not the default choice in \Muon{}. 

\paragraph{(iii) Nesterov-style momentum \emph{before} projection (Muon default).}
This is the default choice in \Muon{} (\Cref{alg:muon-speedrun}). 
It applies a Nesterov-style momentum before projection:
\begin{equation}
\begin{aligned}
  g_t &= \nabla_\sv \ell(\sv_t) \\
  m_t &= \mu\, m_{t-1} + (1-\mu)\,g_t \\
  \tilde m_t &= \mu\, m_t + (1-\mu)\,g_t \\
  \sv_{t+1} &= \sv_t - \eta\,\proj(\tilde m_t). &&
\end{aligned}
  \label{eq:muon-nesterov}
\end{equation}
Again, in the scalar proxy we have $\proj(\tilde m_t)=\sign(\tilde m_t)\in\{\pm 1\}$, so $\sv_{t+1}-\sv_t\in\{\pm\eta\}$ and therefore $\sv_t\in \sv_0+\eta\Z$ for all $t$.
Thus the same grid confinement mechanism persists. 

\section{Simple Stationarity Result}
\label{app:stationarity}

A standard descent-lemma argument (used throughout the \Muon{} stationarity literature) yields the following generic inequality, which is \Cref{eq:stationarity-constant} in the main text, with $d=\min(d_1,d_2)$:
\begin{equation}
  \left[ \min_{0\le t\le T-1}\|\nabla \loss(W_t)\|_\ast \right]
  \ \le\
  \frac{\loss(W_0)-\loss_\ast}{T\stepsize}
  \;+\;
  \frac{\Lips}{2}\,d\,\stepsize.
  \label{eq-app:stationarity-constant}
\end{equation}

Assume $\loss \colon \R^{d_1\times d_2}\to\R$ is $\Lips$-smooth w.r.t.\ the Frobenius norm, i.e.,
\[
\loss(W') \le \loss(W) + \langle \nabla\loss(W), W'-W\rangle + \frac{\Lips}{2}\|W'-W\|_{\mathrm{F}}^2
\qquad \forall\,W,W'\in\R^{d_1\times d_2}.
\]
Consider the exact Stiefel / polar / spectral-norm LMO update~\eqref{eq:polardef} with step sizes $\stepsize_t>0$:
\begin{equation}
  W_{t+1}=W_t-\stepsize_t\,U_t^{}V_t^\top,
  \qquad
  \nabla \loss(W_t)=U_t\diag(\svnn_{t,1}, \ldots, \svnn_{t,d})V_t^\top,
  \label{eq:lmo-update}
\end{equation}
where $U_t\in\R^{d_1\times d}$ and $V_t\in\R^{d_2\times d}$ have orthonormal columns and $\svnn_{t,i} \geq 0$ are the singular values of the gradient at $W_t$. 

Aiming to plug $W'=W_{t+1}$ and $W=W_t$ in the inequality above, notice $W'-W = -\stepsize_t U_t^{}V_t^\top$,
$\|U_t^{}V_t^\top\|_{\mathrm{F}}^2 = d$
and 
$\langle \nabla \loss(W_t), U_t^{}V_t^\top\rangle = \sum_i \svnn_{t,i} = \|\nabla \loss(W_t)\|_\ast$.
Then, the inequality gives
\begin{equation}
  \loss(W_{t+1}) \le \loss(W_t) - \stepsize_t \|\nabla \loss(W_t)\|_\ast + \frac{\Lips}{2}\,d\,\stepsize_t^2.
  \label{eq:descent-lemma-nuclear}
\end{equation}
Rearrange and sum for $t=0,\dots,T-1$:
\[
\sum_{t=0}^{T-1}\stepsize_t \|\nabla\loss(W_t)\|_\ast
\le \loss(W_0)-\loss(W_T) + \frac{\Lips}{2}\,d\sum_{t=0}^{T-1}\stepsize_t^2
\le \loss(W_0)-\loss_\ast + \frac{\Lips}{2}\,d\sum_{t=0}^{T-1}\stepsize_t^2,
\]
where $\loss_\ast\coloneqq \inf_W \loss(W)$.
Since $\stepsize_t>0$ for all $t$, we find
\[
\min_{0\le t\le T-1}\|\nabla \loss(W_t)\|_\ast
\ \le\
\frac{\loss(W_0)-\loss_\ast}{\sum_{t=0}^{T-1}\stepsize_t}
\;+\;
\frac{\Lips}{2}\,d\,\frac{\sum_{t=0}^{T-1}\stepsize_t^2}{\sum_{t=0}^{T-1}\stepsize_t}.
\]
Taking $\stepsize_t=\stepsize$ constant gives \eqref{eq-app:stationarity-constant}.

\section{Controlled-spectrum quadratic experiments}
\label{app:controlled-spectrum}

\subsection{Problem family and spectrum shapes}
\label{app:controlled-spectrum-setup}

All quadratic objectives in \Cref{sec:setup-conditioning} are derived from
a noiseless linear least-squares problem.
We consider matrix parameters $W\in\R^{d_{\mathrm{in}}\times d_{\mathrm{out}}}$
and define
\begin{equation}
  \loss(W)
  \;=\;
  \frac{1}{2\,n\,d_{\mathrm{out}}}\,\|XW - Y\|_{\mathrm{F}}^2,
  \label{eq:lsq}
\end{equation}
where $X\in\R^{n\times d_{\mathrm{in}}}$ and $Y\in\R^{n\times d_{\mathrm{out}}}$.
Expanding \eqref{eq:lsq} yields the expression
\[
\loss(W)
\;=\;
\tfrac12\langle W, A W\rangle
\;+\;
\langle B, W\rangle
\;+\;
c,
\qquad
A=\tfrac{1}{n\,d_{\mathrm{out}}} X^\top X,
\quad
B=-\tfrac{1}{n\,d_{\mathrm{out}}} X^\top Y,
\quad
c=\tfrac{1}{2n\,d_{\mathrm{out}}}\|Y\|_{\mathrm{F}}^2 .
\]

\paragraph{Construction of $X$ with prescribed spectrum.}
We set $n=d_{\mathrm{in}}=d_{\mathrm{out}}=100$.
To control the spectrum of $A$, we construct $X$ directly from a prescribed
eigenvalue sequence.
For a chosen spectrum family and fixed $s_{\max} > s_{\min} > 0$, we generate then sort eigenvalues
$s_1\ge \cdots \ge s_n \in [s_{\min},s_{\max}]$ while enforcing the endpoints
$s_1=s_{\max}$ and $s_n=s_{\min}$.
We then define singular values
\[
\sigma_i = \sqrt{n\,d_{\mathrm{out}}\, s_i}, \qquad i=1,\dots,n,
\]
and construct
\[
X = U\,\mathrm{diag}(\sigma)\,V^\top,
\]
where $U\in\R^{n\times n}$ and $V\in\R^{n\times n}$ are random orthogonal matrices
obtained as the $Q$ factors of QR factorizations computed by \texttt{torch.linalg.qr}
applied to matrices with i.i.d.\ $\mathcal N(0,1)$ entries. 
By construction,
\[
A = \tfrac{1}{n\,d_{\mathrm{out}}} X^\top X
= V\,\mathrm{diag}(s)\,V^\top ,
\]
so the spectrum of $A$ is exactly $\{s_i\}$.

\paragraph{Targets and optimum.}
We fix a planted ground-truth matrix $W_\star$ and set $Y = X W_\star$.
The problem is noiseless, and since $s_{\min}>0$,
the minimizer is unique and given by $W_\star$,
with $\loss(W_\star)=0$ up to numerical precision.

\paragraph{Spectrum families.}
Across all experiments we fix
$s_{\min}=10^{-3}$ and $s_{\max}=10$,
so that the condition number is $\kappa=10^4$.
We vary only the \emph{shape} of the spectrum (\Cref{fig:eig-distributions}):
\begin{itemize}[leftmargin=*, itemsep=0.25em]
\item \texttt{max\_spiked}: $s_n=s_{\min}$ and $s_1=\cdots=s_{n-1}=s_{\max}$ (a single small eigenvalue, mass at the top).
\item \texttt{min\_spiked}: $s_1=s_{\max}$ and $s_2=\cdots=s_n=s_{\min}$ (a single large eigenvalue, mass at the bottom). 
\item \texttt{uniform}: eigenvalues sampled i.i.d.\ uniformly in $[s_{\min},s_{\max}]$, then sorted and endpoints enforced.
\item \texttt{gaussian}: eigenvalues sampled i.i.d.\ from a truncated normal distribution centered at $(s_{\min}+s_{\max})/2$. For $k:=3$, $\text{mid}:=(s_{\min}+s_{\max})/2$, and $\sigma:=(s_{\max}-s_{\min})/(2k)$, we
sample $z_i\sim\mathcal{N}(0,1)$ i.i.d., clip to $[-k,k]$, set $\tilde s_i = \text{mid} + \sigma z_i$, then sort and enforce endpoints. 
\item \texttt{linear\_decay\_to\_max}: eigenvalues sampled (then sorted) via $\tilde s = s_{\max} - (s_{\max}-s_{\min})\sqrt{u}$ with $u\sim\mathrm{Unif}[0,1]$, yielding higher density near $s_{\min}$. 
\item \texttt{geometric\_decay\_to\_max}: a deterministic geometric progression starting at $s_{\max}$ with ratio $q=0.9$, clipped below at $s_{\min}$: $s_i := \max(s_{\min}, s_{\max} \cdot q^{i-1})$ for $i=1,\dots,n$.
\item \texttt{u\_shaped}: mixture with mass near both endpoints. We use a symmetric Beta$(\alpha,\alpha)$ distribution on $[0,1]$ with $\alpha=0.2$, mapped affinely to $[s_{\min},s_{\max}]$: $\tilde s_i = s_{\min} + (s_{\max}-s_{\min}) z_i$ where $z_i\sim\text{Beta}(0.2,0.2)$.
\end{itemize}

\begin{figure}[t]
  \centering
  \includegraphics[width=0.98\linewidth]{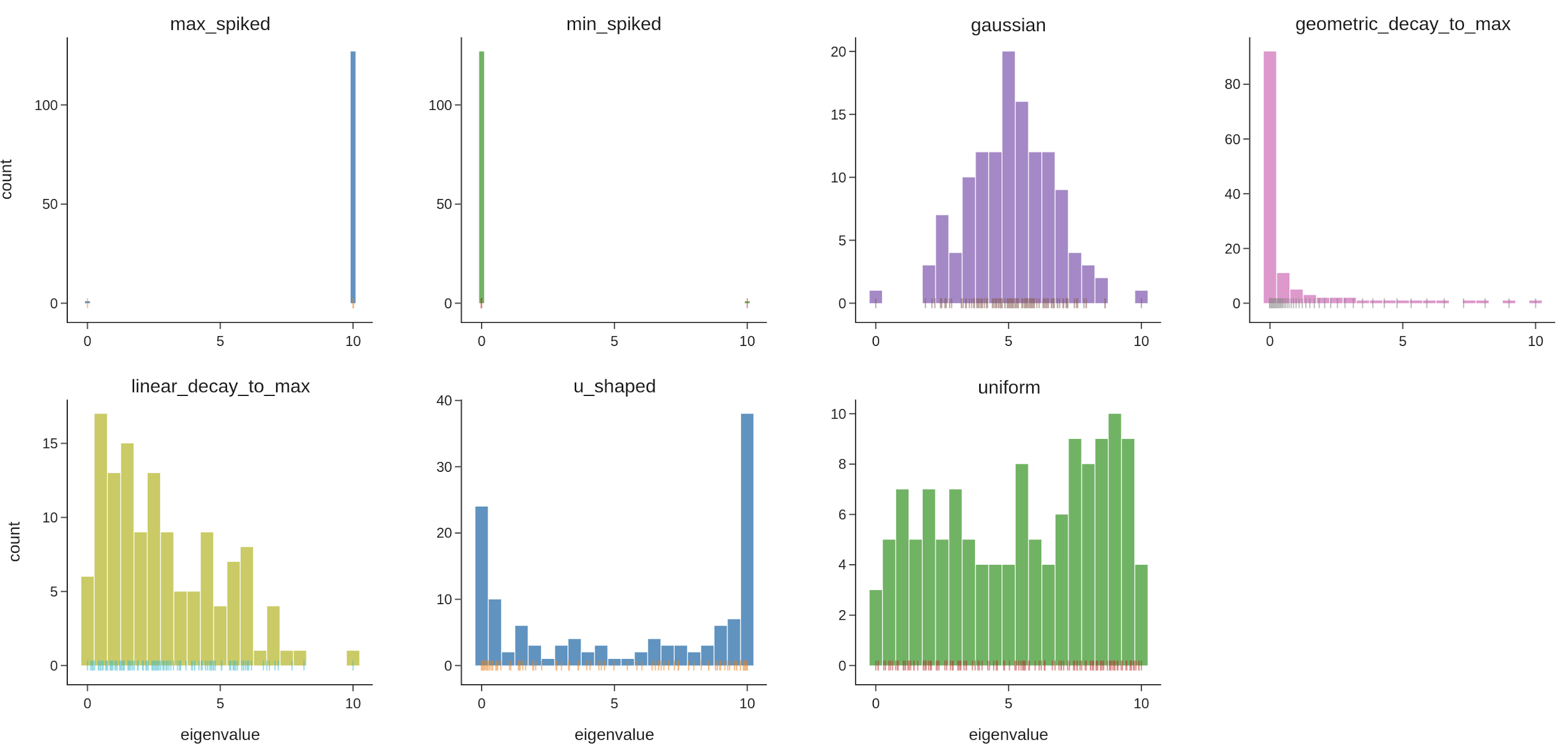}
  \caption{Eigenvalue distributions for the controlled-spectrum families used in \Cref{app:controlled-spectrum}. All families share the same endpoints $(s_{\min},s_{\max})=(10^{-3}, 10^1)$ and thus the same condition number $\kappa=s_{\max}/s_{\min} = 10^4$, but have different spectrum shapes.}
  \label{fig:eig-distributions}
\end{figure}

\subsection{Initialization, horizon, learning-rate selection, and metrics}
\label{app:controlled-spectrum-protocol}

For each fixed quadratic, we sample $100$ i.i.d.\ initializations
$[W_0]_{ij}\overset{\text{i.i.d.}}{\sim}\mathcal N(0,1/n)$. 
We run each method for $T=500$ iterations.

Learning rates are selected from $\{10^{-1},10^{-2},10^{-3}\}$.
When reporting a comparison between two methods,
we select, for each method separately, the learning rate that achieves the smallest loss \emph{up to time $T$}. 
This mirrors a finite-budget ``best-tuned constant step size'' perspective. 

We also experiment with vanishing step sizes, using the \texttt{torch.optim.lr\_scheduler.CosineAnnealingLR} schedule from PyTorch with initial learning rate in $\{10^{-1},10^{-2},10^{-3}\}$ and final learning rate $10^{-3}$ at time $T$.

We report:
(i) \textbf{win rates} (fraction of initializations where optimizer 1 attains lower loss than optimizer 2 at time $t$),
and (ii) \textbf{magnitude summaries} via mean best-loss ratios, defined as
\[
R_t := \frac{\min_{s\leq t }\loss^{\text{optimizer 2}}_s}{\min_{s\leq t}\loss^{\text{optimizer 1}}_s},
\]
reported as $\mathrm{mean}(R_t)\pm 1.96$ standard error over initializations.

\subsection{Methods compared}
\label{app:controlled-spectrum-methods}

In addition to standard \GD{} and Adam, we study several \Muon{} variants: with or without Nesterov-Style momentum, and with exact or Newton--Schulz approximate projection (turning on or off the relevant features in \Cref{alg:muon-aux}).

\subsection{Additional tables: magnitude of improvements}
\label{app:controlled-spectrum-magnitude}

\Cref{tab:exact-muon-vs-gd-ratios} reports mean loss ratios,  complementing the win-rate in \Cref{tab:exact-muon-vs-gd-win}. 
This distinguishes ``barely wins'' from ``wins by a large margin'' at different milestones $t=T/10$, $t=T/2$, and $t=T$, 
instead of just $t=T$ as in the bar visualizations in \Cref{app:controlled-spectrum-bars}. 
We observe that when \Muon{} 
either wins or loses against \GD{}, it often does so by at least an order of magnitude in loss value. 

\begin{table}[H]
\centering
\small
\setlength{\tabcolsep}{3pt}
\renewcommand{\arraystretch}{1.15}
\caption{ Mean $\pm$ 1.96 standard error (over initializations) of best-\GD{}-loss divided by best-\Muon{}-loss (exact projection, no momentum, constant step size). 
Values $>1$ favor \Muon{}. 
This can be read as ``\Muon{} is better by a factor of \texttt{value} on average''.}
\label{tab:exact-muon-vs-gd-ratios}
\begin{tabular}{lccc}
\toprule
kind & $t=T/10$ & $t=T/2$ & $t=T$ \\
\midrule
\texttt{max\_spiked} & $0.00 \pm 0.000000$ & $0.00 \pm 0.000000$ & $0.00 \pm 0.000000$ \\
\texttt{min\_spiked} & $15.0 \pm 6.0$ & $165.0 \pm 28.0$ & $949.0 \pm 620.0$ \\
\texttt{uniform} & $0.10 \pm 0.0010$ & $0.00 \pm 0.000064$ & $0.01 \pm 0.00040$ \\
\texttt{gaussian} & $0.00 \pm 0.000003$ & $0.00 \pm 0.000007$ & $0.00 \pm 0.0005$ \\
\texttt{linear\_decay\_to\_max} & $0.45 \pm 0.0030$ & $0.00 \pm 0.000020$ & $0.02 \pm 0.0007$ \\
\texttt{u\_shaped} & $0.44 \pm 0.003$ & $0.11 \pm 0.003$ & $2.48 \pm 0.018$ \\
\texttt{geometric\_decay\_to\_max} & $7.0 \pm 0.03$ & $10.0 \pm 0.3$ & $78.0 \pm 0.3$ \\
\bottomrule
\end{tabular}
\end{table}

\subsection{Bar visualizations: aligned vs.\ absolute loss levels}
\label{app:controlled-spectrum-bars}

To complement the win rates and loss ratios, we report bar plots summarizing final loss levels after $T=500$ iterations, averaged over random initializations.
We use two visualizations, see \Cref{fig:improvement-bars} and \Cref{fig:improvement-bars-not-aligned} for constant learning rates, 
and \Cref{fig:improvement-bars-vanishing} and \Cref{fig:improvement-bars-vanishing-not-aligned} for vanishing learning rates 
(same qualitative pattern as in the constant step size setting).
The \emph{aligned} version shifts all bars to share a common starting height (the average initial loss are aligned across spectrum families),
so that bar length directly represents the number of orders of magnitude by which the loss is reduced.
The \emph{not-aligned} version plots the average initial and final loss levels on the same log scale, making absolute starting and ending values visible.
These views are complementary: the aligned plot isolates \emph{relative decrease}, while the not-aligned plot reveals the corresponding \emph{absolute} scale.
For numerical readability, for the \texttt{max\_spiked} family we stopped optimizing \GD{} once the loss reached $10^{-5}$, 
even if it keeps improving beyond that point, since the loss has already decreased by at least 8 orders of magnitude and further improvements would 
distort the scale of the plot and make it harder to visually compare the other spectrum families.

\begin{figure}[H]
  \centering
  \includegraphics[width=0.8\linewidth]{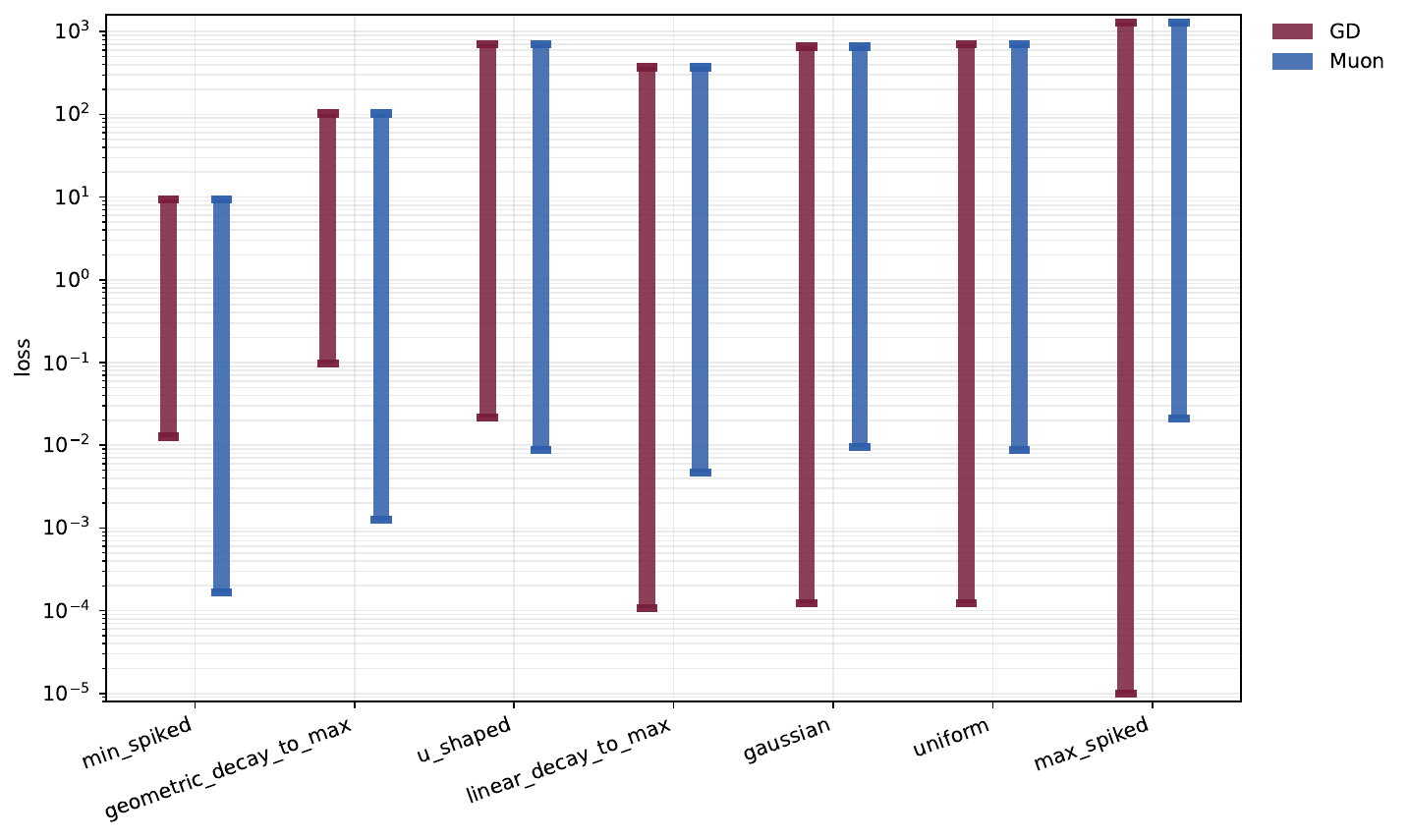}
  \caption{
  Bar plot of \emph{absolute} initial and final loss levels after $T=500$ iterations on the controlled-spectrum quadratic family (constant step size; same setting as \Cref{sec:setup-conditioning}).
  Unlike the aligned view in \Cref{fig:improvement-bars}, bars are not shifted and therefore reflect both the starting loss scale and the final loss achieved.
  For numerical readability, for the \texttt{max\_spiked} family we clip \GD{} losses below $10^{-5}$ when plotting. 
  Spectrum families are ordered from left to right by increasing loss reduction achieved by \GD{} relative to its starting loss (while \Muon{} achieves a comparable order-of-magnitude decrease across families).
  }
  \label{fig:improvement-bars-not-aligned}
\end{figure}

\begin{figure}[H]
  \centering
  \includegraphics[width=0.8\linewidth]{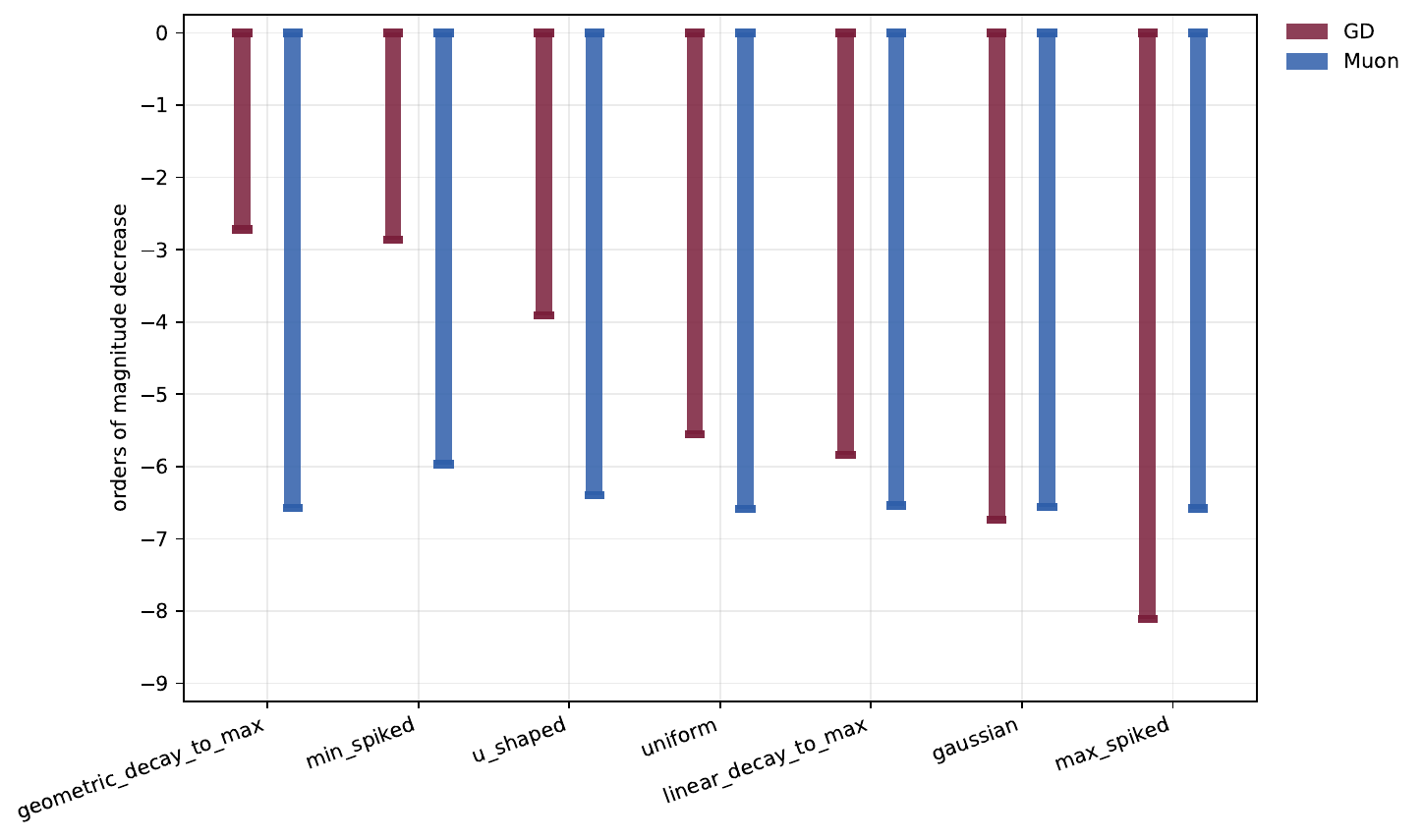}
  \caption{
  Orders of magnitude of loss decrease after $T=500$ iterations under a vanishing learning-rate schedule (CosineAnnealingLR; see \Cref{app:controlled-spectrum-protocol}).
  Bars are aligned at the common initial loss so that their lengths represent the logarithmic decrease achieved.
  The same qualitative pattern as in the constant-step setting persists across spectrum families.
  For numerical readability, for the \texttt{max\_spiked} family we clip \GD{} losses once they reach $10^{-5}$ when plotting. 
    Spectrum families are ordered from left to right by increasing loss reduction achieved by \GD{} relative to its starting loss (while \Muon{} achieves a comparable order-of-magnitude decrease across families).
  }
  \label{fig:improvement-bars-vanishing}
\end{figure}

\begin{figure}[H]
  \centering
  \includegraphics[width=0.8\linewidth]{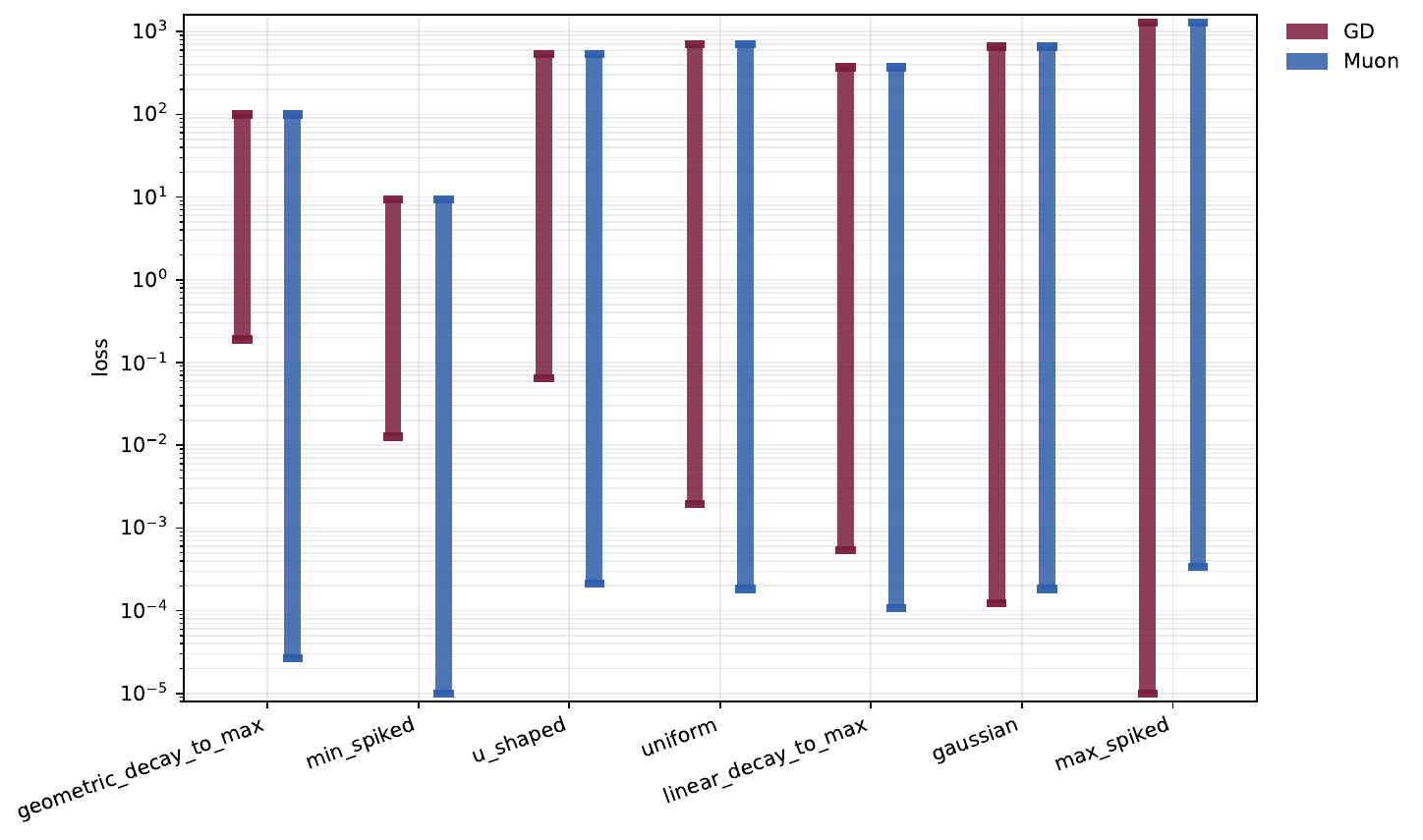}
  \caption{
  Bar plot of absolute initial and final loss levels after $T=500$ iterations under a vanishing learning-rate schedule (same setting as \Cref{fig:improvement-bars-vanishing}).
  Bars are not shifted and therefore reflect both the starting loss scale and the final loss achieved.
  For numerical readability, for the \texttt{max\_spiked} family we clip \GD{} losses below $10^{-5}$ when plotting. 
    Spectrum families are ordered from left to right by increasing loss reduction achieved by \GD{} relative to its starting loss (while \Muon{} achieves a comparable order-of-magnitude decrease across families).
  }
  \label{fig:improvement-bars-vanishing-not-aligned}
\end{figure}

\subsection{Sample trajectories and robustness checks}
\label{app:controlled-spectrum-trajectories}

We report representative sample trajectories (one random initialization $W_0$)
for each spectrum family under two settings:
(i) constant-step optimizers with multiple \Muon{} variants
(with or without Nesterov-style momentum and with exact or Newton--Schulz approximate projection), and
(ii) vanishing step sizes. 
See Figures~\ref{fig:sample-grid-spikes}--\ref{fig:sample-grid-ushape} for representative trajectories across spectrum shapes, shown as paired panels (constant LR vs.\ vanishing LR) for the same underlying quadratic instance.

These representative trajectories support the qualitative picture of
\Cref{sec:setup-conditioning}. 
Across spectrum families, the relative ordering between \GD{} and \Muon{}
variants is stable under standard ablations of \Muon{}'s components
(momentum and approximate projection). 
Using vanishing learning-rate schedules primarily affects late-stage dynamics
by eventually breaking grid confinement for \Muon{}.
However, this typically occurs after \GD{} has already converged,
and it does not reverse the relative performance within the fixed iteration
budget considered here.
\Cref{tab:vanishing-exact-nomom,tab:vanishing-exact-mom,tab:vanishing-inexact-nomom,tab:vanishing-inexact-mom}
compare \GD{} (best across constant learning rate schedules) to several \Muon{} variants with vanishing learning rates,
starting from the exact-projection, no-momentum baseline and successively
turning on Newton--Schulz approximation and momentum.
Across spectrum families, the dominant ranking between \GD{} and \Muon{}
is preserved: spectrum shapes that favor (resp.\ disfavor) \Muon{} under
constant step sizes remain favorable (resp.\ unfavorable) under these
ablations, with only minor quantitative changes in win rates. 
The same pattern holds when comparing variants of \Muon{} with \emph{constant} 
step sizes, against \GD{} (not shown here for brevity). 
Overall, these results reinforce the conclusion that conditioning alone does
not explain the relative performance of \GD{} and \Muon{} variants on these quadratics, and that
finer spectral structure plays a decisive role.

\begin{table}[H]
\centering
\small
\setlength{\tabcolsep}{3pt}
\renewcommand{\arraystretch}{1.15}
\caption{Win rates for \Muon{} (exact projection, no momentum, \textbf{vanishing learning rate})
vs.\ \GD{}, comparing the best loss achieved up to time $t$.
All spectra share the same endpoints $(s_{\min},s_{\max})$ and condition number $\kappa$.}
\label{tab:vanishing-exact-nomom}
\begin{tabular}{lccc}
\toprule
kind & $t=T/10$ & $t=T/2$ & $t=T$ \\
\midrule
\texttt{max\_spiked}              & 0   & 0   & 0   \\
\texttt{min\_spiked}              & 1   & 1   & 1   \\
\texttt{uniform}                & 0   & 0   & 0   \\
\texttt{gaussian}               & 0   & 0   & 0   \\
\texttt{linear\_decay\_to\_max} & 0   & 0   & 0.05\\
\texttt{u\_shaped}       & 0   & 0   & 1   \\
\texttt{geometric\_decay\_to\_max}          & 1   & 1   & 1   \\
\bottomrule
\end{tabular}
\end{table}

\begin{table}[H]
\centering
\small
\setlength{\tabcolsep}{3pt}
\renewcommand{\arraystretch}{1.15}
\caption{Win rates for \Muon{} (exact projection \textbf{with momentum}, \textbf{vanishing learning rate})
vs.\ \GD{}, comparing the best loss achieved up to time $t$.}
\label{tab:vanishing-exact-mom}
\begin{tabular}{lccc}
\toprule
kind & $t=T/10$ & $t=T/2$ & $t=T$ \\
\midrule
\texttt{max\_spiked}              & 0   & 0   & 0   \\
\texttt{min\_spiked}              & 1   & 1   & 1   \\
\texttt{uniform}                & 0   & 0   & 0   \\
\texttt{gaussian}               & 0   & 0   & 0   \\
\texttt{linear\_decay\_to\_max} & 0   & 0   & 0.03\\
\texttt{u\_shaped}       & 0   & 0   & 1   \\
\texttt{geometric\_decay\_to\_max}          & 1   & 1   & 1   \\
\bottomrule
\end{tabular}
\end{table}

\begin{table}[H]
\centering
\small
\setlength{\tabcolsep}{3pt}
\renewcommand{\arraystretch}{1.15}
\caption{Win rates for \Muon{} (\textbf{Newton--Schulz projection}, no momentum, \textbf{vanishing learning rate})
vs.\ \GD{}, comparing the best loss achieved up to time $t$.}
\label{tab:vanishing-inexact-nomom}
\begin{tabular}{lccc}
\toprule
kind & $t=T/10$ & $t=T/2$ & $t=T$ \\
\midrule
\texttt{max\_spiked}              & 0 & 0 & 0 \\
\texttt{min\_spiked}              & 1 & 1 & 1 \\
\texttt{uniform}                & 0 & 0 & 0 \\
\texttt{gaussian}               & 0 & 0 & 0 \\
\texttt{linear\_decay\_to\_max} & 0 & 0 & 0 \\
\texttt{u\_shaped}       & 0 & 0 & 1 \\
\texttt{geometric\_decay\_to\_max}          & 1 & 1 & 1 \\
\bottomrule
\end{tabular}
\end{table}

\begin{table}[H]
\centering
\small
\setlength{\tabcolsep}{3pt}
\renewcommand{\arraystretch}{1.15}
\caption{Win rates for \Muon{} (\textbf{Newton--Schulz projection} \textbf{with momentum}, \textbf{vanishing learning rate})
vs.\ \GD{}, comparing the best loss achieved up to time $t$.}
\label{tab:vanishing-inexact-mom}
\begin{tabular}{lccc}
\toprule
kind & $t=T/10$ & $t=T/2$ & $t=T$ \\
\midrule
\texttt{max\_spiked}              & 0 & 0 & 0 \\
\texttt{min\_spiked}              & 1 & 1 & 1 \\
\texttt{uniform}                & 0 & 0 & 0 \\
\texttt{gaussian}               & 0 & 0 & 0 \\
\texttt{linear\_decay\_to\_max} & 0 & 0 & 0.31\\
\texttt{u\_shaped}       & 0 & 0 & 1 \\
\texttt{geometric\_decay\_to\_max}          & 1 & 1 & 1 \\
\bottomrule
\end{tabular}
\end{table}

\begin{figure}[H]
\centering

\begin{subfigure}[t]{0.49\linewidth}
  \centering
  \includegraphics[width=\linewidth]{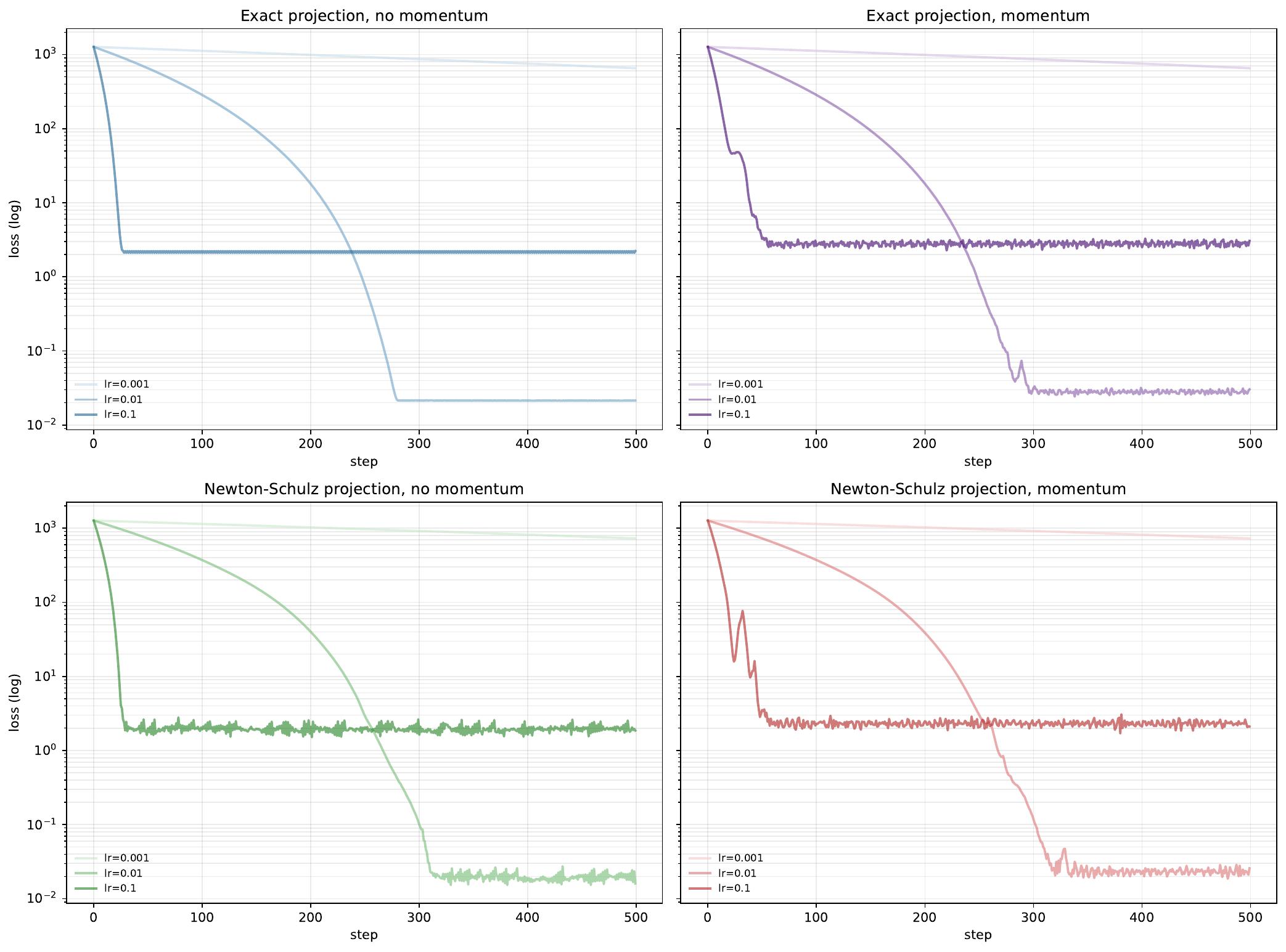}
  \caption{\texttt{max\_spiked} (const LR)}
\end{subfigure}\hfill
\begin{subfigure}[t]{0.49\linewidth}
  \centering
  \includegraphics[width=\linewidth]{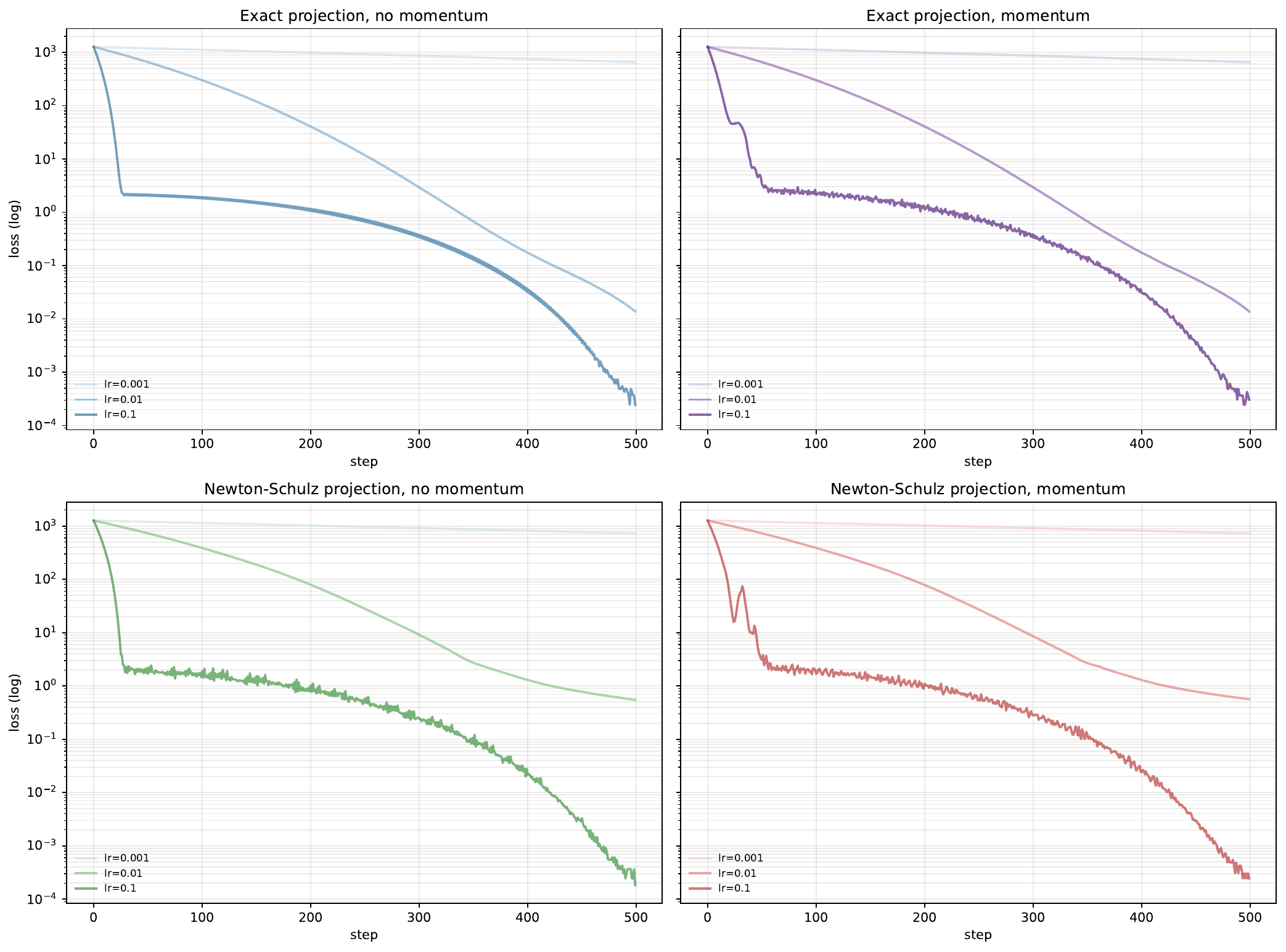}
  \caption{\texttt{max\_spiked} (vanish LR)}
\end{subfigure}

\vspace{0.6em}

\begin{subfigure}[t]{0.49\linewidth}
  \centering
  \includegraphics[width=\linewidth]{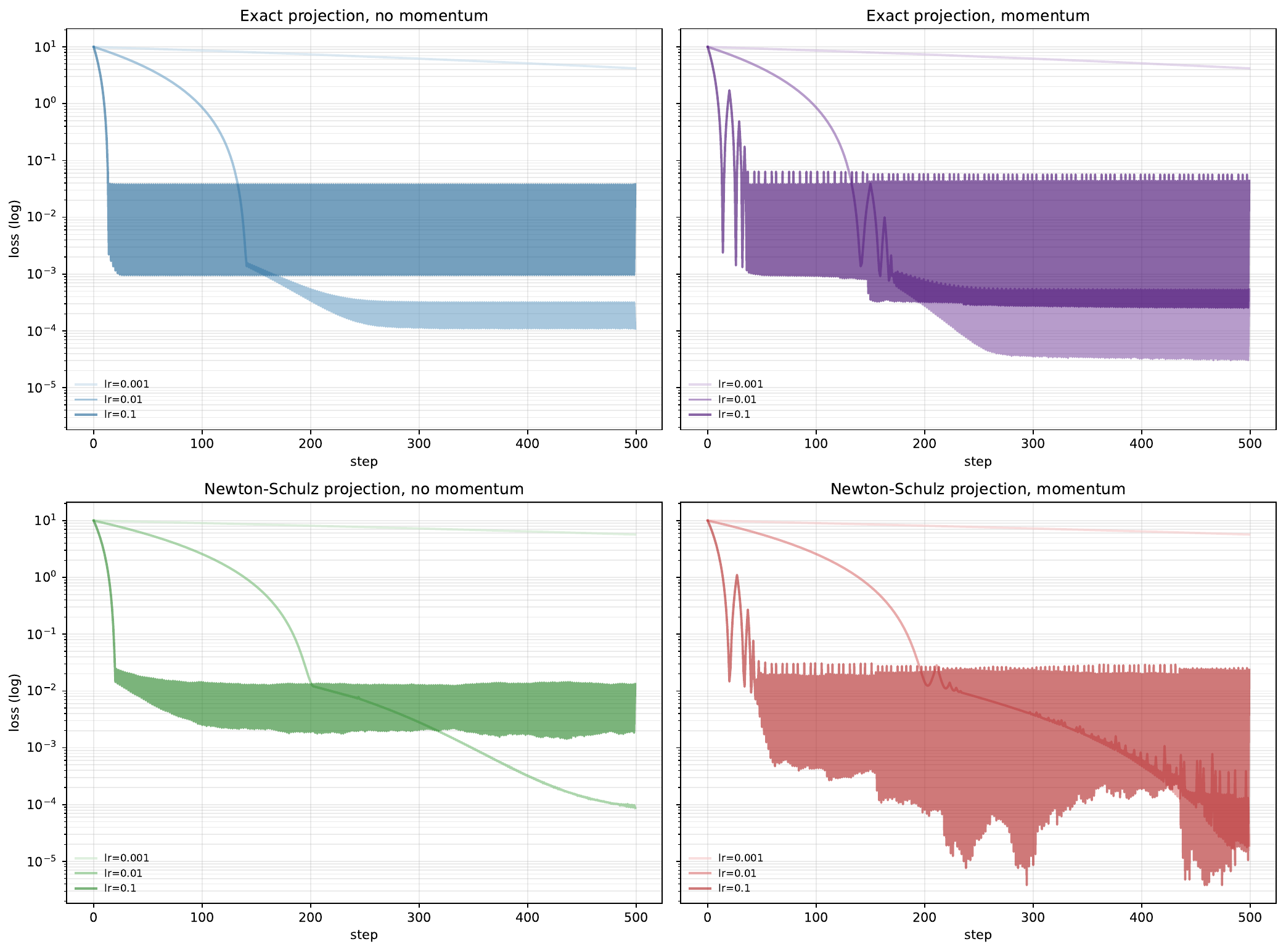}
  \caption{\texttt{min\_spiked} (const LR)}
\end{subfigure}\hfill
\begin{subfigure}[t]{0.49\linewidth}
  \centering
  \includegraphics[width=\linewidth]{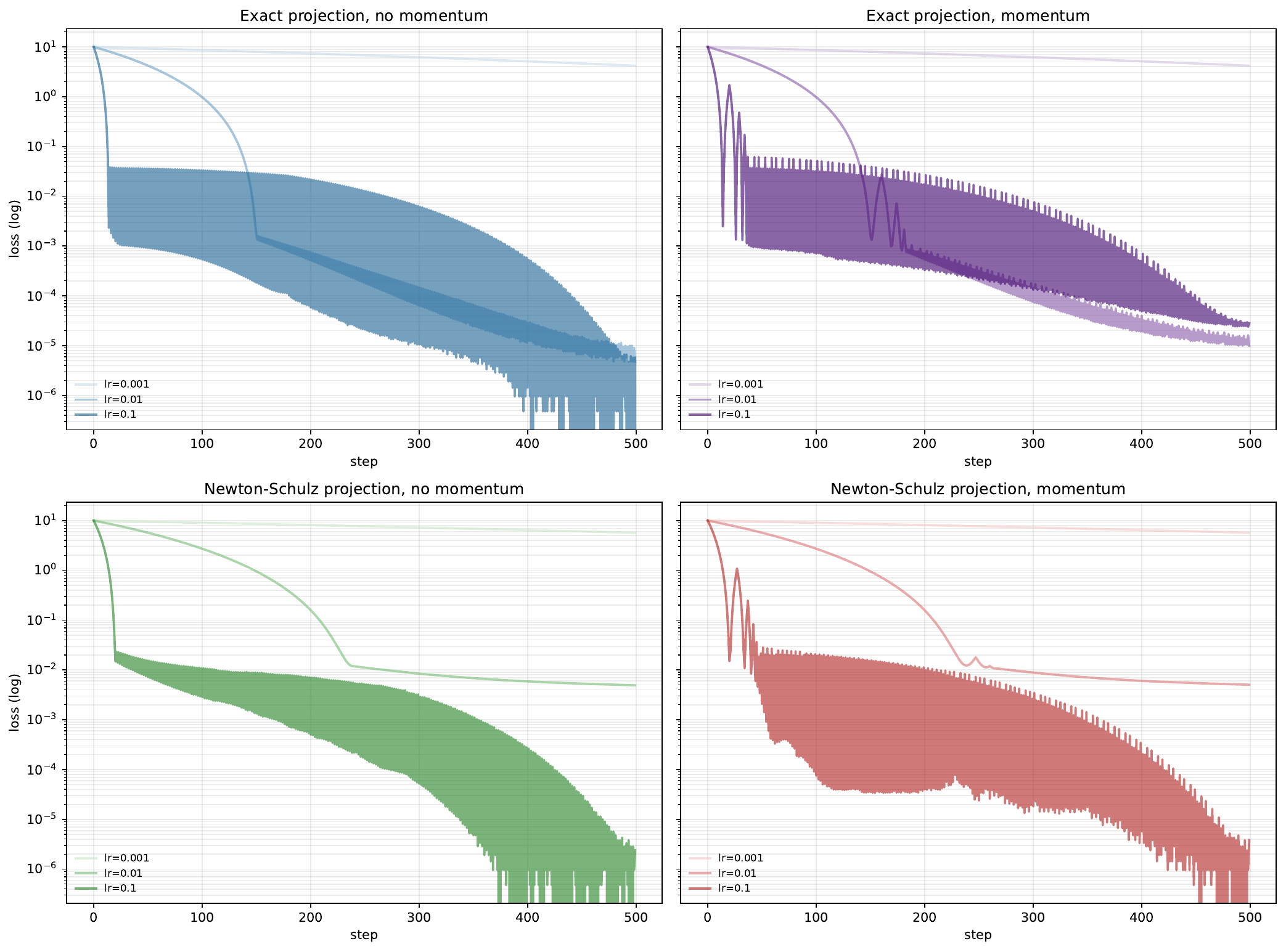}
  \caption{\texttt{min\_spiked} (vanish LR)}
\end{subfigure}

\caption{Sample trajectories (one random $W_0$) comparing \Muon{} variants. Left: constant learning rate (LR). Right: vanishing LR.}
\label{fig:sample-grid-spikes}
\end{figure}

\begin{figure}[H]
\centering

\begin{subfigure}[t]{0.49\linewidth}
  \centering
  \includegraphics[width=\linewidth]{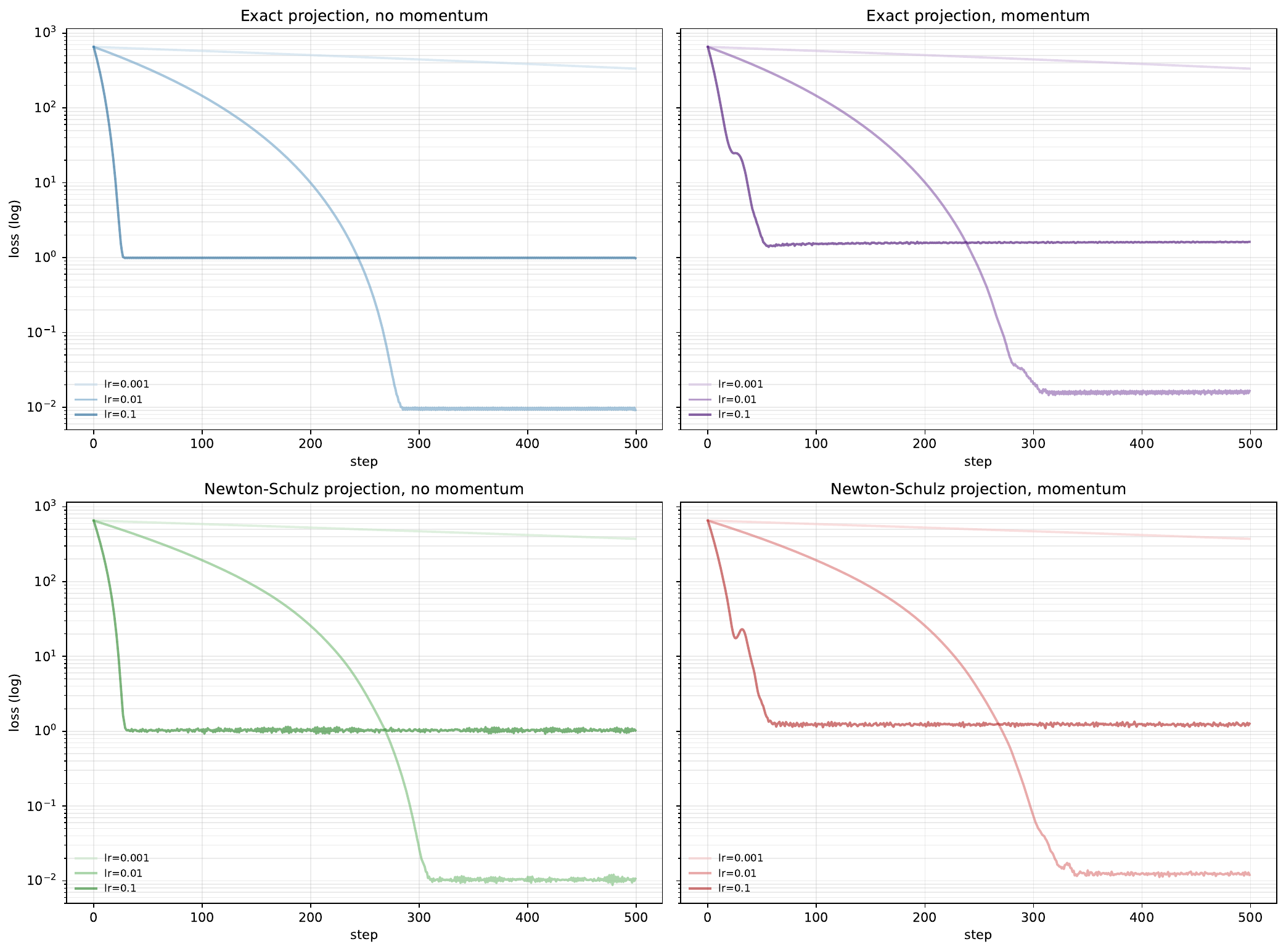}
  \caption{\texttt{gaussian} (const LR)}
\end{subfigure}\hfill
\begin{subfigure}[t]{0.49\linewidth}
  \centering
  \includegraphics[width=\linewidth]{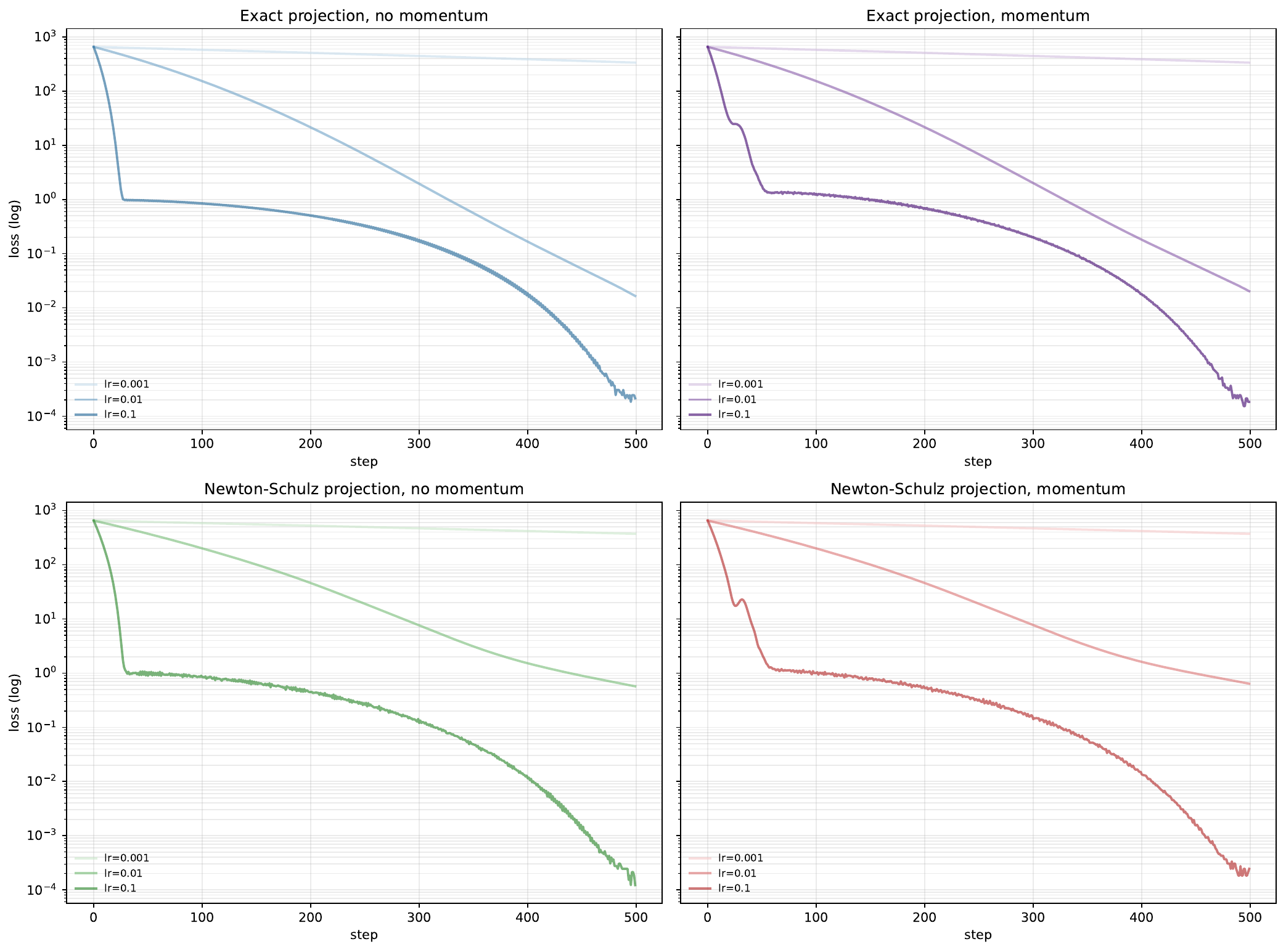}
  \caption{\texttt{gaussian} (vanish LR)}
\end{subfigure}

\vspace{0.6em}

\begin{subfigure}[t]{0.49\linewidth}
  \centering
  \includegraphics[width=\linewidth]{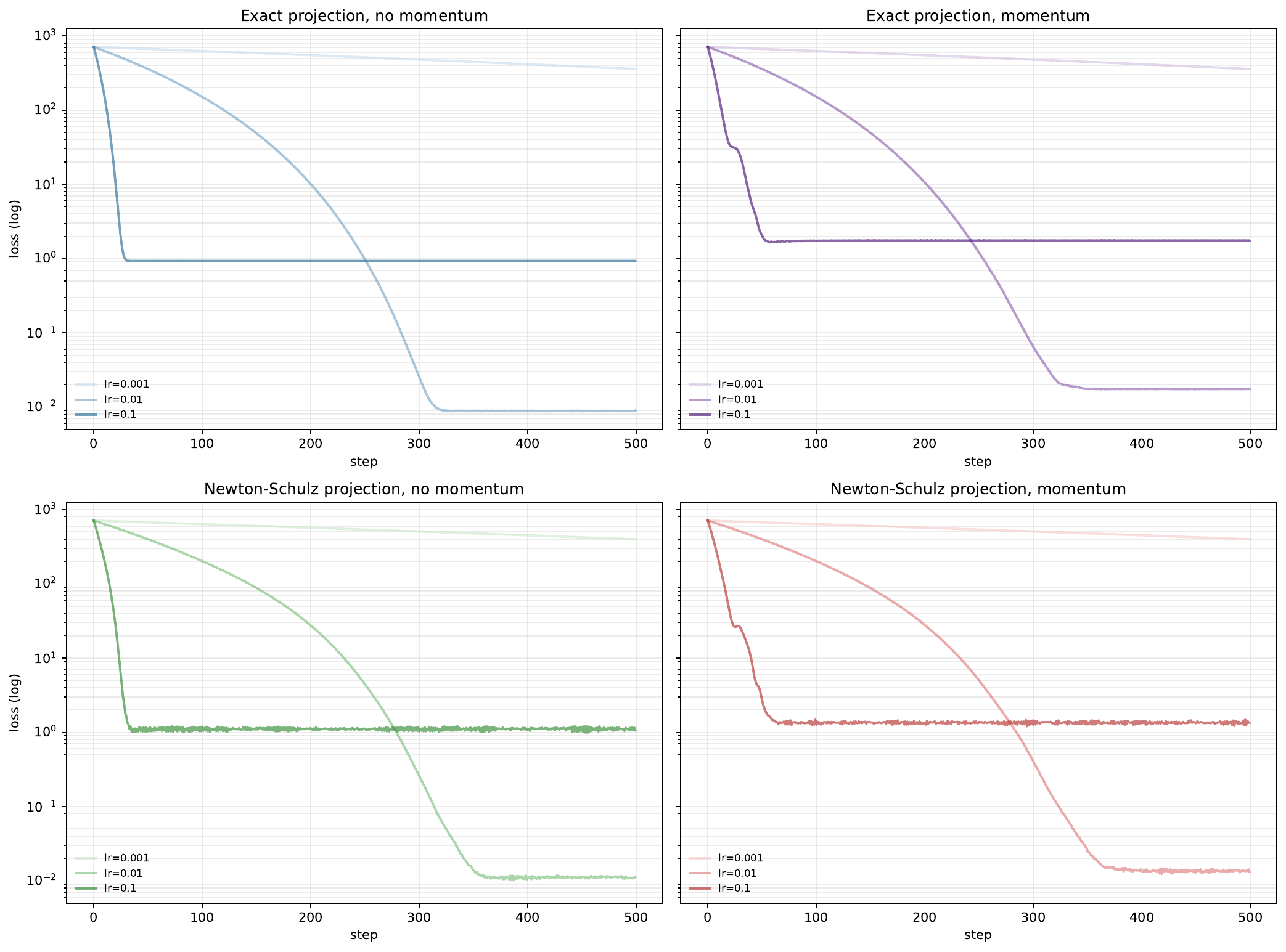}
  \caption{\texttt{uniform} (const LR)}
\end{subfigure}\hfill
\begin{subfigure}[t]{0.49\linewidth}
  \centering
  \includegraphics[width=\linewidth]{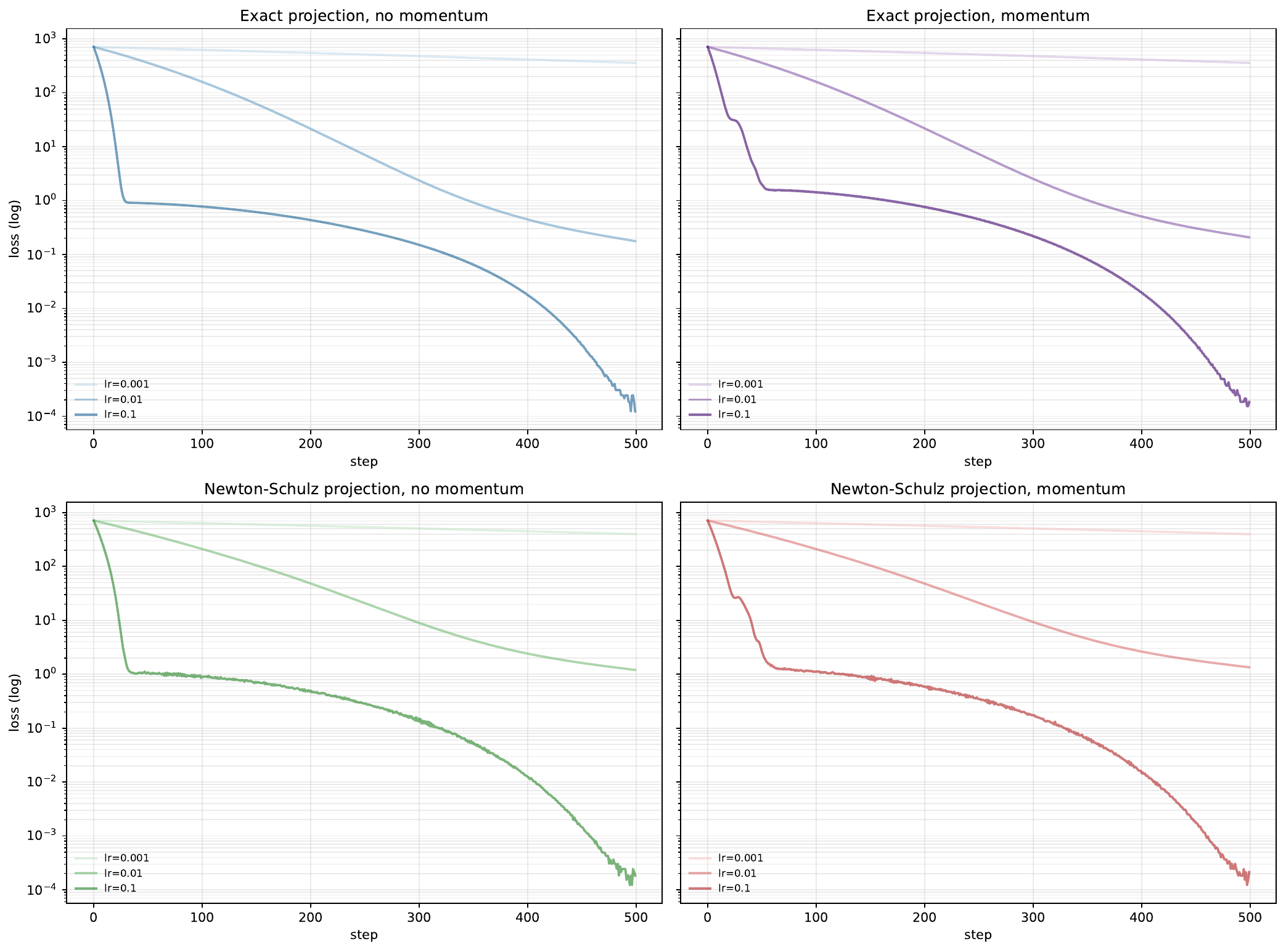}
  \caption{\texttt{uniform} (vanish LR)}
\end{subfigure}

\caption{Sample trajectories for \texttt{gaussian} and \texttt{uniform} spectra (same setting as \Cref{fig:sample-grid-spikes}).}
\label{fig:sample-grid-gauss-unif}
\end{figure}

\begin{figure}[H]
\centering

\begin{subfigure}[t]{0.49\linewidth}
  \centering
  \includegraphics[width=\linewidth]{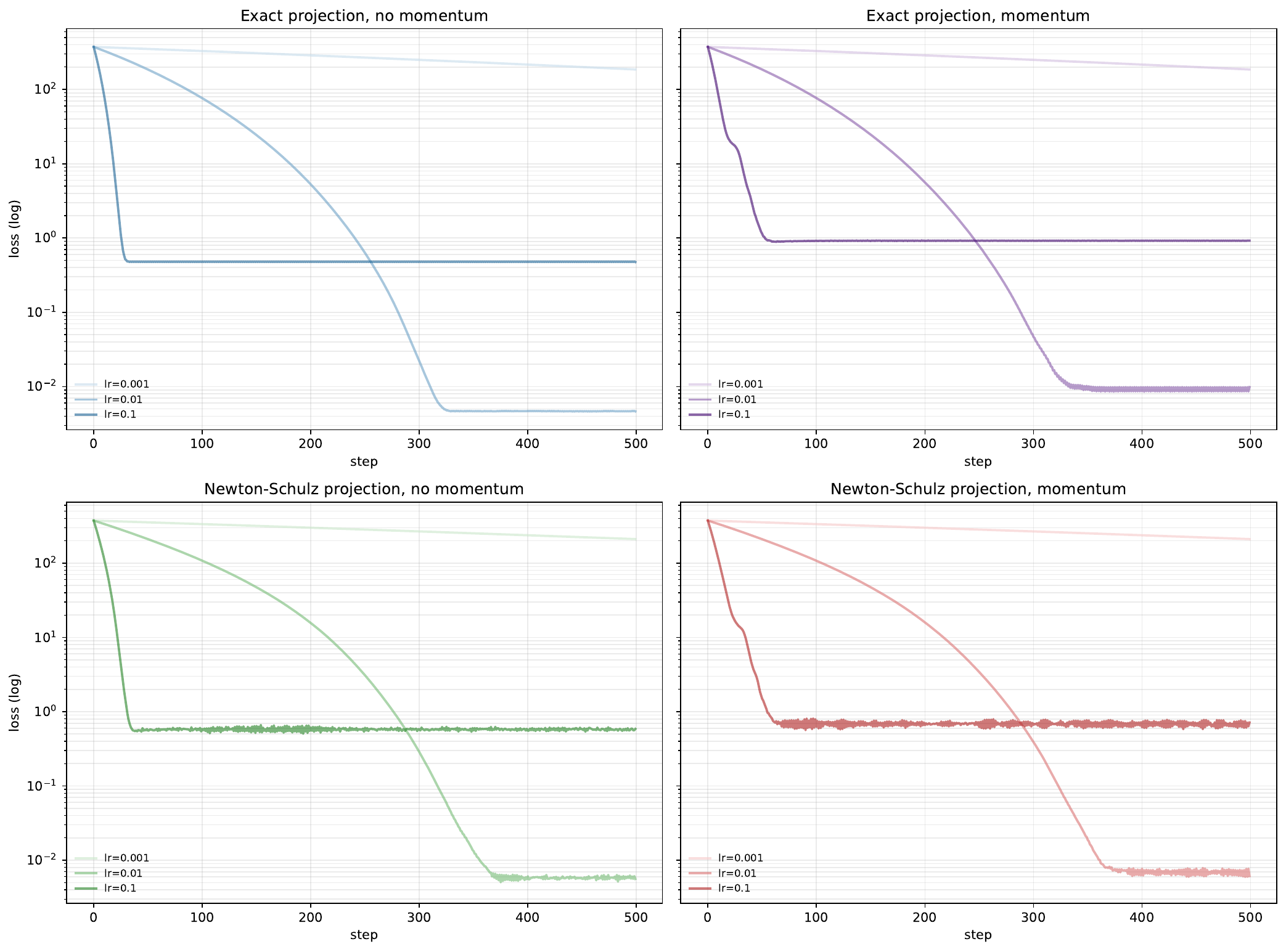}
  \caption{\texttt{linear\_decay\_to\_max} (const LR)}
\end{subfigure}\hfill
\begin{subfigure}[t]{0.49\linewidth}
  \centering
  \includegraphics[width=\linewidth]{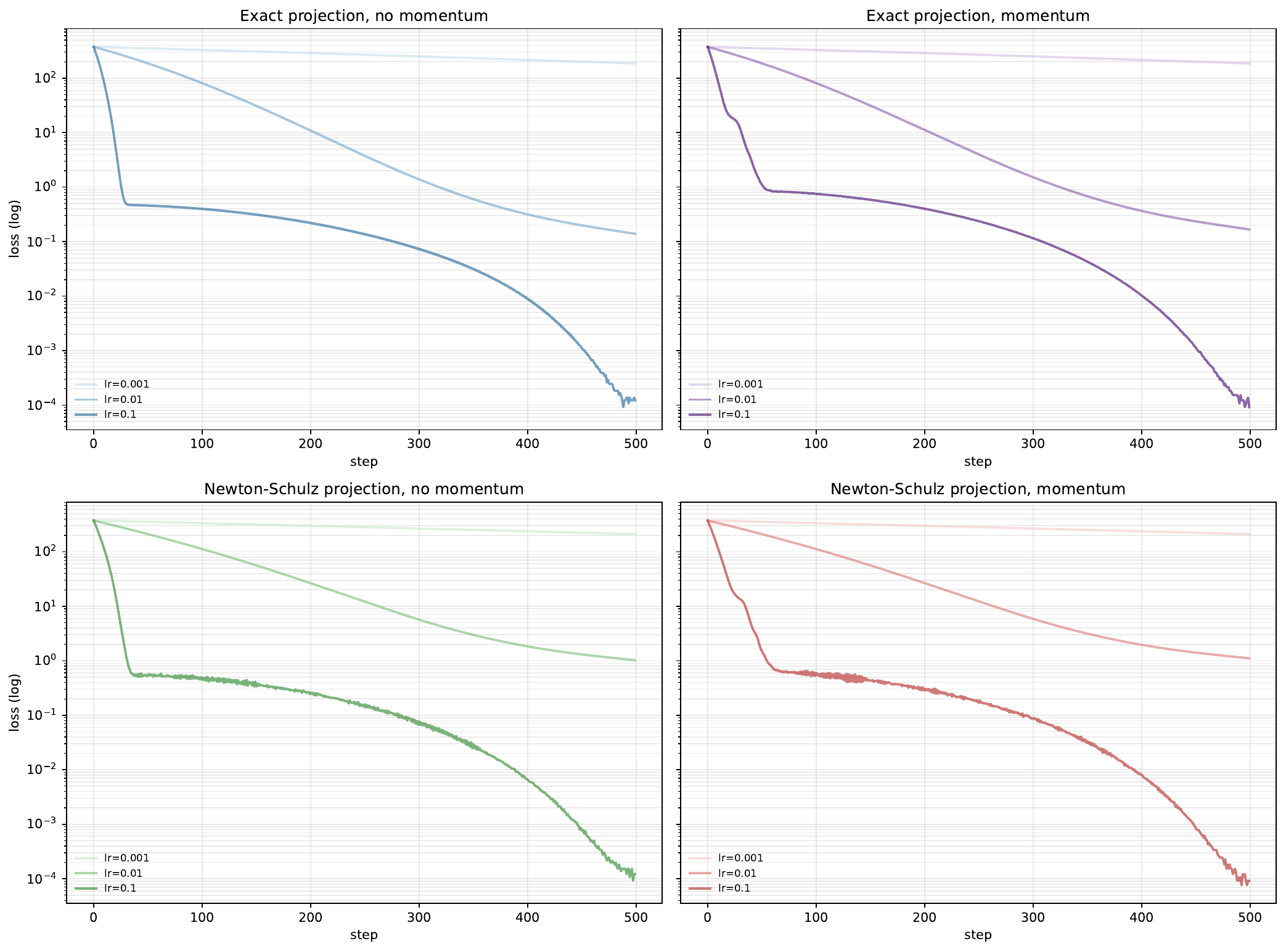}
  \caption{\texttt{linear\_decay\_to\_max} (vanish LR)}
\end{subfigure}

\vspace{0.6em}

\begin{subfigure}[t]{0.49\linewidth}
  \centering
  \includegraphics[width=\linewidth]{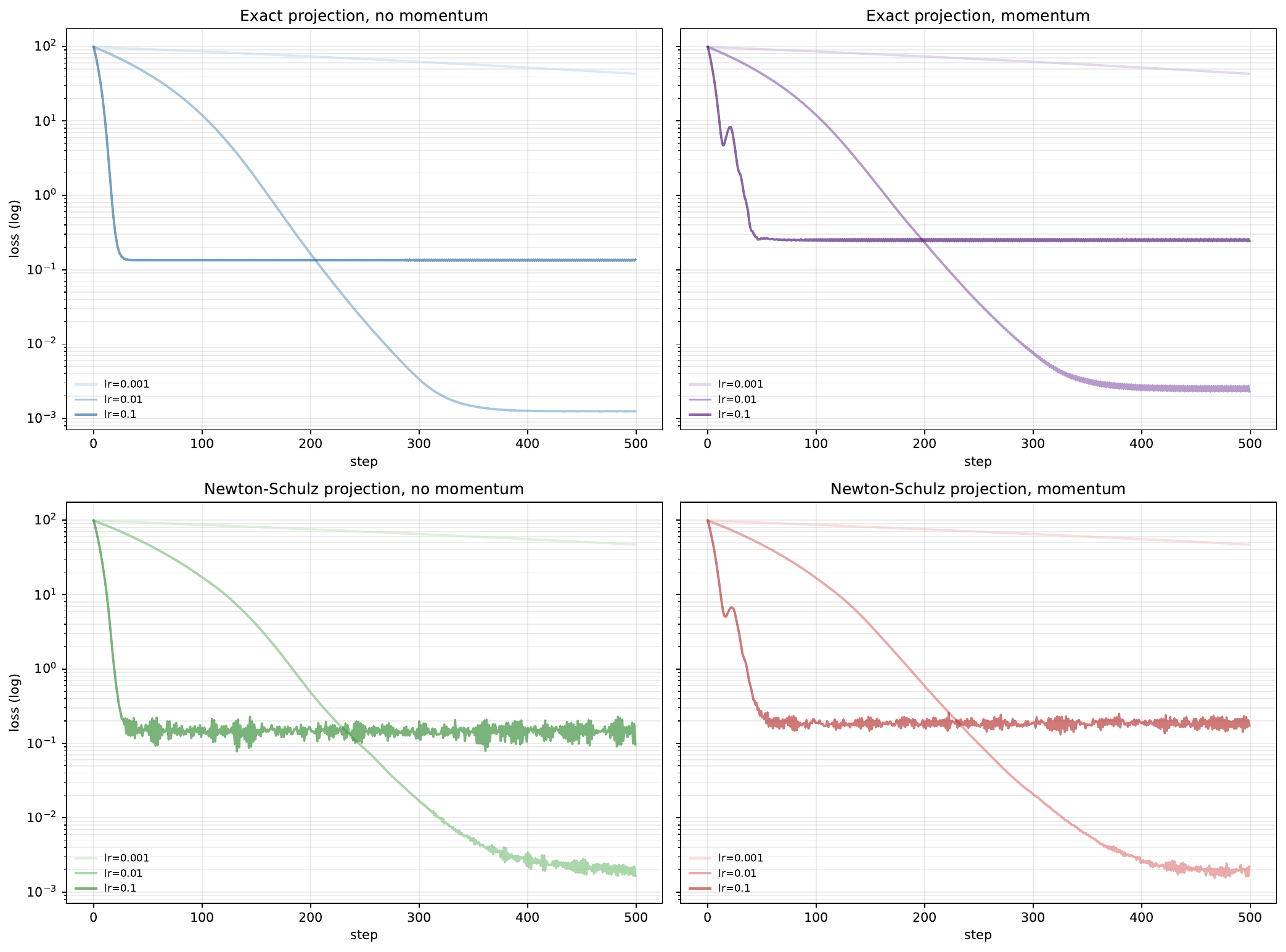}
  \caption{\texttt{geometric\_decay\_to\_max} (const LR)}
\end{subfigure}\hfill
\begin{subfigure}[t]{0.49\linewidth}
  \centering
  \includegraphics[width=\linewidth]{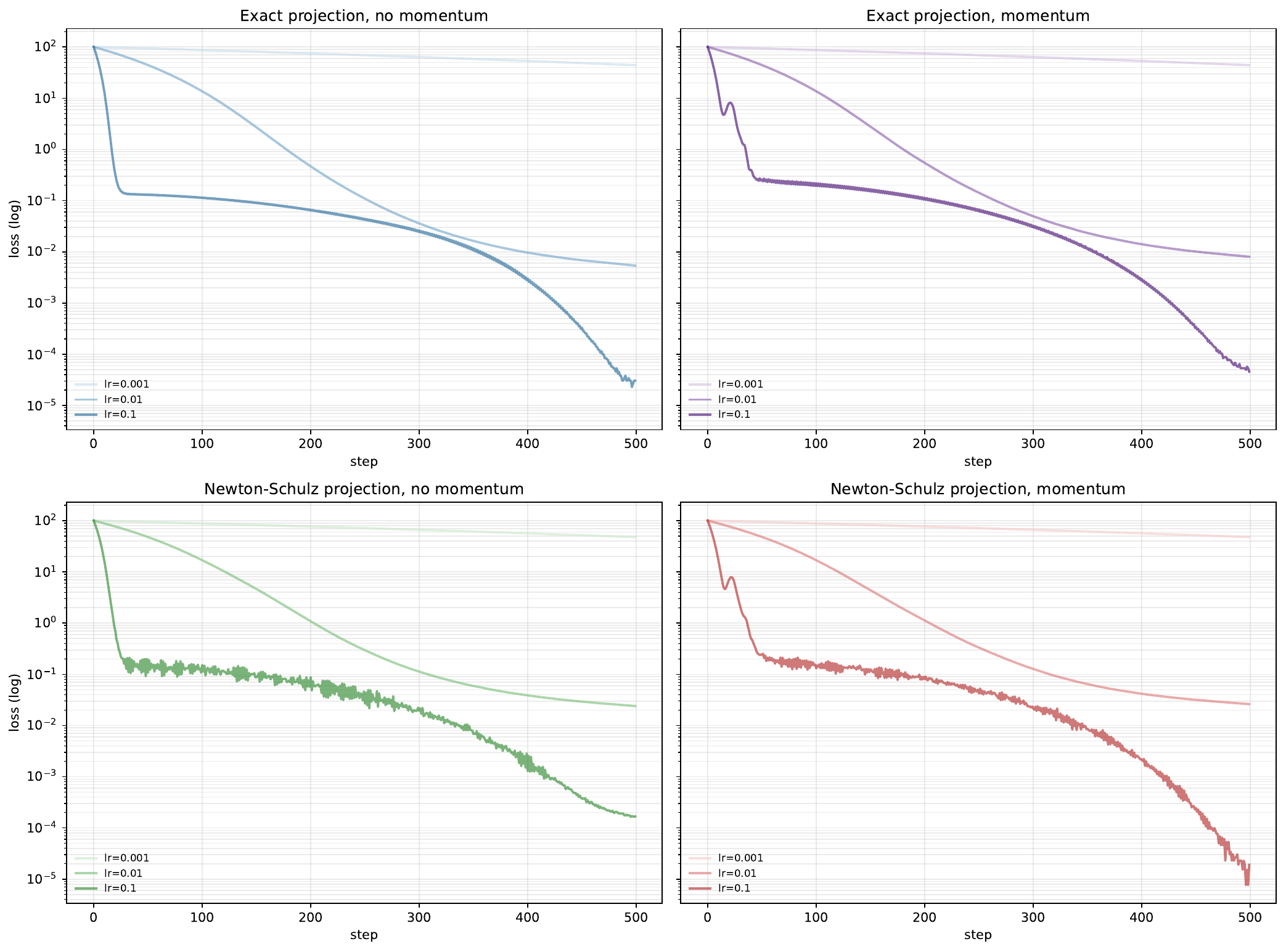}
  \caption{\texttt{geometric\_decay\_to\_max} (vanish LR)}
\end{subfigure}

\caption{Sample trajectories for decaying-spectrum families (same setting as \Cref{fig:sample-grid-spikes}).}
\label{fig:sample-grid-decays}
\end{figure}

\begin{figure}[H]
\centering

\begin{subfigure}[t]{0.49\linewidth}
  \centering
  \includegraphics[width=\linewidth]{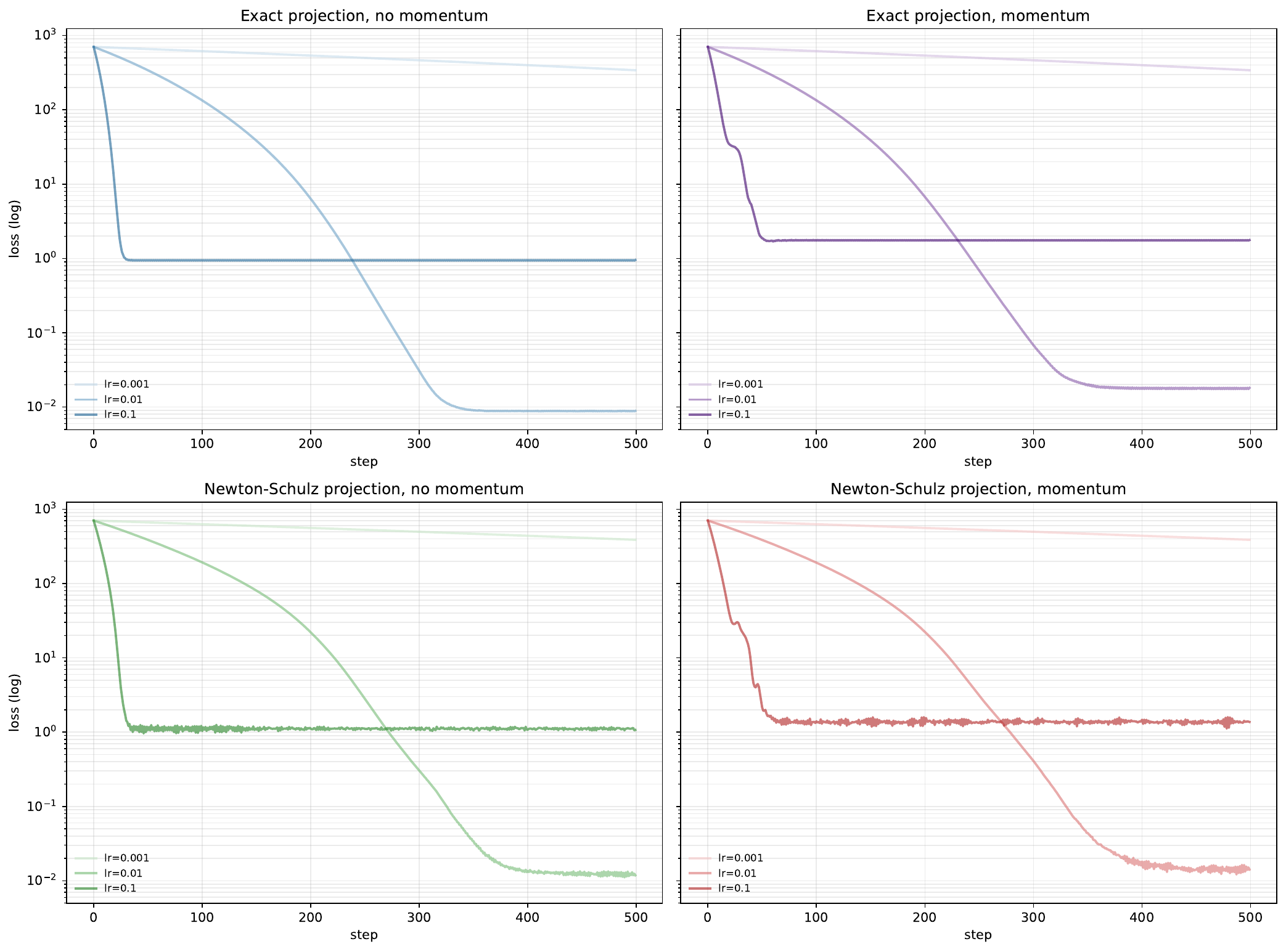}
  \caption{\texttt{u\_shaped} (const LR)}
\end{subfigure}\hfill
\begin{subfigure}[t]{0.49\linewidth}
  \centering
  \includegraphics[width=\linewidth]{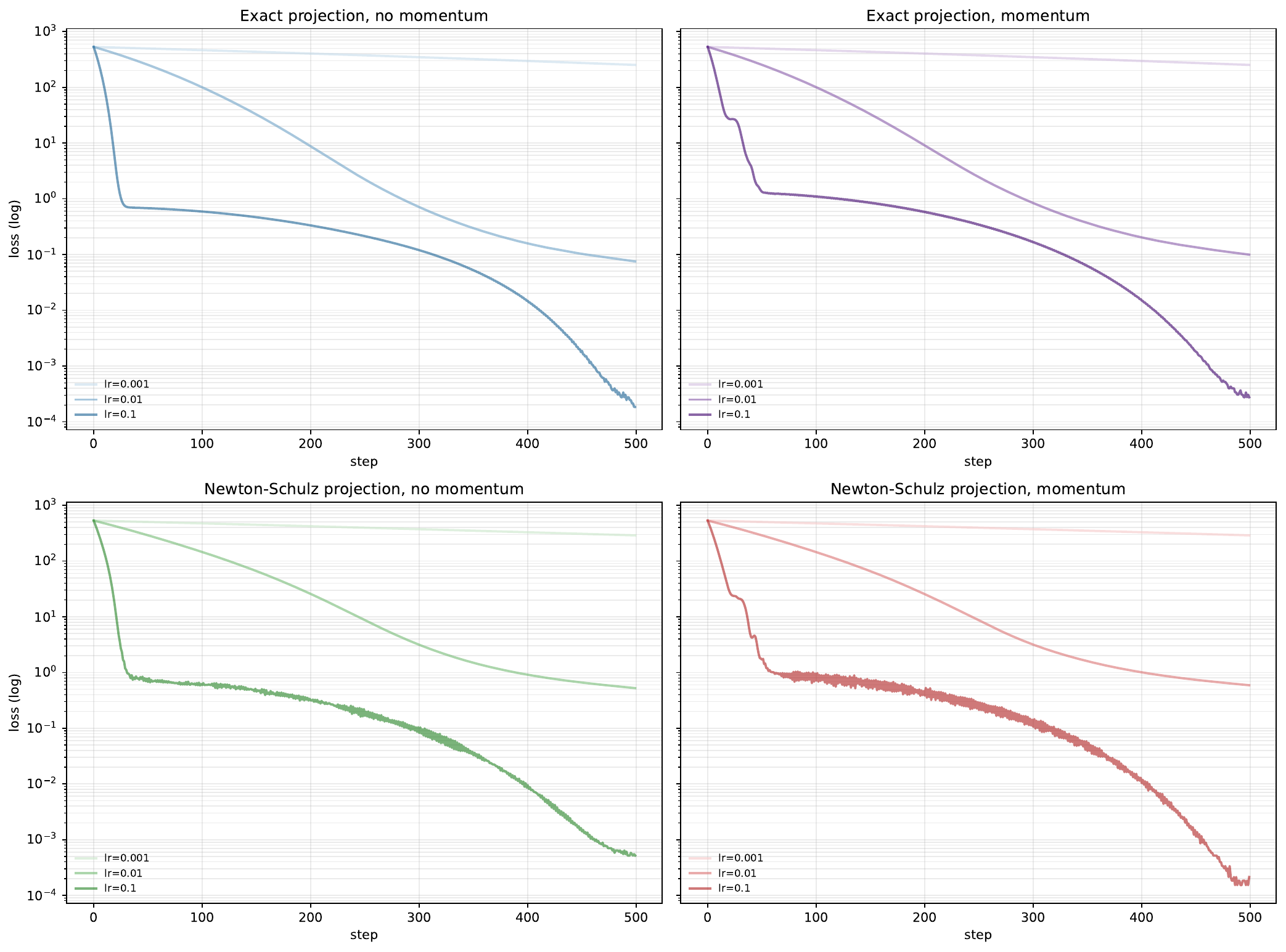}
  \caption{\texttt{u\_shaped} (vanish LR)}
\end{subfigure}

\caption{Sample trajectories for the \texttt{u\_shaped} spectrum (same setting as \Cref{fig:sample-grid-spikes}).}
\label{fig:sample-grid-ushape}
\end{figure}

\subsection{Gradient-conditioning diagnostics}
\label{app:grad-conditioning}

Some practitioner discussions refer to ``conditioning of the gradients along training'' rather than Hessian conditioning.
We report the averaged gradient condition numbers in our controlled-spectrum experiments in Figures~\ref{fig:grad-cond-spikes-ushape}--\ref{fig:grad-cond-geom-unif-linear-gaussian}, together with the averaged loss trajectories. 
The qualitative conclusion matches \Cref{sec:setup-conditioning}. 
Across spectrum families, similar gradient-conditioning profiles 
along classical \GD{} 
can correspond to qualitatively different benefits when switching to \Muon{}.
For instance, the condition number profiles for \texttt{geometric\_decay\_to\_max} and \texttt{uniform} spectra 
are 
similarly decreasing from $10^8$ to $10^7$ (first two rows in \Cref{fig:grad-cond-geom-unif-linear-gaussian}), 
yet switching to \Muon{} is only beneficial for the \texttt{geometric\_decay\_to\_max} case (\Cref{tab:exact-muon-vs-gd-win}). 
This suggests that the gradient conditioning measure alone does not explain the relative speed of \GD{} and \Muon{} on these quadratics.

\begin{figure}[H]
\centering

\includegraphics[width=0.95\linewidth]{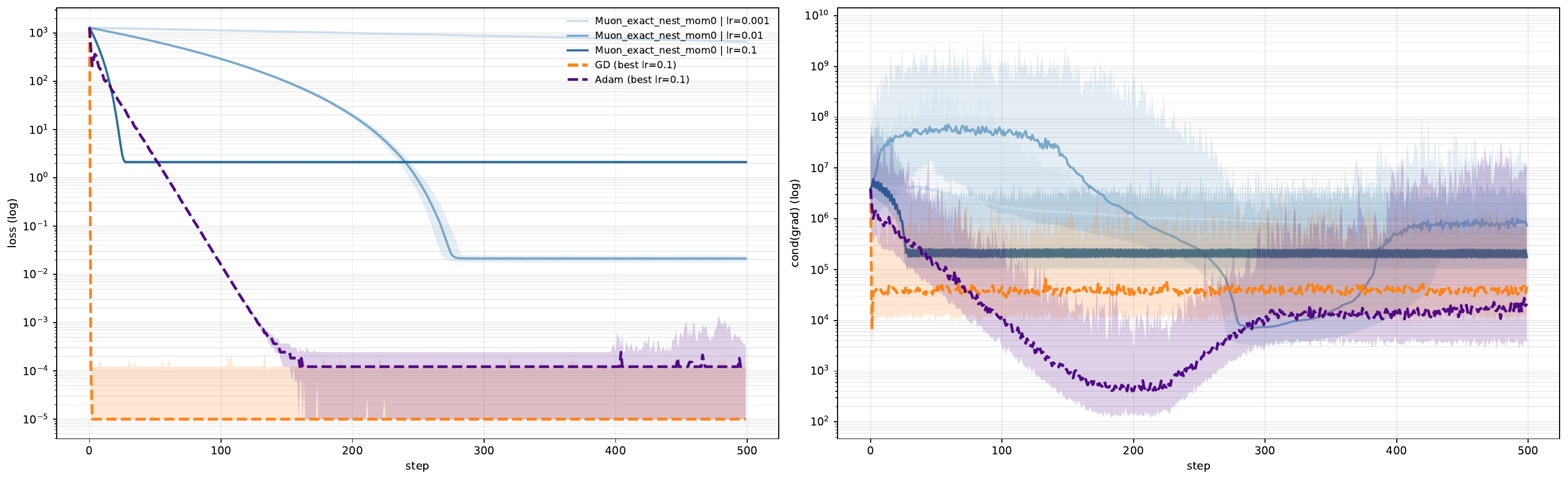}

\vspace{0.8em}

\includegraphics[width=0.95\linewidth]{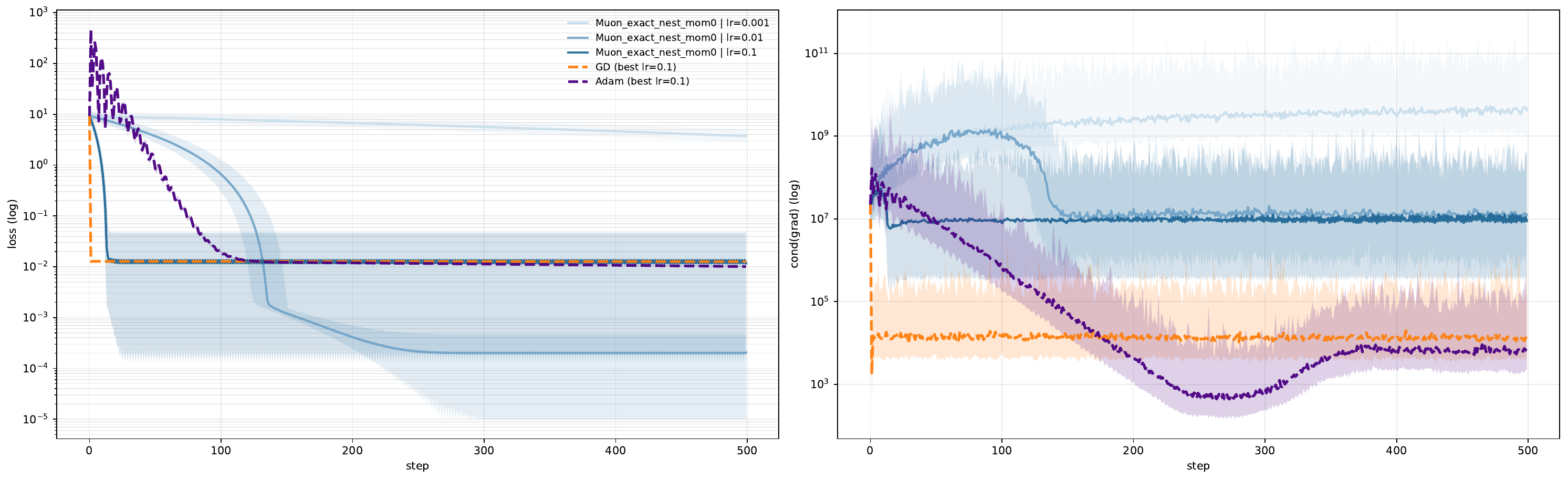}

\vspace{0.8em}

\includegraphics[width=0.95\linewidth]{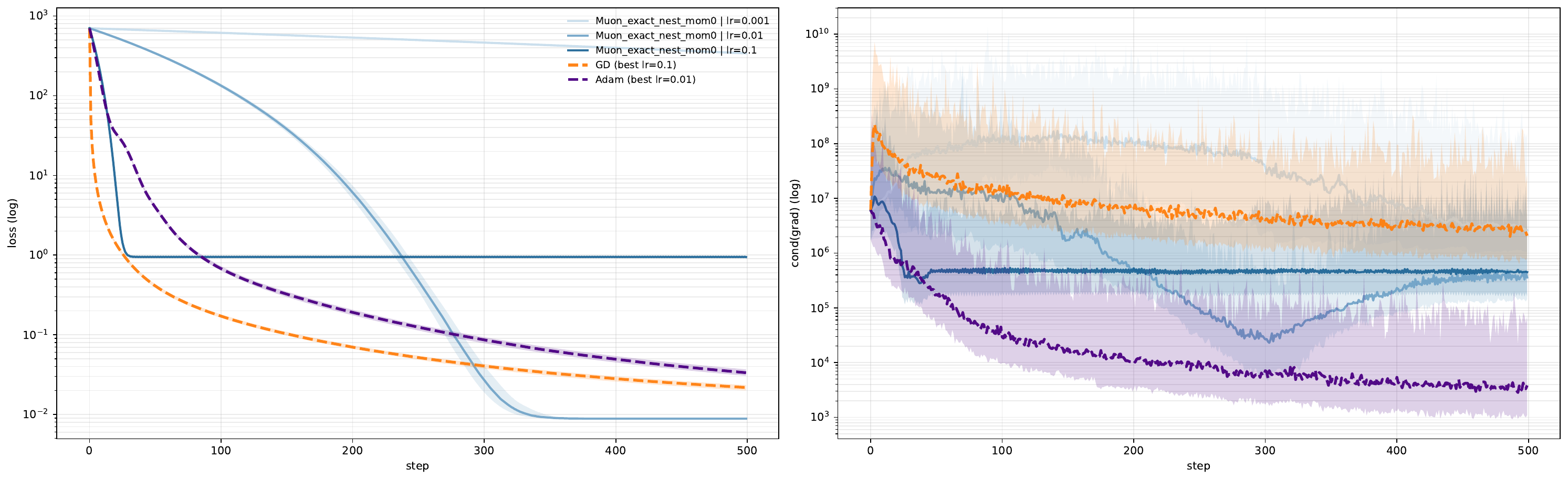}

\caption{
Averaged trajectories for \texttt{max\_spiked} (top), \texttt{min\_spiked} (middle), and \texttt{u\_shaped} (bottom) spectra.
Each panel shows the loss curves (left) and the gradient conditioning measure along the trajectory (right), averaged over random initializations.
}
\label{fig:grad-cond-spikes-ushape}
\end{figure}

\begin{figure}[H]
\centering

\includegraphics[width=0.95\linewidth]{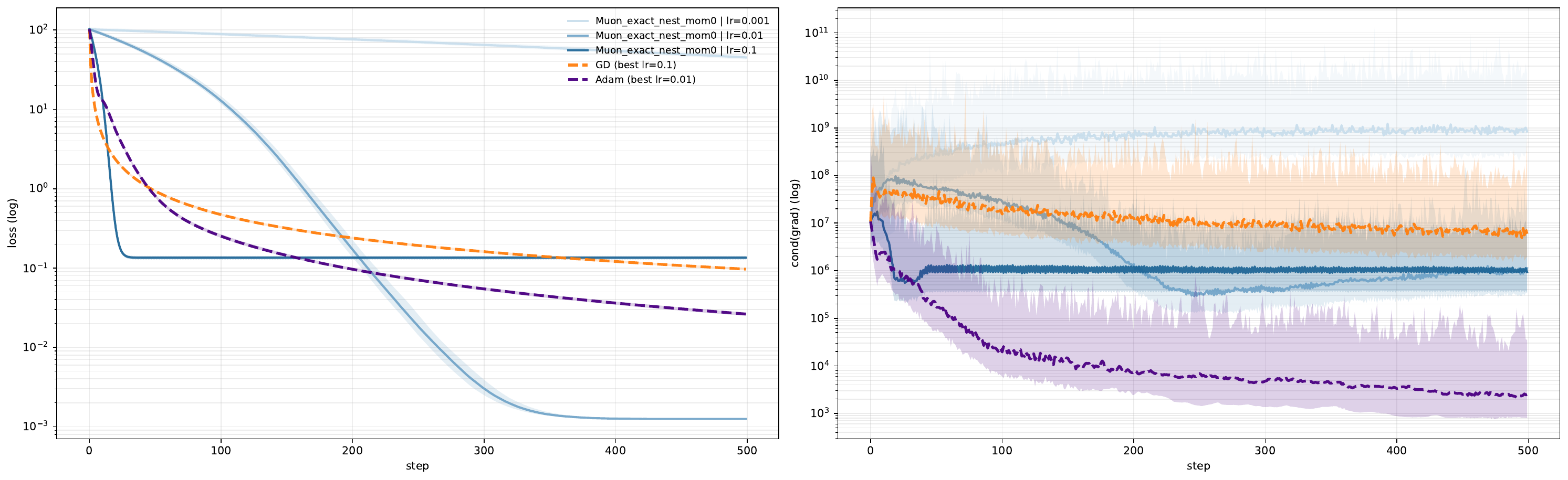}

\vspace{0.8em}

\includegraphics[width=0.95\linewidth]{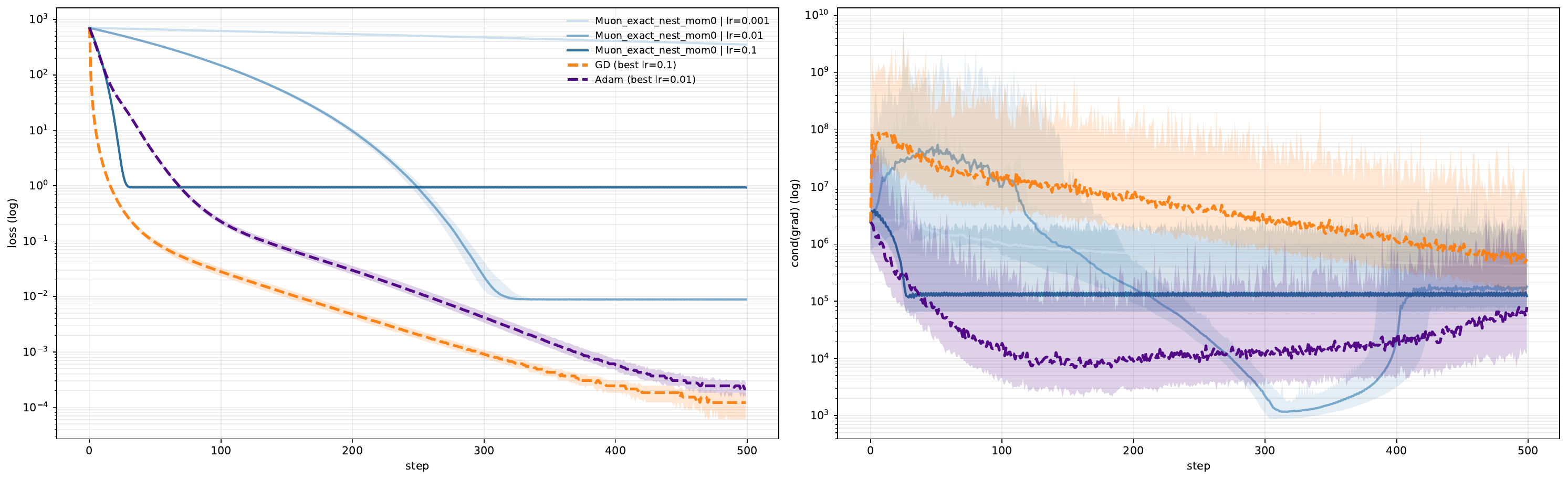}

\vspace{0.8em}

\includegraphics[width=0.95\linewidth]{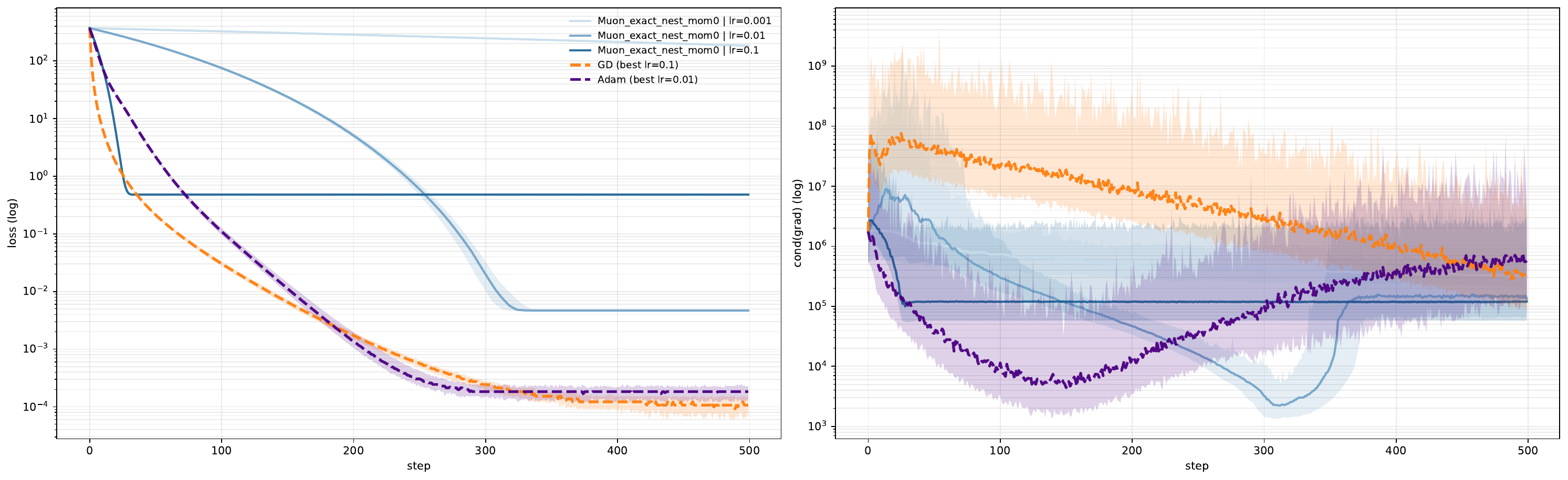}

\vspace{0.8em}

\includegraphics[width=0.95\linewidth]{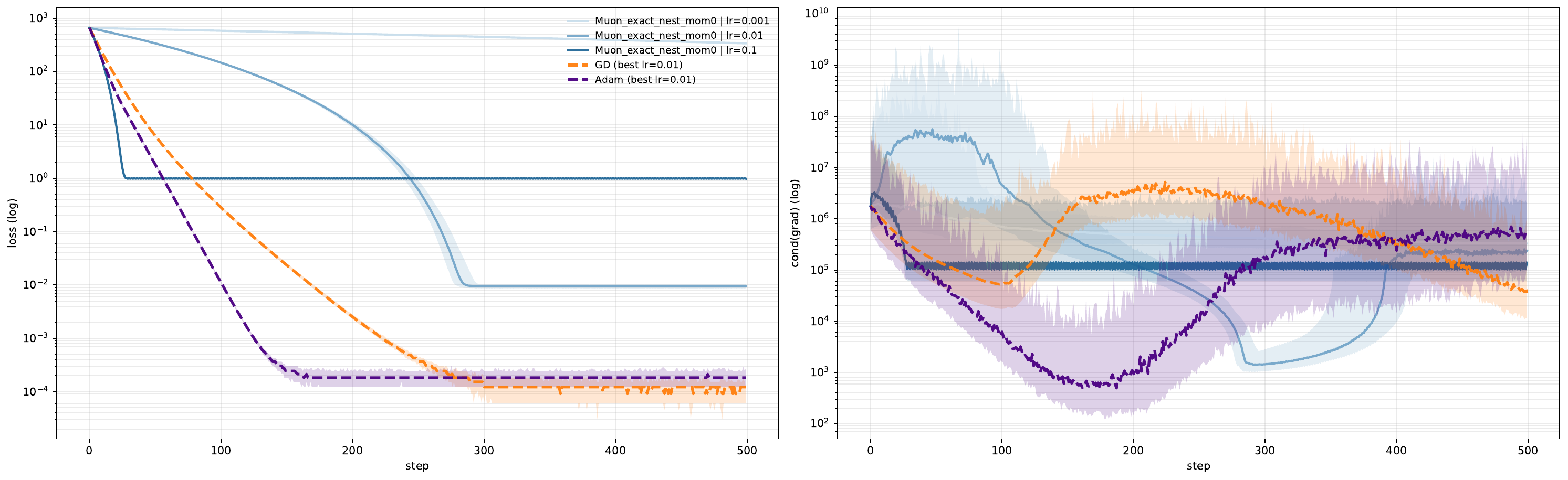}

\caption{
Averaged trajectories for \texttt{geometric\_decay\_to\_max} (top),
\texttt{uniform} (second),
\texttt{linear\_decay\_to\_max} (third),
and \texttt{gaussian} (bottom) spectra.
Each panel shows the loss curves (left) and the gradient conditioning measure along the trajectory (right), averaged over random initializations.
Gradient conditioning alone does not predict which method converges faster.
}
\label{fig:grad-cond-geom-unif-linear-gaussian}
\end{figure}

\section{Noisy Sign Dynamics Results}
\label{app:noisy-sign-proofs}

In this appendix, 
we analyze the one-dimensional noisy sign dynamics
of \Cref{eq:noisy-sign}.
We use a different notation than in the main text:
time is indexed by $n$ instead of $t$,
and the state of the process at time $n$ with noise level $\sigma\ge 0$
is denoted by $w_n^{(\sigma)}$ rather than $\sv_t^{(\sigma)}$.

\begin{proposition}[Noise breaks the deterministic grid trap]
\label{prop:noise-breaks-grid}
Consider the one-dimensional noisy sign dynamics with step size $\stepsize>0$
and noise level $\sigma\ge 0$:
\[
  w_{n+1}^{(\sigma)}
  = w_n^{(\sigma)} - \stepsize\bigl(\sign(w_n^{(\sigma)}) + \sigma \xi_{n+1}\bigr),
  \qquad n \ge 0,
\]
where $(\xi_n)_{n\ge1}$ are i.i.d.\ $\mathcal{N}(0,1)$ random variables and $w_0^{(\sigma)}\in\R$ is fixed. 

Fix $\eps>0$ and define the hitting time
\begin{align}
  T_\eps^{(\sigma)} := \inf\{n \ge 0 : |w_n^{(\sigma)}| \le \eps\}.
  \label{eq:hittingtimefoo}
\end{align}

Assume $\sigma>0$. 
Then, for any initialization $w_0\in\R$,
\[
  \PP\bigl(T_\eps^{(\sigma)} < \infty\bigr) = 1.
\]

By contrast, for the deterministic dynamics ($\sigma = 0$), if the grid
$w_0 + \stepsize\Z$ does not intersect $[-\eps,\eps]$, then
$T_\eps^{(0)} = \infty$ for all such initializations.
\end{proposition}

\begin{proof}[Proof idea for \Cref{prop:noise-breaks-grid}]
  (Full proof in \Cref{sec:noise-breaks-grid}.)
The proof combines a drift-to-compact argument
with a uniform ``attempt'' mechanism:
once the process enters a fixed neighborhood of the origin, 
there is a state-uniform positive probability to land in $[-\eps,\eps]$ 
within a bounded number of steps, 
and repeated returns to that neighborhood 
(which happen almost surely because of the drift)
guarantee eventual success.
\end{proof}

\begin{proposition}[Small- and large-noise have long hitting times]
\label{prop:small-large-noise-long-hitting}
Assume the deterministic iteration
$
  w_{n+1}^{(0)} = w_n^{(0)} - \stepsize \sign(w_n^{(0)})
$
satisfies
$
  |w_n^{(0)}| \ge \eps + \Delta, \forall n\ge 0,
$
for some margin $\Delta>0$; that is, the step grid misses the target interval.
Consider, for $\sigma\ge 0$,
\[
  w_{n+1}^{(\sigma)}
  = w_n^{(\sigma)} - \stepsize\bigl(\sign(w_n^{(\sigma)}) + \sigma\xi_{n+1}\bigr),
  \qquad w_0^{(\sigma)} = w_0,
\]
where $(\xi_n)$ are i.i.d.\ standard Gaussian random variables. 

Let $T_\eps^{(\sigma)}$ denote the hitting time of $[-\eps,\eps]$ as in~\eqref{eq:hittingtimefoo}. 
For any positive integer $N$, we have
\[
  \lim_{\sigma\to 0}
  \PP\left(T_\eps^{(\sigma)} > N\right)
  =
  \lim_{\sigma\to \infty}
  \PP\left(T_\eps^{(\sigma)} > N\right)
  = 1.
\]
\end{proposition}

\begin{proof}[Proof idea for \Cref{prop:small-large-noise-long-hitting}]
  (Full proof in \Cref{sec:small-large-noise-long-hitting}.)
Informally, for very small noise the process shadows 
the deterministic two-cycle for a long time before 
accumulated fluctuations are large enough 
to sufficiently escape the lattice and hit the target interval. 
For very large noise, 
the update becomes dominated by the random term $\sigma\xi_{n+1}$, and the chain behaves like a 
rapidly diffusing random walk whose steps are so large that it frequently overshoots the small target interval $[-\eps,\eps]$.  
\end{proof}

\subsection{Proof of \Cref{prop:noise-breaks-grid}} \label{sec:noise-breaks-grid}

The case $\sigma=0$ is clear since the dynamics live on the grid
$w_0 + \stepsize\Z$, so if that grid misses $[-\eps,\eps]$ then the process never hits that interval. 
Now, fix $\sigma>0$.

\paragraph{Step 1: a drift-to-compact lemma.}
We show that there exists $R>0$ such that, for any starting point $w_0^{(\sigma)}$,
the hitting time
\[
\tau_R:=\inf\{n\ge0:\ |w_n^{(\sigma)}|\le R\}
\]
is almost surely finite.

It suffices to prove that there exist constants $R>0$ and $\delta>0$ such that, for all $|w|\ge R$,
\begin{equation}
\label{eq:drift-compact}
\E\!\left[\,|w_1^{(\sigma)}|-|w_0^{(\sigma)}| \,\middle|\, w_0^{(\sigma)}=w\right]\le -\delta.
\end{equation}
Indeed, let $\cF_n:=\sigma(\xi_1,\dots,\xi_n)$ be the natural filtration, 
and consider the stopped process
\[
Z_n := |w_{\min(n, \tau_R)}^{(\sigma)}|+\delta \min(n,\tau_R).
\]
We claim that  $\E\!\left[ Z_{n+1}\mid \cF_n\right]\le Z_n$ for all $n\ge0$, 
i.e., $(Z_n)_{n\ge0}$ is a supermartingale with respect to $(\cF_n)_{n\ge0}$. 

To see this, consider the two cases:
\begin{itemize}
  \item On the event $\{\tau_R\le n\}$, the process is already stopped, so $Z_{n+1}=Z_n$ 
  and 
  \[
  \E\!\left[ Z_{n+1} 1_{\{\tau_R\le n\}} \mid \cF_n\right] = \E\!\left[ Z_{n} 1_{\{\tau_R\le n\}} \mid \cF_n\right] = Z_n 1_{\{\tau_R\le n\}}.
    \] 
  since $\{\tau_R \le n\}$ is $\cF_n$-measurable. 
  \item On the event $\{\tau_R> n\}$, we have $\min(n+1,\tau_R)=n+1$ 
  so that pointwise
  \[
  Z_{n+1}-Z_n = 1_{\{\tau_R> n\}} \Bigl(\,|w_{n+1}^{(\sigma)}|-|w_n^{(\sigma)}|+\delta\Bigr).
  \]
  Taking conditional expectation given $\cF_n$ and using that $1_{\{\tau_R> n\}}$ is $\cF_n$-measurable,
  \[
\E[Z_{n+1}-Z_n\mid \cF_n]
=
1_{\{\tau_R> n\}} \Bigl(\E[\,|w_{n+1}^{(\sigma)}|-|w_n^{(\sigma)}| \mid \cF_n] + \delta\Bigr).
  \]
  Define the one-step update map
\[
\Phi(x,u) := x - \stepsize\bigl(\sign(x)+\sigma u\bigr),
\]
so that $w_{n+1}^P=\Phi(w_n,\xi_{n+1})$, where we omit the superscript $(\sigma)$ for brevity. 

Let $g:\R\to\R$ be any measurable function with $\E|g(w_{n+1})|<\infty$.
Since $w_n$ is $\cF_n$-measurable, and $\xi_{n+1}$ is independent of $\cF_n$ with the same distribution as $\xi_1$,
we have:
\begin{equation}
\label{eq:kernel-step}
\E\!\left[g(w_{n+1})\mid \cF_n\right]
=
\psi_g(w_n)
\quad\text{a.s.},
\qquad
\psi_g(x):=\E\!\left[g(\Phi(x,\xi_1))\right] = \E\!\left[g(w_1)\mid w_0=x\right].
\end{equation}
Applying \eqref{eq:kernel-step} with $g(y)=|y|$, we get
\[
\E\!\left[|w_{n+1}|-|w_n|\mid \cF_n\right]
=
\E\!\left[|w_{1}| -|w_0|\mid w_0=w_n\right]
\]
The event $\{n<\tau_R\}$ is $\cF_n$-measurable and implies $|w_n|>R$.
Therefore, on $\{n<\tau_R\}$ we may apply \eqref{eq:drift-compact} at the (random) starting point $w=w_n$:
\[
\E\!\left[\,|w_1|-|w_0| \,\middle|\, w_0=w_n\right] \le -\delta.
\]
This yields
\[
\mathbf 1_{\{n<\tau_R\}}\E\!\left[\,|w_{n+1}|-|w_n| \,\middle|\, \cF_n\right]
\le
-\delta\,\mathbf 1_{\{n<\tau_R\}}.
\]
\end{itemize}
Combining the two cases, we get $\E\!\left[ Z_{n+1}\mid \cF_n\right]\le Z_n$. 
Taking expectations yields $\E[Z_{n+1}]\le \E[Z_n]\le \cdots \le \E[Z_0] = |w_0^{(\sigma)}|$.
In particular, $\delta\,\E[\min(n,\tau_R)]\le \E[Z_n]\le |w_0^{(\sigma)}|$ for all $n\ge0$. 
Letting $n\to\infty$ and using monotone convergence gives 
$\E[\tau_R] = \lim_{n\to\infty} \E[ \min(n,\tau_R)]
\le |w_0^{(\sigma)}|/\delta<\infty$,
hence $\PP(\tau_R<\infty)=1$. 

We now verify \eqref{eq:drift-compact}. We keep omitting the superscript $(\sigma)$ for brevity. 
Let $\Delta:=\stepsize(1+\sigma\xi_1)$. By symmetry it suffices to consider starting at $w>0$, in which case
$w_1=w-\Delta$ and
\[
|w_1|=|w-\Delta|=(w-\Delta)+2(\Delta-w)_+,
\qquad (a)_+:=\max(0,a).
\]
Taking expectations over the centered noise $\xi_1$, 
we have  $\E[\Delta]=\stepsize$ and thus
\begin{equation}
\label{eq:abs-decomp}
\E[|w_1| \mid w_0=w]= w-\stepsize + 2\,\E[(\Delta-w)_+].
\end{equation}
We claim that $\E[(\Delta-w)_+]\to 0$ as $w\to\infty$.
For $w\ge1$, define $g_w(\Delta):=(\Delta-w)_+$. Then $g_w(\Delta)\downarrow 0$ pointwise if $w\to\infty$ and
$0\le g_w(\Delta)\le \Delta_+$, where $\E[\Delta_+]<\infty$ since $\Delta$ is Gaussian.
By dominated convergence, $\E[(\Delta-w)_+]\to0$.

Therefore we may choose $R$ large enough so that for all $w\ge R$,
$2\,\E[(\Delta-w)_+]\le \stepsize/2$. Plugging this into \eqref{eq:abs-decomp} yields
\[
\E[|w_1|\mid w_0=w]\le w-\stepsize/2 \qquad (\forall w\ge R),
\]
and by symmetry the same bound holds for $w\le -R$.
Thus \eqref{eq:drift-compact} holds with $\delta:=\stepsize/2$.

\paragraph{Step 2: a uniform bounded-noise contraction window.}  
We now show that starting from any $|w_0^{(\sigma)}|\le R$, 
there is a positive probability of hitting $[-\eps,\eps]$ in a fixed number of steps 
which does not depend on the starting point. 

Pick any $M>0$ such that $c:=\stepsize(1-\sigma M)>0$.  
Fix $n\ge0$. On the event $\{|\xi_{n+1}|\le M\}$ we have 
\begin{equation}
\label{eq:one-step-bounds}
  w_{n+1}^{(\sigma)} \le w_n^{(\sigma)} - c \ \text{ if } w_n^{(\sigma)} > 0,
  \qquad
  w_{n+1}^{(\sigma)} \ge w_n^{(\sigma)} + c \ \text{ if } w_n^{(\sigma)} < 0,
\end{equation}
i.e., the process moves towards zero by at least $c$ in that case. 
In $L$ steps, this ensures a movement of at least $Lc$ towards zero. 
If $L> R/c$, starting from $|w_0^{(\sigma)}|\le R$ and with
$|\xi_1|,\dots,|\xi_L|\le M$, the process must cross 
the level $\pm \eps$ for some
$n\le L$. 
At the crossing time, 
since the one-step move is at most
\begin{equation}
\label{eq:step-bound}
  |w_{n}^{(\sigma)} - w_{n-1}^{(\sigma)}|
  = \stepsize|1+\sigma\xi_n|
  \le \stepsize(1+\sigma M) =: C,
\end{equation}
the pre-crossing iterate lies in $[-\eps-C,\eps+C]$. 
We now show that from there, there is a positive probability of landing in $[-\eps,\eps]$ at the next step. 

Define the compact interval $B:=[-\eps-C,\eps+C]$. For any $x\in B$, the one-step transition is
\[
w_1^{(\sigma)} = x - \stepsize\sign(x) - \stepsize\sigma\xi_1,
\]
so $w_1^{(\sigma)}$ has a Gaussian density with variance $(\stepsize\sigma)^2$.
Hence $\PP_x(w_1^{(\sigma)}\in[-\eps,\eps])>0$ for every $x\in B$.
Moreover, $x\mapsto \PP_x(w_1^{(\sigma)}\in[-\eps,\eps])$ is continuous
(it can be expressed in terms of Gaussian CDFs), so by compactness,
\[
p' := \inf_{x\in B}\PP_x\bigl(w_1^{(\sigma)}\in[-\eps,\eps]\bigr) > 0.
\]

Consequently, defining $L_R := \lceil R/c \rceil$, $p := \PP(|\xi_1|\le M,\dots,|\xi_{L_R-1}|\le M)>0$, and
\[
  q := p^{L_R}p' \in (0,1),
\]
we just showed that on the event
$\{|\xi_1|\le M,\dots,|\xi_{L_R-1}|\le M\}$, starting from any $x$ with $|x|\le R$,
we reach $[-\eps,\eps]$ with probability at least $q$ in $L_R$ steps:
\begin{equation}
\label{eq:one-attempt-success}
\inf_{|x|\le R}  
\PP\Bigl(\exists n\in\{1,\dots,L_R\}:\ w_n^{(\sigma)}\in[-\eps,\eps] \mid w_0^{(\sigma)}=x\Bigr)
  \ge q.
\end{equation}

\paragraph{Step 3: successive attempts and independence.} 
We now combine the two previous steps to show that
$T_\eps^{(\sigma)}<\infty$ almost surely: 
we repeatedly wait for the process to return 
to $[-R,R]$ (which happens a.s.\ by Step 1),
then we make an attempt of length $L_R$ to hit $[-\eps,\eps]$. 
Each attempt succeeds with probability at least $q$ 
independently of the past, so eventually one attempt must succeed a.s. 

Define the successive \emph{attempt start times} $(S_k)_{k\ge 0}$ and end times
$T_k:=S_k+L_R$ recursively by
\[
  S_0 := \tau_R,
  \qquad
  S_{k+1} := \inf\{n\ge T_k:\ |w_n^{(\sigma)}|\le R\},
  \qquad
  T_k := S_k+L_R.
\]

\medskip
\noindent\textbf{Claim: $S_k<\infty$ a.s.\ for all $k$.}
We prove this by induction.
Step~1 gives $S_0=\tau_R<\infty$ a.s.
Assume $S_k<\infty$ a.s.\ for some $k$. Then $T_k=S_k+L_R<\infty$ a.s. as well.
On the event $\{T_k<\infty\}$, apply the strong Markov property at time $T_k$:
conditionally on $\cF_{T_k}$
(the sigma-algebra generated by $(\xi_1,\dots,\xi_{T_k})$ \citep[Def 12.8]{LeGall2022}), 
the post-$T_k$ chain
$(w_{T_k+n}^{(\sigma)})_{n\ge0}$ has the same law as the chain started from
$w_{T_k}^{(\sigma)}$. In particular,
\[
\PP\!\left(S_{k+1}<\infty \,\middle|\, \cF_{T_k}\right)
=
\PP_{w_{T_k}^{(\sigma)}}(\tau_R<\infty)
\qquad\text{a.s. on }\{T_k<\infty\}.
\]
Since Step~1 holds for \emph{any} initial condition, we have
$\PP_x(\tau_R<\infty)=1$ for every $x\in\R$, hence
\[
\PP\!\left(S_{k+1}<\infty \,\middle|\, \cF_{T_k}\right)=1
\qquad\text{a.s.}
\]
Taking expectations yields $\PP(S_{k+1}<\infty)=1$, completing the induction.

\medskip
To conclude the proof, we now show that at least one attempt succeeds a.s. 
The argument essentially relies on the conditional independence of the attempts $A_k$ defined below. 

For each $k$, consider the event that the $k$-th attempt succeeds:
\[
  A_k := \Bigl\{\exists n\in\{S_k+1,\dots,T_k\}:\ w_n^{(\sigma)}\in[-\eps,\eps]\Bigr\}.
\]
It depends only on $(\xi_{S_k+1},\dots,\xi_{T_k})$ and on $w_{S_k}^{(\sigma)}$.
Conditioned on $\cF_{S_k}$, the noises $(\xi_{S_k+1},\dots,\xi_{T_k})$ are
independent of $\cF_{S_k}$ and i.i.d.\ Gaussian, hence by Strong Markov property and \eqref{eq:one-attempt-success},
\begin{equation}
\label{eq:conditional-success}
  \PP(A_k \mid \cF_{S_k}) \ge q \qquad \text{a.s.}
\end{equation}
Note that $\PP(A_k\mid \cF_{S_k})$ is an $\cF_{S_k}$-measurable random variable (a conditional success probability given the past 
up to time $S_k$). 
Therefore,
\[
  \PP\Bigl(\bigcap_{k=0}^N A_k^c\Bigr)
  = \E\Bigl[\prod_{k=0}^N \PP(A_k^c \mid \cF_{S_k})\Bigr]
  \le (1-q)^{N+1}.
\]
Letting $N\to\infty$ yields
\[
  \PP\Bigl(\forall k\ge 0,\ A_k^c\Bigr)=0,
\]
so with probability one, some attempt succeeds, i.e., $T_\eps^{(\sigma)}<\infty$
almost surely. This proves $\PP(T_\eps^{(\sigma)}<\infty)=1$.

\subsection{Proof of \Cref{prop:small-large-noise-long-hitting}} \label{sec:small-large-noise-long-hitting}

We treat the small-noise and large-noise regimes separately.

\paragraph{Small-noise regime \texorpdfstring{$\sigma\to 0$}{sigma→0}.}
Fix $N>0$, consider $S_n := \sum_{k=1}^n \xi_k$, and define the event
\[
  A_\sigma
  := \left\{
       \max_{1\le n\le N}
       \bigl|\sigma S_n\bigr|
       \le \frac{\Delta}{2\stepsize}
     \right\}.
\]
We claim that on $A_\sigma$ the noisy 
and deterministic iterates have the same sign
up to time $N$ if they start from the same initialization:
\begin{equation}
\label{eq:dev-identity}
  w_n^{(\sigma)}-w_n^{(0)} = -\stepsize\sigma S_n
  \qquad \text{for all } n\le N.
\end{equation}
Indeed, argue by induction. For $n=0$ it is trivial.
Assume the signs agree up to time $n-1$. Then \eqref{eq:dev-identity} holds at time $n-1$,
so on $A_\sigma$ we have
\[
  |w_{n-1}^{(\sigma)}-w_{n-1}^{(0)}|
  \le \frac{\Delta}{2}.
\]
Since $|w_{n-1}^{(0)}|\ge \eps+\Delta$ by assumption, this implies
$\sign(w_{n-1}^{(\sigma)})=\sign(w_{n-1}^{(0)})$, completing the induction.

Therefore, on $A_\sigma$, for all $n\le N$,
\[
  |w_n^{(\sigma)}|
  \ge |w_n^{(0)}| - \stepsize\sigma |S_n|
  \ge (\eps+\Delta) - \frac{\Delta}{2}
  = \eps + \frac{\Delta}{2}
  > \eps,
\]
so $T_\eps^{(\sigma)}>N$ on $A_\sigma$, and thus
\[
  \PP(T_\eps^{(\sigma)} \le N) \le \PP(A_\sigma^c).
\]
By Kolmogorov's inequality and $\E[\xi_1^2]=1$,
\[
  \PP(A_\sigma^c)
  = \PP\left(
      \max_{1\le n\le N} |S_n|
      > \frac{\Delta}{2\stepsize\sigma}
    \right)
  \le \frac{N}{(\Delta/(2\stepsize\sigma))^2}
  = O(\sigma^2),
\]
which tends to $0$ as $\sigma\to 0$. Hence
\[
  \PP(T_\eps^{(\sigma)} > N)\longrightarrow 1
  \quad \text{as } \sigma\to 0.
\]

\paragraph{Large-noise regime \texorpdfstring{$\sigma\to\infty$}{sigma→∞}.}
Fix $N>0$. By a union bound,
\[
  \PP(T_\eps^{(\sigma)} \le N)
  \le \sum_{k=1}^N \PP\bigl(|w_k^{(\sigma)}|\le \eps\bigr).
\]
Let $\cF_{k-1}$ be the natural filtration up to time $k-1$. Since
\[
  w_k^{(\sigma)}
  = w_{k-1}^{(\sigma)}
    - \stepsize\sign\bigl(w_{k-1}^{(\sigma)}\bigr)
    - \stepsize\sigma\xi_k,
\]
we have, conditionally on $\cF_{k-1}$,
\[
  w_k^{(\sigma)} \big| \cF_{k-1}
  \sim \mathcal N\left(
    w_{k-1}^{(\sigma)} - \stepsize\sign(w_{k-1}^{(\sigma)}),
    \ \stepsize^2\sigma^2
  \right).
\]
Thus, almost surely, the conditional probability $\PP\bigl(|w_k^{(\sigma)}|\le \eps \mid \cF_{k-1}\bigr)$ 
(a random variable conditioned on the past up to time $k-1$) can be expressed in terms of the Gaussian density as
\[
  \PP\bigl(|w_k^{(\sigma)}|\le \eps \mid \cF_{k-1}\bigr)
  = \int_{-\eps}^{\eps}
    \frac{1}{\stepsize\sigma\sqrt{2\pi}}
    \underbrace{\exp\left(-\frac{(y-m)^2}{2\stepsize^2\sigma^2}\right)}_{\le 1} dy
  \le \frac{2\eps}{\stepsize\sigma\sqrt{2\pi}},
\]
where $m=w_{k-1}^{(\sigma)}-\stepsize\sign(w_{k-1}^{(\sigma)})$ is random. 
Taking expectations yields
\[
  \PP\bigl(|w_k^{(\sigma)}|\le \eps\bigr)
  \le \frac{2\eps}{\stepsize\sigma\sqrt{2\pi}}.
\]
Hence the result:
\[
  \PP(T_\eps^{(\sigma)} \le N)
  \le N\frac{2\eps}{\stepsize\sigma\sqrt{2\pi}}
  \xrightarrow[\sigma\to\infty]{} 0.
\]

\section{Experiments on Noisy Sign Dynamics}

\subsection{1D toy model experiments}
\label{app:noisy-sign-xps}
\Cref{fig:median-T} has been generated as follows. 
We simulate \eqref{eq:noisy-sign} 
with i.i.d.\ standard Gaussian noise $\xi_t\sim\mathcal{N}(0,1)$ and parameters 
\[
  \stepsize = 0.1,\qquad
  \eps = \frac{\stepsize}{5},\qquad
  \sv_0 = 10\stepsize + 1.3\eps.
\]
The choice of $\sv_0$ ensures that the exact sign dynamics never hit $[-\eps,\eps]$.
For each noise level $\sigma$ on a logarithmic grid (from $10^{-3}$ to $10$) 
we run $n_{\mathrm{samples}}=10^4$ independent trajectories, 
cap the simulation at $n_{\max}=1000$ steps, and measure 
$\tilde T\coloneqq\min(T_\eps^{(\sigma)}, n_{\max})$. 
We compute the median of $\tilde T$ over the $n_{\mathrm{samples}}$ trajectories, 
and its $[2.5\%, 97.5\%]$ empirical quantile interval.
These are plotted in \Cref{fig:median-T} (the quantile interval is shaded in light blue). 
\subsection{Newton--Schulz approximate error vs.\ exact projection on controlled-spectrum quadratics}
\label{app:controlled-spectrum-ns}

\Cref{tab:ns-vs-exact-lossratios} complements \Cref{tab:ns-vs-exact-win} by reporting 
the best-loss ratios for \Muon{} with Newton--Schulz projection (no momentum) 
vs.\ \Muon{} with exact projection (no momentum) 
on the same controlled-spectrum quadratic instances as in \Cref{tab:ns-vs-exact-win}.
It shows that Newton--Schulz \Muon{} attains a loss within a factor $2$ 
of exact-projection \Muon{} on most spectrum families and time horizons, 
except on \texttt{min\_spiked}, where Newton-Schulz is significantly better.

\begin{table}[H]
\centering
\small
\setlength{\tabcolsep}{3pt}
\renewcommand{\arraystretch}{1.15}
\caption{ Mean $\pm$ 1.96 standard error (over initializations) of best-loss for \Muon{} with exact projection, divided by best-loss for \Muon{} with Newton--Schulz projection (no momentum, constant step size). 
Values $>1$ favor Newton--Schulz.  
This can be read as ``Newton--Schulz \Muon{} is better by a factor of \texttt{value} on average''.}
\label{tab:ns-vs-exact-lossratios}
\begin{tabular}{lccc}
\toprule
kind & $t=T/10$ & $t=T/2$ & $t=T$ \\
\midrule
\texttt{max\_spiked} & $0.61 \pm 0.016488$ & $0.48 \pm 0.025829$ & $1.15 \pm 0.020648$ \\
\texttt{min\_spiked} & $1.32 \pm 0.275357$ & $8.58 \pm 3.246216$ & $33.18 \pm 8.962830$ \\
\texttt{uniform} & $0.36 \pm 0.003587$ & $0.71 \pm 0.012810$ & $0.71 \pm 0.003194$ \\
\texttt{gaussian} & $0.40 \pm 0.003579$ & $0.55 \pm 0.023125$ & $0.81 \pm 0.005050$ \\
\texttt{linear\_decay\_to\_max} & $0.33 \pm 0.002264$ & $0.82 \pm 0.004592$ & $0.78 \pm 0.003515$ \\
\texttt{u\_shaped} & $0.46 \pm 0.003740$ & $0.42 \pm 0.012617$ & $0.71 \pm 0.004944$ \\
\texttt{geometric\_decay\_to\_max} & $0.58 \pm 0.007069$ & $0.21 \pm 0.006506$ & $0.83 \pm 0.008606$ \\
\bottomrule
\end{tabular}
\end{table}

\section{Exact line search on a quadratic: formulas and additional plots}
\label{app:exact-1d-line-search}

This appendix records the exact line-search expressions used in \Cref{sec:line-search} and provides additional trajectory-level plots for the same quadratic instance.

\paragraph{Quadratic model.}
We consider the quadratic loss
\begin{equation}
\loss(W)
\;=\;
\frac12 \langle W, A W\rangle,
\qquad A \succ 0,
\label{eq:quad-app}
\end{equation}
with Frobenius inner product $\langle X,Y\rangle=\tr(X^\top Y)$.
The unique minimizer is $W_\star=0$, and the gradient is
\begin{equation}
\nabla \loss(W) = A W .
\label{eq:grad-app}
\end{equation}

\paragraph{Exact line search along a direction.}
Fix an iterate $W$ and a direction $D$.
Along the ray $W-\alpha D$,
\begin{equation}
\loss(W-\alpha D)
=
\loss(W)
-
\alpha \langle \nabla \loss(W), D\rangle
+
\frac12 \alpha^2 \langle D, A D\rangle .
\label{eq:quad-line}
\end{equation}
Since $A\succ 0$, we have $\langle D, A D\rangle>0$ for any $D\neq 0$, and the exact line-search minimizer over $\alpha\in\R$ is
\begin{equation}
\alpha^\star(D)
=
\frac{\langle \nabla \loss(W), D\rangle}{\langle D, A D\rangle}.
\label{eq:alpha-star}
\end{equation}
Plugging \eqref{eq:alpha-star} into \eqref{eq:quad-line} yields the exact one-step decrease
\begin{equation}
\Delta(D)
=
\loss(W)-\loss(W-\alpha^\star(D)D)
=
\frac{\langle \nabla \loss(W), D\rangle^2}{2\langle D, A D\rangle}.
\label{eq:delta-app}
\end{equation}

\paragraph{Directions and greedy policy.}
As in \Cref{sec:line-search}, we compare:
\[
D_{\rm GD}(W)=\nabla \loss(W),
\qquad
D_{\rm St}(W)=\proj(\nabla \loss(W)),
\]
where $\proj(G)=UV^\top$ denotes the polar factor of $G=U\Svnn V^\top$.
The greedy policy selects, at each iterate, 
\[
D_{\rm greedy}(W)=
\begin{cases}
D_{\rm GD}(W), & \text{if }\Delta(D_{\rm GD}(W))\ge \Delta(D_{\rm St}(W)),\\
D_{\rm St}(W), & \text{otherwise.}
\end{cases}
\]
and then takes the corresponding (exact line-search) step.
In the experiments reported in \Cref{sec:line-search}, the greedy policy always selects the Stiefel direction along its trajectory,
even though the \GD{} trajectory is far better overall.

\paragraph{Additional trajectory-level plots.}
\Cref{fig:counterexample-quad-grad-dist} reports the gradient norm
$\|\nabla \loss(W_t)\|_{\mathrm{F}}$ and the distance to the minimizer
$\|W_t-W_\star\|_{\mathrm{F}}$
for the same quadratic instance as in \Cref{fig:counterexample-quad-main}.
Consistent with the objective gap results in \Cref{fig:counterexample-quad-main},  
the one-step superiority of the greedy policy does not translate 
into better end-to-end convergence in terms of these metrics either.

\begin{figure}[H]
  \centering
  \includegraphics[width=0.48\textwidth]{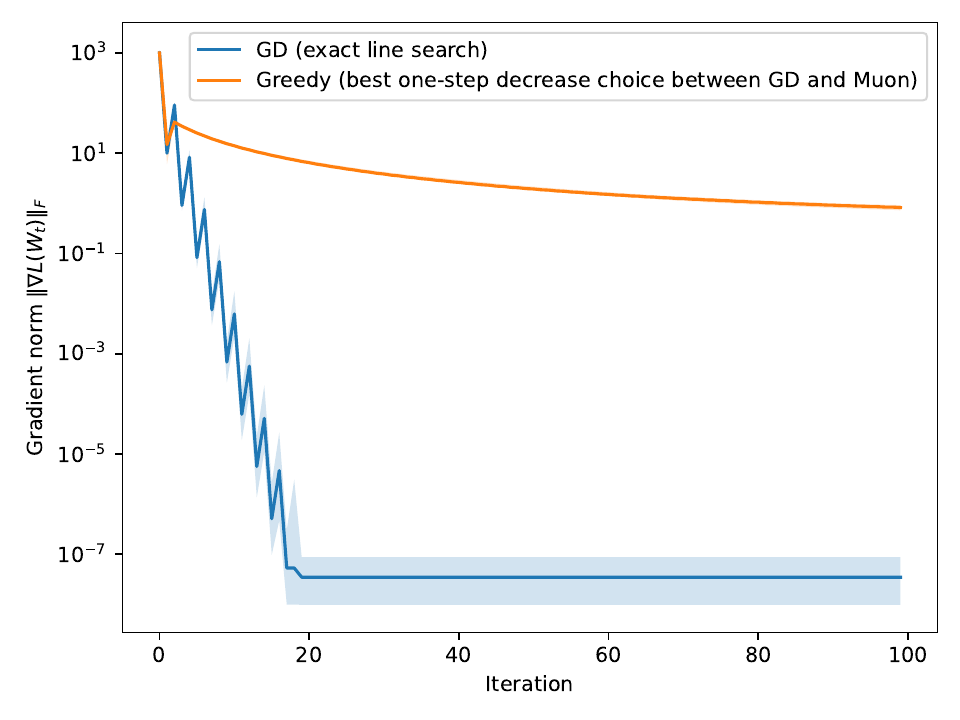}
  \hfill
  \includegraphics[width=0.48\textwidth]{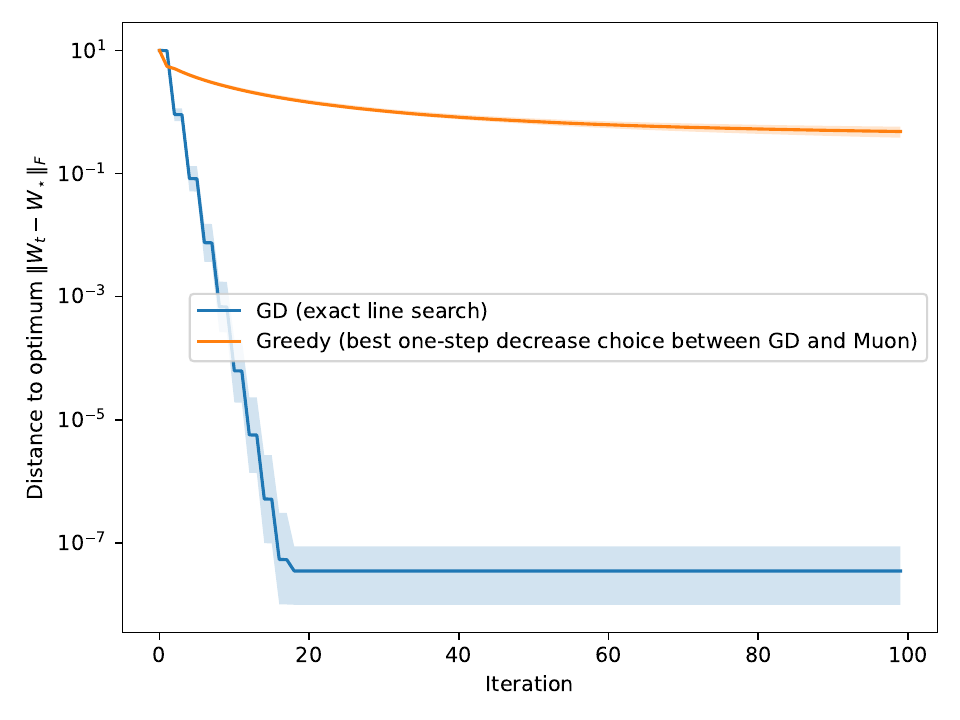}
  \caption{
    Gradient norm $\|\nabla \loss(W_t)\|_{\mathrm{F}}$ (left) and distance to the minimizer $\|W_t-W_\star\|_{\mathrm{F}}$ (right)
    for the same quadratic instance as in \Cref{fig:counterexample-quad-main}.
    Results are averaged over $100$ random initializations. 
    Shaded bands contain the central $95\%$ of runs.
  }
  \label{fig:counterexample-quad-grad-dist}
\end{figure}

\end{document}